\documentclass[a4paper,11pt]{report}

\pdfoutput=1

\usepackage[english]{babel}
\usepackage[utf8x]{inputenc}
\usepackage[T1]{fontenc}
\usepackage{amsthm,amsmath,amssymb,amscd,mathrsfs}
\usepackage{makeidx}
\usepackage[all, cmtip]{xy}
\usepackage{enumitem}
\usepackage{hyperref}
\usepackage{lettrine}
\usepackage{xspace,txfonts,esvect,undertilde}
\usepackage{bussproofs}
\usepackage{geometry}
\usepackage{tikz}
\usepackage{etex}
\usepackage{mathpazo}
\usepackage[scaled=0.95]{helvet}

\usepackage{courier}

\entrymodifiers={+!!<0pt,\fontdimen22\textfont2>}

\newcommand{\overbar}[1]{\mkern 1.5mu\overline{\mkern-1.5mu#1\mkern-1.5mu}\mkern 1.5mu}

\newcommand{\void}{\emptyset}
\newcommand{\nind}{\noindent}

\newcommand{\righttail}{\mathrel{\tikz [x=1.4ex,y=1.4ex,line width=.14ex, baseline] {\draw (0,0.5) -- (1,0.5); \draw [<-] (0.95,0.5) -- (1,0.5);}}}%
\newcommand{\leftrighttail}{\mathrel{\tikz [x=1.4ex,y=1.4ex,line width=.14ex, baseline] {\draw (0,0.5) -- (1,0.5); \draw [<-] (0.05,0.5) -- (0,0.5); \draw [<-] (0.95,0.5) -- (1,0.5);}}}%

\newtheoremstyle{thmstyle}{6mm}{3mm}{\itshape}{}{\bfseries}{}{\newline}{}
\newtheoremstyle{defstyle}{6mm}{3mm}{\normalfont}{}{\bfseries}{}{\newline}{}

\theoremstyle{thmstyle}
\newtheorem{thm}{Theorem}[chapter]
\newtheorem{cor}[thm]{Corollary}
\newtheorem{lem}[thm]{Lemma}

\theoremstyle{defstyle}
\newtheorem{defn}{Definition}[chapter]
\newtheorem{rem}{Remark}[chapter]
\newtheorem{ax}{Axiom}

\makeindex

\title{Constructive Set Theory from a Weak Tarski Universe}
\author{Cesare Gallozzi}

\begin{document}
\begin{titlepage}
\begin{figure}[h]
\begin{center}
\includegraphics[width=.1\textwidth]{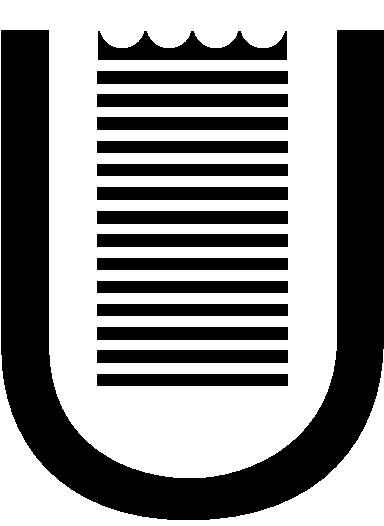}

\end{center}
\end{figure}

    \begin{center}
        {\Large \textsc{Università degli Studi di Roma Tor Vergata}}\\
        \vspace{1 em}
        {\Large \textsc{Facoltà di Scienze Matematiche Fisiche e Naturali}}\\
        \vspace{1 em}
        {\Large \textsc{Corso di laurea magistrale in Matematica Pura ed Applicata}}\\
        \vspace{4em}
        {\Huge \textbf{Constructive Set Theory from a Weak Tarski Universe}}\\
       \vspace{4em}
{\Large \textbf{25th September 2014}}\\
       \vspace{4em}
        {\Large \textsc{Academic Year 2013/2014}}\\
        \vspace{4 em}      
    \end{center}
\vspace{6cm}    
{\Large \textbf{Advisors:} }\hspace{8.8cm}{\Large \textbf{Candidate:}}\\
 {\large Prof. Anna Barbara Veit}\hspace{6.75cm}{\large Cesare Gallozzi}\\ 
 {\large Prof. Ieke Moerdijk}\hspace{7.75cm}{\large }
 \newpage
\end{titlepage}

\begin{abstract}\centering
The aim of this thesis is to give a concise introduction to homotopy type theory, to Aczel's constructive set theory and to simplicial sets and their homotopy theory in particular referring to their standard model structure, showing some of their interactions.\\
The original part of this thesis consists in the final chapter where we introduce in the type theoretic context a definition of weak Tarski universe motivated by categorical models like the one given by simplicial sets. The weakening of this notion, although present in some imprecise form in mathematical folklore was not published before, at the best of our knowledge. Moreover, we show using the axiom of function extensionality that the type theoretic interpretation of constructive set theory generalises to homotopy type theory with a weak Tarski universe.
\end{abstract}

\tableofcontents

\chapter*{Introduction}

\lettrine{F}{oundations} of mathematics are usually settled in one of these three context: set theory, type theory or category theory.\\
On the categorical side there are two quite different approaches: one focuses on the category of all categories whereas the second relies on the notion of topos. Topoi are well-behaved categories modelled on the key examples of the category of sets and the categories of sheaves over a topological space, they can be conceived as categories of continuously variable sets.

Type theory is certainly the less known of these three outside the circle of logicians, for this reason we briefly sketch some of its features.\\
Type theories were introduced the first time by Russell and Whitehead in the celebrated \textit{Principia Mathematica} as a way to overtake the paradoxes. Self-reference is avoided imposing a stratification: atoms can be only elements of sets, sets can be only elements of "sets of sets", and so on; the type of a term, being the level of the stratification in which it takes place. Contemporary type theories have some differences but they share the characterising tract to require every element to have a fixed type. Moreover, types can be conceived as intensional sets.\\
One of the most important features of type theory is the so-called \textit{curry-Howard isomorphism} (although it is not an isomorphism), sometimes referred as the \textit{propositions-as-types} paradigm. The underlying idea is that propositions are identified with the type of their proofs, conversely a term in a type is sometimes called a proof of this type. This symmetry extends to the logical connectives $\land$, $\lor$, $\Rightarrow$ which correspond respectively to the type-theoretic product, the sum and the function type. Moreover, normalisation theorems for proofs correspond to normalisation theorems for terms.\\
The kind of theory of our interest is Martin-L\"of type theory with intensional identity types, it is an intuitionistic type theory that formalizes dependencies i.e. families of types parametrised by a "base type". Some constructors are available: products, sums, well-founded trees, which produce a type from a family of types. Moreover, types for the natural numbers and for finite sets are available and usually also a universe type is added. We will focus on Tarski universes that are universes of names for small types, together with a constructor function which builds the small types from the names. One of the motivations for the development of Martin-L\"of type theory was to extend the propositions-as-types correspondence to intuitionistic predicate logic, so that dependent products and dependent sums correspond respectively to universal and existential quantification. Finally, the most distinctive feature of Martin-L\"of type theory is its rich treatment of equality; indeed, two kinds of equality are available, a \textit{definitional equality}, that is a syntactical equality between terms, and a \textit{propositional equality} that is the type corresponding to an equality in first order logic under the Curry-Howard isomorphism. Moreover, the version of Martin-L\"of type theory that will be discussed in this thesis has intensional identity types, where two terms can be equal in (intensionally) different ways accordingly to different proofs of their equality. One of the best successes of homotopy type theory is to give an homotopical interpretation of identity types.

Now we return to the three ways to approach foundations; we will focus only on topoi on the categorical side (for the reader interested in the formalisation of the category of all categories we refer to \cite{McL}), so that we have a correspondence between topos theory, a kind of intuitionistic set theory and on the type-theoretic side, intuitionistic higher order logic, as in the following diagram:

$$\xymatrix@1{
& & \mbox{Sets} \ar@{<->}[ld] \ar@{<->}[rd]\\
& \mbox{Topoi} \ar@{<->}[rr]
&
& \mbox{Types}
}$$

\nind
For example, starting from set theory we can build the category of sets which is a topos. Moreover, every topos has an internal language which allows to interpret type theory inside it. For more details about this tripos see \cite{Awo1}.\\
The content of this thesis can be seen as a path in a similar tripos that we shall sketch: the constituents are simplicial sets, homotopy type theory with a weak Tarski universe and Aczel's constructive set theory (or constructive Zermelo-Fraenkel, CZF for short).

A remarkable example of topos is the one given by the category of simplicial sets. The idea behind singular homology is to describe the topology of a space using the algebraic structure of the singular simplices inside it. Simplicial sets are an abstract combinatorial generalisation of Euclidean simplicial complexes in which for each dimension a set of n-faces is provided, together with degeneracy and face maps that describe the structure of the complex: which faces are glued and which vertices are collapsed. This sequence of sets can be organised in a functor describing simplicial sets as a presheaf category.\\
The simplicial set given by two vertices linked by a 1-face is the simplicial interval, which allows to define homotopies between simplicial maps, and to develop a theory of homotopy for simplicial sets. A careful analysis of the common features shared by topological and simplicial homotopy theory yields to the definition of model structure on a category that is an abstract playground for homotopy even in absence of an interval. A model structure consists in three distinguished classes of maps modelled on (topological) weak equivalences, Serre fibrations and cofibration. We will see that identity types of Martin-L\"of type theory, can be interpreted categorically using model structures.

The history of the interplay between type theory and homotopy begun with the groupoid interpretation of Martin-L\"of type theory by Hofmann and Streicher \cite{HS} in which identity types are interpreted as the $\mathit{Hom}$-sets of isomorphisms between two objects.\\
The basic ideas about homotopy type theory were developed in independent work by Awodey and Warren \cite{AW} and by Voevodsky \cite{Voe} around 2006. 
Homotopy type theory can be synthetically described as Martin-L\"of type theory with some further axioms and rules added, namely, the univalence axiom, some axioms needed to describe identity types of dependent products and possibly higher inductive types. The expression "homotopy type theory", should be thought as a family of theories, as we usually refer to set theory. However, in order to simplify the exposition we will omit higher inductive types, which are also excluded from the presentation of the core theory given in \cite{KLV}.\\
The most important aspect of homotopy type theory is the \textit{univalence axiom}\index{axiom!univalence}; in presence of a universe it can be stated as follows: under the Curry-Howard isomorphism we can form the types of identity and equivalence of any two small types, identical types are also equivalent so that there is a canonical map from the identity type to the equivalence one. The univalence axiom states that this map is itself an equivalence. It allows to obtain a proof of identity from a proof of equivalence, identifying equivalent objects, as it is done everyday in mathematical practice.\\
More conceptually, the univalence axiom can be restated as a rule expressing the principle of \textit{indiscernibility of equivalents}.\\
This thesis will not deal with higher categories (categories with objects, morphisms, morphisms between morphisms and so on where the associativity and unit conditions hold up to higher isomorphisms), but we need to mention these structures in order to properly contextualise it, providing motivations for the study of homotopy type theory. In fact model categories present a specific kind of $\infty$-categories in which all $k$-morphisms for $k>1$ are invertible, called $(\infty , 1)$-categories. The generalisations of the notion of topos to the higher setting is still an open field of study, but it is conjectured that a good notion of $(\infty , 1)$-topos should have some kind of homotopy type theory as internal language (see \cite{Awo2} and \cite{Joy}).\\
Hopefully, homotopy type theory can give some contribution to the long-standing open problem of the calculation of homotopy groups of sphere, in fact some classical calculations were already performed synthetically in homotopy type theory.

Before talking about CZF we need to briefly introduce predicativism: it is a flavour of constructivism, that does not allow the kind of circular arguments that provide the existence of an element quantifying over a totality to which this element belongs. A typical example is the definition of supremum as the least upper bound of a set of real numbers.\\
CZF is a predicative version of ZF where the underlying logic is intuitionistic. The axioms of ZF incompatible with predicative principles that are separation, powerset and foundation, are weakened. As a compensation, some other axioms are strengthened in order to make them working predicatively.\\
In the series of three articles \cite{Acz1} \cite{Acz2} and \cite{Acz3} Aczel developed the type-theoretic interpretation of constructive set theory building a model of CZF from Martin-L\"of type theory with extensional identity types as a justification of CZF from the already established type-theoretic setting, and used it to justify some further axioms like choice principles and axioms for predicative inductive definitions (namely the \textit{regular extension axiom}\index{axiom!regular extension}).

We can finally see how the work of this thesis fits in the tripos: we take motivations from the model category of simplicial sets in order to justify the introduction of weak Tarski universes and then construct a model of CZF with dependent choices and the regular extension axiom from homotopy type theory with a weak universe.

$$\xymatrix@1{
& & \mbox{CZF}  \ar@{<-}[rd]\\
& \mbox{SSet} \ar[rr]
&
& \mbox{HoTT + weak universe}
}$$

\section*{Organisation}

The first two chapters are devoted to algebraic topology, the former will deal with the basics of simplicial sets, presented without the language of model categories which are introduced in the second chapter, the definitions being motivated by the theory developed in first one.\\
The third chapter provides an introduction to Martin-L\"of type theory and to homotopy type theory. 
\\
The fourth chapter, following the article \cite{KLV}, sketches the proof of the interpretation of type theory in the model category of simplicial sets, in particular of the univalence axiom. We will refer to the article for the many technical proofs involved, especially for the coherence issues needed to interpret the type-theoretic substitution.\\
The fifth chapter presents CZF and develops some basic constructions needed in the sequel.\\
Finally the last chapter contains the proof of the generalised type-theoretic interpretation. In this chapter, proofs are given in detail.\\
A recall of the basics of locally cartesian closed categories is confined in the appendix.\\
In the chapters regarding simplicial sets (namely, the first two and the fourth) we will use freely the excluded middle, the axiom of choice and impredicative definitions. In the fourth chapter when dealing with the interpretation of type theory in simplicial sets we will assume the existence of two inaccessible cardinals. 

We assume the reader familiar with category theory, in particular the language of limits, colimits and adjunctions, with the basics of algebraic topology namely singular homology, higher homotopy groups of spaces, the long exact sequence of a fibration and the like. On the logical side we assume familiarity with intuitionism and constructivism, in particular some knowledge about natural deduction for intuitionistic logic; for set theory just little familiarity with Zermelo-Fraenkel set theory is needed. We have tried to keep prerequisites to a minimum.\\
For the reader in need to fill some gaps we suggest \cite{Mac} for category theory, and \cite{Hat}, \cite{Spa}, \cite{May1} for algebraic topology, these three references has increasing categorical flavour.\\
Finally, we mention \cite{vPla} for the logical prerequisites (for its high readability and conceptual insights \cite{Abr} is warmly suggested to the Italian reader).

\section*{Acknowledgements}
I am very grateful to Prof. Ieke Moerdijk for the time he has spent following my work and for the idea he gave me to study homotopy type theory with a weakening of the notion of universe.\\
I am in debt with Prof. Anna Barbara Veit for her wise suggestions and all the support and guidance during these years starting from my bachelor thesis.\\
I also thank Giovanni Caviglia and Urs Schreiber for helpful discussions, Benno van den Berg for clearing my doubts about categorical interpretations of type theory, Paolo Salvatore for useful corrections and Mike Shulman for suggesting me the right definition of weak Tarski universe.\\
The stay at Radboud Universiteit Nijmegen was partially funded by the international thesis scholarship granted by the University of Rome 2 "Tor Vergata".

\chapter{Simplicial Sets}

\lettrine{A} \; simplicial set is an abstract generalisation of geometrical simplices which are a standard tool in basic algebraic topology. Simplicial sets turn out to be combinatorial models of nice topological spaces, encoding their homotopic structure and allowing explicit calculations.\\
The theory of simplicial sets also provides an abstract playground for homotopy theory and allows to detect some common features with the classical theory of homotopy for topological spaces. A convenient axiomatisation yields to the definition of model categories, which are the usual setting for abstract homotopy theory.\\
In the first two chapters we will use the axiom of choice, its main consequence is the existence theorem for minimal fibrations.\\
The main source for the first two chapters is \cite{Hov} integrated with parts from \cite{GZ} and from these notes on simplicial homotopy theory \cite{JT}.

\section*{Definition of Simplicial Set}

We now start with some basic definitions; the abstract definition of simplicial sets can obscure at first the geometrical intuition of these objects.\\
The intuition behind the formalism is that a simplicial set is a graded set whose elements of degree $n$ are $n$-simplices that can be glued along common boundaries. Simplices of arbitrary dimension are allowed as well as degenerated simplices obtained collapsing some vertices.\\
Recall that the singular complex of a topological space is defined as $S(X)_n := \{ f: \Delta_n \to X \}$, where $\Delta_n$ is the standard Euclidean $n$-simplex in $\mathbb R^n$, with face and degeneracy maps satisfying the simplicial identities. \\
We want to express concisely this idea and have a similar description of a simplicial set in terms of maps from a "standard simplex". Recall that the statement of Yoneda lemma have this structure: $\mathit{Nat}(\mathit{Hom}(r,-),K) \cong Kr$, hence we may choose to define simplicial sets as functors from an appropriate base category, morphisms as natural transformations and a "standard $n$-simplex" as an $\mathit{Hom}$-functor. So we start defining the base category,:

\begin{defn}
\begin{enumerate}[label=(\alph*)]
\item[]

\item The \textbf{simplicial category} \index{simplicial!category}\index{category!simplicial}, written $\Delta$, is the the category of finite ordinals and nondecreasing maps. We will denote its objects as $[n]$.

\item We will write $\partial_n^i : [n-1] \to [n]$ for the injective map which omits the value $i$.
\item We will write $\sigma_n^i : [n+1] \to [n]$ for the  surjective map which takes twice the value $i$.
\end{enumerate}
\end{defn}

\begin{rem}\label{decomposizione}
\begin{enumerate}
\item[]

\item It is easy to check that these maps satisfy the so called \textit{cosimplicial identities} \index{cosimplicial!identities} $\partial^j \partial^i = \partial^i \partial^{j-1}$ if $i<j$\\
$\sigma^j \sigma^i = \sigma^i \sigma^{j+1}$ if $i \leq j$\\
$\sigma^j \partial^i = \begin{cases}
\partial^i \sigma^{j-1}\quad \mbox{ if }i<j\\
1_{[n-1]} \quad \; \mbox{ if }i=j \; \mbox{ or } i=j+1\\
\partial^{i-1} \sigma^{j} \quad \mbox{ if } i> j+1\\
\end{cases}$

\item It is straightforward to check that every nondecreasing map $\mu : [m] \to [n]$ can be written in a unique way as $$\mu = \partial^{i_s} \partial^{i_{s-1}} \dots \partial^{i_1}\sigma^{j_t}\sigma^{j_{t-1}} \dots \sigma^{j_1}$$ with $n \geq i_s \geq \dots \geq i_1 \geq 0$,  $0 \leq j_1 \leq \dots \leq j_t$ and $n=m-t+s$.

\item By the previous remarks $\Delta$ can be identified with the category generated by the objects $[n]$, the arrows $\partial, \sigma$ and the cosimplicial identities.

\item Note that the epimorphisms are exactly the surjections, and that the monomorphisms are exactly the injections. Hence, every epi is a split epi and every monic is a split monic.
\end{enumerate}
\end{rem}

\begin{defn}
\begin{enumerate}[label=(\alph*)]
\item[]

\item A \textbf{simplicial set} \index{simplicial!set} is an object in the functor category $\mathit{Funct}(\Delta^{op}, \textbf{Set})$.\\
Similarly, a \textbf{simplicial object} \index{simplicial!object} in a category $\mathcal C$ is an object in the functor category $\mathit{Funct}(\Delta^{op}, \mathcal C)$.\\
A \textbf{cosimplicial object} \index{cosimplicial!object} is an object in the functor category $\mathit{Funct}(\Delta, \mathcal C)$.\\
A \textbf{simplicial map} \index{simplicial!map} is a natural transformation in the appropriate functor category. The category of simplicial sets is written as $\textbf{SSet}$\index{$\textbf{SSet}$}.

\item An $n$-\textbf{simplex}\index{$n$-simplex} is an element $x \in X_n = X[n]$.

\item The \textbf{face maps} \index{face maps} of a simplicial set $X$ are defined as $d_i := X(\partial^i)$.

\item The \textbf{degeneracy maps} \index{degeneracy maps} of a simplicial set $X$ are defined as $s_j := X(\sigma^j)$.

\item A \textbf{subsimplicial set}\index{simplicial!sub- set} $Y$ of a given simplicial set $X$ is a subfunctor of $X$, i.e. a functor such that for each $[n] \in \Delta$ we have $Y[n] \subseteq X[n]$, and for each map $\mu : [m] \to [n]$ that $Y(\mu)$ is a restriction of $X(\mu)$.
\end{enumerate}
\end{defn}

\nind
The face and degeneracy maps of a simplicial set satisfy the \textit{simplicial identities} \index{simplicial!identities} which are the dual of the previous stated identities.

A simplicial set can be equivalently described as a sequence of sets $\{ X_n \}_n$ with face and degeneracy maps satisfying the simplicial identities; these maps are the abstract data encoding the structure of the simplicial set, telling us which simplices are faces of the others, which simplices are glued together and so on.\\
Note that the category of simplicial sets is complete and cocomplete, with limits and colimits calculated pointwise, as all presheaf categories.\\

\vspace{10pt}

\begin{defn}
\begin{enumerate}[label=(\alph*)]
\item[]

\item Any element $v \in X_0$ is called a \textbf{vertex}\index{vertex of a simplicial set}.

\item Any image of a simplex $x \in X_n$ under a face map is called a \textbf{face}\index{simplex!face of}.

\item Similarly, any image of a simplex under a degeneracy map is called a \textbf{degeneracy}\index{simplex!degeneracy of}.

\item A \textbf{nondegenerate simplex}\index{simplex!nondegenerate} $x \in X_n$ is a simplex which is a degeneracy only of itself.\\
A \textbf{degenerate simplex}\index{simplex!degenerate} is a simplex which is not nondegenerate.

\item A simplicial set is \textbf{finite}\index{simplicial!finite simplicial set} iff it has only a finite number of nondegenerate simplices.

\end{enumerate}
\end{defn}

Now we provide some examples of common simplicial sets and give a description of simplices in term of maps from a standard simplex:

\begin{defn}
The \textbf{standard $n$-simplex} \index{standard $n$-simplex} is defined to be the complex $[p] \mapsto \Delta([p],[n])$ and it is written $\Delta[n]$\index{$\Delta[n]$}.
\end{defn}

\nind
Geometrically, the interior of the standard $n$-simplex is represented by the identity map $1 : [n] \to [n]$, the faces are represented by the face maps $\partial [n-1] \to [n]$, and nondegenerate simplices are represented by injective monotone maps. Therefore, the intuition of the standard simplex is just an the abstract combinatorial structure of the usual Euclidean standard simplex.\\

\begin{defn}
\begin{enumerate}[label=(\alph*)]
\item[]
\item The \textbf{boundary}\index{boundary of the standard $n$-simplex} of the standard $n$-simplex, written $\partial \Delta[n]$\index{$\partial \Delta[n]$}, has nondegenerate $r$-simplices the non-identity injective monotone maps $i: [r] \to [n]$.

\item Given a $k$ with $0 \leq k \leq n$ the $k$-\textbf{horn}\index{horn} $\Lambda^k[n]$\index{$\Lambda^k[n]$} has non-degenerate $r$-simplices all injective order-preserving maps $[r] \to [n]$ except the identity and the injective order-preserving maps $\partial^k: [n-1] \to [n]$ whose image does not contain $k$.

\item The \textbf{simplicial circle}\index{simplicial!circle} is the coequalizer of the pair of morphisms $\Delta(\partial^0), \Delta(\partial^1) : \Delta[0] \to \Delta[1]$.

\end{enumerate}
\end{defn}

\nind
The non-degenerate simplices of the boundary $\partial \Delta[n]$ are exactly the ones of the standard $n$-simplex except the interior represented by the identity.\\
Whereas the horn can be thought as obtained by the standard simplex omitting the interior and the face opposed to the $k$-th vertex.\\

\begin{defn}
A \textbf{singular simplex}\index{simplex!singular} of a simplicial set $X$ is a simplicial map $\Delta[n] \to X$.\\
The \textbf{category of simplices}\index{category!of simplices} of a given simplicial set $K$ is just the category of singular simplices and natural transformations between them. We write it as $\Delta K$.
\end{defn}

\nind
By the Yoneda lemma we have the isomorphism $\textbf{SSet}(\Delta[n], K) \cong K_n$; so simplices $x \in X$ correspond to singular simplices, accordingly to our initial example of the singular complex of a topological space.

\begin{lem}
Given a simplicial set $K$, it is the colimit of the functor $\Delta K \to \textbf{SSet}$ which takes the singular simplex $f: \Delta[n] \to K$ to $\Delta[n]$ itself.
\end{lem}
\begin{proof}
It is an immediate corollary of the well-known fact that every presheaf is canonically a colimit of representable functors.
\end{proof}

\section*{Geometric Realisation}

\begin{thm}
Let $\mathcal C$ be a category with small colimits, then the copresheaf category $\mathcal C^{\Delta}$ is equivalent to the category of adjunctions $\textbf{SSet} \to \mathcal C$. We denote the image of a simplicial object $A$ under this equivalence by $(A \otimes -, \mathcal C(A, -), \varphi
)$.
\end{thm}
\begin{proof}
Start with an adjunction $(F,U,\varphi): \textbf{SSet} \to \mathcal C$, and let consider the functor $D: \Delta \to \textbf{SSet}$ which takes $[n]$ to $\Delta[n]$, then we can consider the composite $\displaystyle \Delta \mathop{\longrightarrow}^{D} \textbf{SSet} \mathop{\longrightarrow}^{F} \mathcal C$. This defines a functor from adjunctions to copresheaves.\\
Conversely, given a cosimplicial object $K$, there is a functor $\Delta K \to \Delta$ which takes a singular simplex $f: \Delta[n] \to K$ to $[n]$. We then have the corresponding restriction functor $\mathcal C^{\Delta} \to \mathcal C^{\Delta K}$ and the usual colimit functor $\mathcal C^{\Delta K} \to \mathcal C$. Then we define $A \otimes -$ to be the image of $A$ under this composite functor. Since a map of cosimplicial sets induces a functor $\Delta K \to \Delta L$ this assignment actually defines a functor.\\
The cosimplicial set $\mathcal C(A,Y)$ is defined to have $n$-simplices $\mathcal C (A[n],Y)$, and the adjointness isomorphism is the composite $$\mathcal C(A \otimes K,Y) \cong \mathcal C( \mbox{colim}_{\Delta K}A[n],Y)  \cong \mbox{lim} \mathcal C (A[n],Y) \cong \mbox{lim} \textbf{SSet}(\Delta[n], \mathcal C(A,Y)) \cong$$ 
 $$ \cong \textbf{SSet}(\mbox{colim}_{\Delta K} \Delta[n], \mathcal C(A,Y) ) \cong \textbf{SSet}(K, \mathcal C (A,Y))$$ Since the identity map of $\Delta[n]$ is cofinal in the category $\Delta \Delta[n]$, we get an isomorphism $A \otimes \Delta[n] \cong A[n]$. Conversely, if $F$ preserves limits, then there is a natural isomorphism $F(\Delta[-]) \otimes K \cong FK$.
\end{proof}

\begin{rem}
\begin{enumerate}
\item[]

\item By the previous theorem each functor $\mathcal C \to \mathcal C^{\Delta}$ gives rise to a functor $\mathcal C \to \mathit{Adj}(\textbf{SSet}, \mathcal C)$, which in turn gives a bifunctor $- \otimes - : \mathcal C \times \textbf{SSet} \to \mathcal C$.

\item We have an obvious functor $\textbf{SSet} \to \textbf{SSet}^{\Delta}$ that takes a simplicial set $K$ to the cosimplicial simplicial set $K \times \Delta[-]$. Under the correspondence of the previous point we get the associated bifunctor which is simply the product on $\textbf{SSet}$, since the product commutes with colimits. The functor $K \times -$ has a right adjoint like all other presheaf categories with its usual description: $n$-simplices of $ \underline{\mathit{Hom}}(K,L)$ are simplicial maps $K \times \Delta[n] \to L$.
\end{enumerate}
\end{rem}

Now we recall the definition and the basic properties of a kind of convenient topological spaces that will be useful in the study of the topological spaces associated to a simplicial sets.

\begin{defn}
A \textbf{Kelley space}\index{Kelley space} $X$ is an Hausdorff topological space such that a subset $F \subseteq X$ is closed whenever its intersection with each compact subset of $X$ is closed.
We will write $\textbf{Ke}$\index{$\textbf{Ke}$} for the full subcategory of Kelley spaces.
\end{defn}

\begin{defn}
Given an Hausdorff topological space $Y$ we can the associated topological space $Y_{Ke}$ with the same underlying set, whose closed sets are the ones whose intersection with all compact subset of $Y$ is closed in $Y$. $Y_{Ke}$ is called the \textbf{Kelleyfication}\index{Kelleyfication} of the Hausdorff space $Y$.
\end{defn}

\begin{thm}
The category of Kelley spaces is a full coreflective subcategory of the category of Hausdorff spaces. Moreover, it is complete, cocomplete and cartesian closed with internal $\underline{\mathit{Hom}}$ given by the Kelleyfication of the $\mathit{Hom}$-space with the compact open topology.
\end{thm}
\begin{proof}
See chapter VII.8 in \cite{Mac}.
\end{proof}

\nind
Note that we have a cosimplicial topological space $|\Delta[-]|$. By the previous theorem it gives rise to an adjunction $(| \cdot|, S, \varphi): \textbf{SSet} \to \textbf{Top}$. $| \cdot|$ is a left adjoint, therefore it preserves colimits.\\
Since $|\Delta[n]|$ is a compact Hausdorff space and the category of Kelley spaces $\textbf{Ke}$ is closed under colimits the functor $| \cdot |$ takes values in $\textbf{Ke}$.

\begin{defn}
We call \textbf{geometric realisation}\index{geometric realisation functor} the functor $| \cdot|: \textbf{SSet} \to \textbf{Top}$ defined in the previous remark.\\
Its right adjoint $S: \textbf{Top} \to \textbf{SSet}$ is called the \textbf{singular functor}\index{singular functor}.\\
We will give the same names for the functors $| \cdot|: \textbf{SSet} \to \textbf{Ke}$ and $S: \textbf{Ke} \to \textbf{SSet}$ where we restrict to Kelley spaces.
\end{defn}

\nind
In order to help the reader to develop some intuition about simplicial sets we state the following:

\begin{thm}
The geometric realisation of a simplicial set is a CW-complex.
\end{thm}
\begin{proof}
See section 1.3 in \cite{JT}.
\end{proof}

\begin{thm}
The geometric realisation $| \cdot|: \textbf{SSet} \to \textbf{Ke}$ preserves finite products.
\end{thm}
\begin{proof}
Since the product functor is a left adjoint it preserves colimits, therefore it suffices to show that the canonical map $| \Delta[m] \times \Delta[n] | \to |\Delta[m]| \times | \Delta[n] |$ is a homeomorphism. We will prove that the domain and the codomain are compact Hausdorff spaces, hence it will suffice to show that this map is a bijection.\\
We start with some combinatorial preliminaries studying the nondegenerate simplices of $\Delta[m] \times \Delta[n]$; a $p$-simplex of this simplicial set is the same thing as a monotone map $[p] \to [m] \times [n]$, with the order in the codomain given by $(a,b) \leq (a',b')$ iff $a \leq a'$ and $b \leq b'$. A nondegenerate $p$-simplex is an injective map $[p] \to [m] \times [n]$, then it can be thought as a chain in this ordered set.\\
Any chain can be extended to a maximal chain, and therefore any simplex in $\Delta[m] \times \Delta[n]$ is a face of a nondegenerate $(m+n)$-simplex. Such a maximal chain is a path in the ordered set $[m] \times [n]$ from $(0,0)$ to $(m,n)$ which goes only right or up. We can conveniently label the vertices of the rectangle $[m] \times [n]$ moving from left to right and then bottom up row by row. In this way we can label the maximal chains as $m$-subsets of $\{ 1, \dots, m+n  \}$; of which there are ${m+n \choose m}$.\\
Now, let $c(i)$ for $0 \leq i \leq {m+n \choose m}$ be the list of maximal chains of the rectangle $[m] \times [n]$. Given any chain $c$, let $n_c$ be the number of edges in $c$. From what we have said before is easy to check that the following is a coequalizer diagram: 

$$ f,g : \coprod_{1 \leq i < j \leq {m+n \choose m}} \Delta[n_{c(i) \cap c(j)} ] \rightrightarrows  \coprod_{1 \leq i \leq {m+n \choose m}} \Delta[n_{c(i)}] \to \Delta[m] \times \Delta[n]$$

\nind
where $f$ and $g$ are induced by the inclusions $c(i) \cap c(j) \to c(i)$ and $c(i) \cap c(j) \to c(j)$, respectively.

Standard simplices are compact Hausdorff, and since the geometric realisation commutes with coequalizers we get that $|\Delta[m] \times \Delta[n]|$ is a compact Hausdorff space as well.\\
we now describe the maps $h_i : |\Delta[m+n]| \to |\Delta[m]| \times |\Delta[n]|$, defined by the composite $|\Delta[n_{c(i)}] | \to |\Delta[m] \times \Delta[n]| \to |\Delta[m]| \times |\Delta[n]|$. We denote a point in $|\Delta[m+n]|$ as $z=(z_1 , \dots , z_{n+m})$ where $0 \leq z_i$ and $\sum_i z_i \leq 1$. Suppose $c(i)$ corresponds to the $m$-subset $\{ a_1 < \dots <a_m \}$ of $\{ 1 , \dots , m+n \}$ whose complement is $\{ b_1 < \dots < b_n \}$. We write $a_{m+1} = n+m+1 = b_{n+1}$, then $h_i(z) = (u,v) = (u_1 , \dots , u_m , v_1 , \dots , v_n) \in |\Delta[m]| \times |\Delta[n]|$ where $u_j = \sum_{k=a_j}^{a_{j+1} -1} z_k$ and $v_j = \sum_{k=b_j}^{b_{j+1}-1} z_k$. Now it is easy to check that $h_i$ is injective.\\
Given a point $(u,v) \in |\Delta[m]| \times |\Delta[n]|$ we must find a chain $c(i)$ and a point in $|\Delta[n_{c(i)}]|$ whose image under $h_i$ is $(u,v)$. We must also show that different choices of $c(i)$ are related by a coequalizer diagram describing $\Delta[m] \times \Delta[n]$.\\
To find $c(i)$ we let $w_j = u_j + \dots + u_m$ and $x_j = v_j + \dots + v_n$. We then write the set of $x_j$ and $w_j$ in descending order $y_1 \geq \dots \geq y_{m+n}$. Each $w_j$ must be some $y_{k_j}$. The set of the $k_j$ is an $m$-subset of $m+n$ so corresponds to a maximal chain $c(i)$. Now let $z_j=y_j - y_{j+1}$, where $y_{m+n+1} = 0$. Then $h_i(z_1 , \dots , z_{m+n})=(u,v)$ as required. It is not difficult to check that the ambiguity in the choice of $c(i)$ corresponds exactly to points in $\coprod |\Delta[n_{c(i) \cap c(j)}]|$.
\end{proof}

\begin{thm}
The geometric realisation $| \cdot|: \textbf{SSet} \to \textbf{Ke}$ preserves finite limits.
\end{thm}
\begin{proof}
It is sufficient to show that it preserves equalizers, for the details see lemma 3.2.4 in \cite{Hov}.
\end{proof}

\begin{defn}
A map $f : X \to Y$ of simplicial sets is a \textbf{weak equivalence}\index{simplicial!weak equivalence}\index{weak equivalence!of simplicial sets} iff $|f|$ is a weak equivalence of topological spaces i.e. it induces isomorphisms of all homotopy groups.
\end{defn}

\section*{Anodyne Extensions and Kan Fibrations}

Now we introduce the basic ingredients for the homotopy theory of simplicial sets that are anodyne extensions and Kan fibrations. Usually, the latter are easy to introduce using lifting properties, but also the former are often defined using lifting properties with respect to Kan fibrations. This definition has the advantage of being quick, but it is not transparent at first sight, nor explicit. For that reason we prefer a slightly more explicit one.

\begin{defn}
Given a morphism $f: Y \to Y'$ a morphism $g: X \to X'$ is a \textbf{retract}\index{retract of a morphism} of $f$ iff it fits in a diagram

$$\xymatrix@1{
& X  \ar[r]^{u} \ar[d]^{g} 
& Y  \ar[d]^{f} \ar[r]^{v} 
& X  \ar[d]^{g} \\
& X' \ar[r]^{u'}
& Y' \ar[r]^{v'}
& X'
}$$

\nind
such that $v \circ u = 1_X$ and $v' \circ u' = 1_{X'}$.
\end{defn}

\begin{defn}
A set $A$ of morphisms of simplicial sets is \textbf{saturated}\index{saturated set of morphisms} iff 
\begin{enumerate}[label=(\roman*)]

\item it contains all isomorphisms;

\item it is stable under pushouts;

\item it is stable under retracts;

\item it is stable under countable compositions, i.e. if $f_i : X_i \to X_{i+1}$ are morphisms in $A$, then the canonical map $X_1 \to \mbox{colim} X_i$ is a morphism in $A$;

\item it is stable under arbitrary direct sums.

\end{enumerate}
The intersection of all saturated sets contained in a given set of morphisms $B$ is called the \textbf{saturated set generated by $B$}\index{saturated set generated by a given set}.
\end{defn}

\begin{defn}
The elements of the saturated set generated by all horn inclusions $\Lambda^k[n] \hookrightarrow \Delta[n]$ with $1 \leq n$ and $k \leq n$ are called \textbf{anodyne extensions}\index{anodyne extension}.
\end{defn}

\begin{defn}
Given a map $p:E \to B$ and a map $i: K \to L$ we say that $i$ has the \textbf{left lifting property}\index{left lifting property} (LLP\index{LLP} for short) with respect to $p$ or that $p$ has the \textbf{right lifting property}\index{right lifting property} (RLP\index{RLP} for short) with respect to $i$ iff for every commutative square of the form:

$$\xymatrix@1{
& K  \ar[r]^{u} \ar[d]_{i} 
& E  \ar[d]^{p} \\
& L \ar[r]_{v} \ar@{-->}[ru]^{w}
& B 
}$$

\nind
there is an arrow $w:L \to E$ such that $u = w \circ i$ and $v = p \circ w$.
\end{defn}

\begin{defn}
A map $p: E \to B$ of simplicial sets is a \textbf{Kan fibration}\index{fibration!simplicial}\index{fibration!Kan} or simply a fibration iff it has the right lifting property with respect to all anodyne extensions.
A simplicial set $X$ is called a \textbf{Kan complex}\index{Kan complex} or a \textbf{fibrant object}\index{simplicial!fibrant object} iff the unique map $X \to 1$ is a Kan fibration.
\end{defn}

\begin{rem}
Kan fibrations satisfy the closure conditions dual to the ones of a saturated class.\\
Namely, every isomorphism is a fibration, they are stable under pullbacks, retracts and arbitrary products. Moreover, if $p_i : X_{i+1} \to X_i$ are fibrations, the canonical projection $\mbox{lim}X_i \to X_0$ is a fibration.
\end{rem}

\nind
The following theorem provides examples of Kan fibrations.

\begin{thm}[Moore]\label{Moore}
Every simplicial group is a Kan complex
\end{thm}
\begin{proof}
See theorem 3.1.3 in \cite{JT}.
\end{proof}

\nind
We state now a useful result.

\begin{thm}\label{fibrefibration}
The fibres of a fibration over a connected base have the same homotopy type.
\end{thm}
\begin{proof}
See corollary 5.4.2 in \cite{GZ}.
\end{proof}

\begin{defn}
The elements of the saturated set generated by all boundary inclusions $\partial \Delta[n] \hookrightarrow \Delta[n]$ are called \textbf{cofibration}\index{simplicial!cofibration}.\\
A simplicial set $X$ is  \textbf{cofibrant}\index{simplicial!cofibrant object} iff the unique map $0 \to X$ is a cofibration.
\end{defn}

\begin{thm}
A map is a cofibration in $\textbf{SSet}$ iff it is a monomorphism. In particular every simplicial set is cofibrant.
\end{thm}
\begin{proof}
Recall that the monomorphisms are exactly the injective maps. Since boundary inclusions are injective and injective maps are closed under pushouts, countable compositions, coproducts and retracts we have that every cofibration is a monomorphism.\\
Conversely, given $f:K \to L$ a monomorphism we prove that it is countable composition of pushouts of coproducts of boundary inclusions. By induction, define $X_0 := K$ and having defined $X_n$ and an injection $X_n \to L$ which is an isomorphism on simplices of dimension less than $n$. Then let $S_n$ denote the set of simplices not in the image of $X_n$. Each such simplex $s$ is necessarily nondegenerate and corresponds to a map $\Delta[n] \to L $. Notice that the restriction of $s$ to the boundary factors uniquely through $X_n$. Define $X_{n+1}$ to be the pushout in the diagram:

$$\xymatrix@1{
& \coprod_S \partial \Delta[n]  \ar[r] \ar[d]
& X_n  \ar[d]^{p} \\
& \coprod_S \Delta[n] \ar[r]
& X_{n+1}
}$$

\nind
Then the inclusion $X_n \to L$ extends to a map $X_{n+1} \to L$, which is the desired map.
\end{proof}

\nind
The reader may ask why we have introduced a notion and proved immediately after that it is equivalent to a previous one. The reason is that cofibrations will give an essential part of the definition of model category, which we shall see in the next chapter.

Next we give another useful description of anodyne extensions. Put $B_1$ to be the class of horn inclusions used in the definition of anodyne extension. Let $B_2$ be the class of all inclusions of the form:
$$\Delta[1] \times \partial \Delta[n] + \{ e \} \times \Delta[n] \hookrightarrow \Delta[1] \times \Delta[n]$$
\nind
where $e = 0,1$ and $+$ is another symbol for the coproduct. Finally, let $B_3$ be the class of the more general inclusions:
$$\Delta[1] \times Y + \{ e \} \times X \hookrightarrow \Delta[1] \times X$$
\nind
where $X$ runs through all simplicial sets and $Y$ through the subcomplexes of $X$.

\begin{thm}
The saturated sets generated by $B_1$, $B_2$ and $B_3$ coincide.
\end{thm}
\begin{proof}
It proceeds reducing each class to the previous one. See section 2, chapter IV in \cite{GZ}.
\end{proof}

\begin{defn}
Given two maps $f:Y \to X$ and $i: K \to L$, the induced map:
$$i \Box f: P(i,f)=(K \times X )\coprod_{K \times Y} (L \times Y) \to L \times X$$\index{$\Box$}
\nind
is called the \textbf{pushout product}\index{pushout product} of the two maps.\\
The object $P(i,f)$ is also called pushout product.
\end{defn}

\begin{thm}
Given two monomorphisms $i:K \to L$ and $f:Y \to X$ such that the first is an anodyne extension. Then the map $i \Box f: P(i,f) \to L \times X$ is an anodyne extension.
\end{thm}
\begin{proof}
Let $A$ be the set of monomorphisms $i': K' \to L'$ such that the induced morphism on the pushout product $P(i' , f) \to L' \times X$ is an anodyne extension. It is easy to prove that $A$ is a saturated set, so that it is sufficient to check that $A$ contains $B_3$. Let then $Y' \to X'$ be a monomorphism, and let $K=\Delta[1] \times Y'  \coprod \, \{ e \} \times X'$ and $L= \Delta[1] \times X$. We then have $P(i,f) = \Delta[1] \times (Y' \times X \coprod X' \times Y ) \coprod \; \{ e \}\times X' \times X$ and $L \times X = \Delta[1]\times X' \times X$. So that the inclusion of $K \times X \coprod L \times Y$ into $L \times X$ belongs to $B_3$ and that the inclusion of $K$ into $L$ belongs to $A$.
 \end{proof}

\begin{thm}\label{homfib}
If $i:K \to L$ is a monomorphism and $p: X \to Y$ a Kan fibration. Then the induced map $\underline{\mathit{Hom}}(i,p): \underline{\mathit{Hom}}(L,X) \to \underline{\mathit{Hom}}(K,X) \times_{\underline{\mathit{Hom}}(K,Y)} \underline{\mathit{Hom}}(L,Y)$ is a fibration.
\end{thm}
\begin{proof}
Straightforward from the previous theorem, recalling that Kan fibrations are defined using a lifting property in terms of anodyne extensions.
\end{proof}

\begin{thm}\label{factor}
Any simplicial map can be factored as an anodyne extension followed by a fibration. Moreover, the factorisation is functorial.
\end{thm}
\begin{proof}
Consider the set $L$ of commutative diagrams of the form:
$$\xymatrix@1{
& \Lambda^k[n]   \ar[r] \ar[d]
& X  \ar[d]^{f} \\
& \Delta[n] \ar[r]
& Y
}$$
and we sum over the set $L$ and form the pushout obtaining an anodyne extension $i$:

$$\xymatrix@1{
& \coprod_L \Lambda^k[n]  \ar[r] \ar[d]^{i}
& X  \ar[d]^{i^0} \\
& \coprod_L \Delta[n] \ar[r]
& X^1
}$$
So that we have the commutative diagram:
$$\xymatrix@1{
& X  \ar[rr]^{i^0} \ar[dr]^{f}
&
& X^1  \ar[dl]_{f^1} \\
& 
& Y
&
}$$
We apply the same process, now to $f^1$ obtaining $f^2$ and so on. Finally, we put $E := \mbox{colim} X^n$ and define $p:E \to Y$ as the map induced by the $f^n$. Hence we have a factorisation $f=pi$ where $i$ is anodyne, we have to check that $p$ is a fibration, so consider a commutative diagram:
$$\xymatrix@1{
& \Lambda^k[n]   \ar[r]^{h} \ar[d]
& E  \ar[d]^{p} \\
& \Delta[n] \ar[r]
& Y
}$$
and try to lift it. Since $\Lambda^k[n]$ has only finitely many non-degenerate simplices, $h$ factors through some $X^n$, but then we have a lifting in $X^{n+1}$, and hence a diagonal filler in the starting square.\\
Note that functoriality follows by construction.
\end{proof}

\begin{thm}
A map $i$ is anodyne iff it has the left lifting property with respect to fibrations
\end{thm}
\begin{proof}
Factor $i$ as $i=pj$ where $j$ in an anodyne extension and $p$ a fibration. since $i$ has the left lifting property with respect to $p$ we can find a diagonal filler:
$$\xymatrix@1{
& A  \ar[r]^{j} \ar[d]^{i}
& E  \ar[d]^{p} \\
& B \ar@{=}[r] \ar@{-->}[ur]^{k}
& B
}$$
hence $i$ is a retract of $j$, so an anodyne extension as well.
\end{proof}

\begin{defn}
A fibration $P: X \to Y$ is a \textbf{locally trivial}\index{fibration!locally trivial} or a \textbf{fibre bundle}\index{fibre bundle} iff for every simplex $y: \Delta[n] \to Y$ of $Y$, the pullback fibration $y^*p : y^*X=\Delta[n] \times_{Y} X \to \Delta[n]$ is isomorphic over $\Delta[n]$ to a product fibration $\pi_1 : \Delta[n] \times F \to \Delta[n]$.
\end{defn}

\section*{Homotopy Groups}

\begin{defn}
Given two maps of simplicial sets $f,g:X \to Y$ we say that a \textbf{homotopy}\index{homotopy!of simplicial maps} between them is a simplicial map $h:\Delta[1] \times X \to Y$ such that $h(0,x)=f(x)$ and $h(1,x)=g(x)$ for all $x$.
\end{defn}

\nind
Observe that by the definition of the $\mathit{Hom}$-complex a homotopy between two maps $f$ and $g$ can be equivalently described as $1$-simplex $h$ connecting the two vertices $f$ and $g$ i.e. $d_0 h = f$ and $d_1 h = g$. Hence we give the following more general definition.

\begin{defn}
Given a simplicial set and $x,y \in X_0$ two vertices, we say that $x$ is \textbf{homotopic}\index{homotopy!of simplicial vertices} to $y$, written $x \sim y$ iff there is a $1$-simplex $x \in X_1$ such that $d_0 z = x$ and $d_1 z = y$.
\end{defn}

\nind
Unfortunately, as we shall see in the sequel, this relation is not an equivalence relation for arbitrary simplicial sets. However, we have the following result.

\begin{lem}
If $X$ is a Kan complex, then the homotopy of vertices is an equivalence relation.
\end{lem}
\begin{proof}
It is obviously reflexive, since if $x \in X_0$ we have $d_1 s_0 x = d_0 s_0 x = x$. If $x \sim y$, we have a $1$-simplex $z$ connecting the two vertices. Then we get a map $f : \Lambda^0[2] \to X$ which is $s_0 x$ on $d_1 i_2$ and $z$ on $d_2 i_2$. Because $X$ is fibrant, there is an extension of $f$ to a $2$-simplex $w \in X_2$. Then $d_0 w$ is the required homotopy from $y$ to $x$.
Finally, for the transitivity suppose $x \sim y$ and $y \sim z$, so that we have $1$-simplices $a$ and $b$ connecting respectively $x$ to $y$, and $y$ to $z$. Then $a$ and $b$ define a map $f: \Lambda^1[2] \to X$ which is $a$ on $d_2 i_2$ and $b$ on $d_0 i_2$. Since $X$ is fibrant we can extend $f$ to a $2$-simplex $c \in X_2$. Then $d_1 c$ is the required homotopy.
\end{proof}

\begin{defn}
Given a fibrant simplicial set $X$ and $v \in X_0$ a vertex, the \textbf{$n$-th homotopy group}\index{homotopy!groups} written $\pi_n(X , v)$ is the set of equivalence classes of singular simplices $\alpha : \Delta[n] \to X$ that send $\partial \Delta[n] $ to $v$, under the equivalence relation defined by $\alpha \sim \beta$ iff there is a homotopy $H:\Delta[1] \times \Delta[n] \to X$ such that $H$ is $\alpha$ on $\{0\} \times \Delta[n]$, $\beta$ on $\{1\} \times \Delta[n]$ and is the constant map $v$ on $\Delta[1] \times \partial \Delta[n]$.
\end{defn}

\nind
After this definition the reader may ask if the notion of weak equivalence can be equivalently defined requiring that the map induces isomorphisms in all simplicial homotopy groups.

\nind
Using the fibrancy of $X$ it can be shown that $\sim$ is an equivalence relation.\\
Given a map $f:X \to Y$ there is as usual an induced map $\pi_n(f) : \pi_n(X , v) \to \pi_n(Y , f(v))$, making $\pi_n$ functorial.\\
It is not clear from this definition if $\pi_n$ are actually groups, nor if the notion of weak equivalence can be equivalently defined requiring that the map induces isomorphisms in all simplicial homotopy groups. An answer to both questions is given by the following theorem that we state here although it requires notions from the next section:

\begin{thm}
Let $X$ be a fibrant simplicial set and $v \in X_0$ a vertex. Then there is a natural isomorphism $\pi_n (X,v) \cong \pi_n(|X|, |v|)$.
\end{thm}
\begin{proof}
See proposition 3.6.3 in \cite{Hov}.
\end{proof}

\begin{lem}\label{retractlem}
The vertex $n \in \Delta[n]$ is a deformation retract of $\Delta[n]$, in the sense that there is a homotopy $H : \Delta[1] \times \Delta[n] \to \Delta[n]$ from the identity map to the constant map at $n$, which sends $\Delta[1] \times \{ n \}$ to $\{ n \}$.\\
Moreover, this homotopy restricts to a deformation retraction of $\Lambda^n[n]$ onto its vertex $n$.
\end{lem}
\begin{proof}
A simplex of $\Delta[1] \times \Delta[n]$ is a chain in the ordered set $[1] \times [n]$. Hence a homotopy $H : \Delta[n] \times \Delta[1] \to \Delta[n]$ is the same as an ordered map $[1] \times [n] \to [n]$. We choose tha map that takes $(k,0)$ to $k$ and $(k,1)$ to $n$. The corresponding homotopy is the desired $H$.
\end{proof}

\begin{rem}
The standard simplex $\Delta[n]$ is not fibrant.\\
Indeed, in the previous lemma we have found a homotopy from the identity map of the standard simplex to the constant map $v$. There no homotopy going in the opposite direction, because such a homotopy would be induced by a map of ordered sets that takes $(k,0)$ to $n$ and $(k,1)$ to $k$, and there is no such map. Hence homotopy is not an equivalence relation of the set of endomorphisms of $\Delta[n]$, so that by theorem \ref{homfib} the standard simplex cannot be fibrant.
\end{rem}

\begin{thm}
Suppose $X$ is a fibrant simplicial set, $v \in X_0$ a vertex and $\alpha: \Delta[n] \to X$ a simplex such that $d_i \alpha = v$ for all $i$. Then $[\alpha] = [v] \in \pi_n(X,v)$ iff there is an ($n+1$)-simplex $x \in X$ such that $d_{n+1} x = \alpha$ and $d_i x = v$ for all $i \leq n$.
\end{thm}
\begin{proof}
see lemma 3.4.5 in \cite{Hov}.
\end{proof}

\begin{rem}
We want to develop a simplicial analogue of Serre long exact sequence, as a first step we construct the connecting map $\partial$.\\
Let $p:X \to Y$ be a fibration of simplicial sets and $v \in X_0$ a vertex. Let $F$ denote the fibre of $p$ at $p(v)$. We now construct a map $\partial :\pi_n(Y , p(v)) \to \pi_{n-1}(F , v)$ as follows: given a class $[\alpha] \in \pi_n(Y , p(v)) $, let $\gamma$ be the lift in the following diagram:
$$\xymatrix@1{
& \Lambda^k[n]   \ar[r]^{v} \ar[d]
& X  \ar[d]^{p} \\
& \Delta[n] \ar[r]^{\alpha} \ar@{-->}[ru]^{\gamma}
& Y
}$$
and define $\partial[\alpha] := d_n(\gamma)$ where $d_n$ is the $n$-th face map. The commutativity of the diagram implies that $d_n(\gamma)$ lies in the fibre, and it is easy to see that $d_i d_n \gamma = v$ so that by the previous lemma $[d_n \gamma ] \in \pi_{n-1}(F , v)$.\\
We should check that $\partial$ is well defined, but it is a standard exercise in lifting pushout products of anodyne extensions. For the details see lemma 3.4.8 in \cite{Hov}.
\end{rem}

\begin{thm}[the long exact sequence]
Let $p: X \to Y$ be a fibration between fibrant simplicial sets, ad $v \in X_0$ a vertex. Let $F$ denote the fibre of $p$ over $p(v)$.\\
Then we have an exact sequence of pointed sets:
$$\dots \to \pi_n(X , v) \to \pi_n(Y , p(v)) \to \pi_{n-1}(F , v) \to \pi_{n-1}(X , v) \to \dots$$
\end{thm}
\begin{proof}
The proof is analogous to the standard one for topological spaces, for some details see lemma 3.4.9 in \cite{Hov}.
\end{proof}

\section*{Minimal Fibrations}

We know from theorem \ref{fibrefibration} that the fibres of a fibrations are homotopy equivalent. Now we would like to find some reasonable sufficient condition for a fibration to have isomorphic fibres. Note that by definition every locally trivial morphism has isomorphic fibres, and as we shall see our condition will force the fibration to be a locally trivial one.\\
Let start with some preliminary definitions on homotopy equivalences.

\begin{defn}
Two maps $f: X \to Z$ and $g:Y \to Z$ are \textbf{fibre homotopy equivalent}\index{fibre homotopy equivalence} iff there are maps $\theta : X \to Y$ and $\omega : Y \to X$ such that $g \circ \theta = f$ and $f \circ \omega = g$, and there are homotopies from $\theta \circ \omega$ to the identity $1_{Y}$ and from $\omega \circ \theta$ to the other identity $1_X$ that cover the constant homotopy of $Z$.
\end{defn}

\begin{thm}
Let $p: X \to Y$ be a fibration of simplicial sets, and suppose $f,g : K \to Y$ are maps such that there is a homotopy from $f$ to $g$. Then the pullback fibrations $f^*p$ and $g^*p$ are fibre homotopy equivalent.
\end{thm}
\begin{proof}
See proposition 3.5.3 in \cite{Hov}.
\end{proof}

\begin{cor}\label{fibrecor}
Let $p:X \to Y$ be a fibration of simplicial sets and let $y: \Delta[n] \to Y$ be a singular simplex in $Y$. Then the pullback $y^*X \to \Delta[n]$ is fibre homotopy equivalent to the product fibration $\Delta[n] \times F_n \to \Delta[n]$, where $F_n$ is the fibre of $p$ over the vertex $y(n)$.
\end{cor}
\begin{proof}
By lemma \ref{retractlem} the identity map of $\Delta[n]$ is homotopic to the constant map $n$, hence the thesis follows from the previous theorem.
\end{proof}

\begin{defn}
Given a fibration $p:X \to Y$, two $n$-simplices $x,y \in X_n$ are \textbf{$p$-related}\index{$p$-related simplices} written $x \sim_p y$ iff they represent vertices in the same path component of the same fibre of $\underline{\mathit{Hom}}(i,p): \underline{\mathit{Hom}}(\Delta[n],X) \to \underline{\mathit{Hom}}(\partial \Delta[n],X) \times_{\underline{\mathit{Hom}}(\partial \Delta[n],Y)} \underline{\mathit{Hom}}(\Delta[n],Y)$.
\end{defn}

\nind
Observe that thanks to theorem \ref{homfib} $\sim_p$ is an equivalence relation.\\
Moreover, we can rewrite the definition as follows: $x \sim_p y$ iff $p(x) = p(y)$, $d_i x = d_i y$ for all $0 \leq i \leq n$ and there is a fibrewise homotopy stationary on the boundary $H:\Delta[1] \times \Delta[n] \to X$ from $x$ to $y$, i.e. $pH$ is the constant homotopy and $H$ is constant on $\partial \Delta[n]$.

\begin{defn}
A fibration $p:X \to Y$ is \textbf{minimal}\index{fibration!minimal} iff $x \sim_p y$ entails $x=y$, i.e. iff every path component of every fibre of the fibration $\underline{\mathit{Hom}}(i,p)$ has only one vertex.
\end{defn}

\nind
We underline the similarity between the defining condition for minimal fibrations and the - for the moment prosaic and imprecise - statement of univalence, that equivalent objects are equal. It is not surprisingly that we shall use minimal fibrations to prove that univalence holds in simplicial sets.

\begin{lem}
Minimal fibrations are stable under pullbacks.
\end{lem}
\begin{proof}
Suppose $p: X \to Y$ be a fibration, every pullback square:
$$\xymatrix@1{
& X'   \ar[r] \ar[d]^{p'}
& X  \ar[d]^{p} \\
& Y' \ar[r]
& Y
}$$
induces a pullback square:
$$\xymatrix@1{
& \underline{\mathit{Hom}}(\Delta[n],X')   \ar[r] \ar[d]^{\underline{\mathit{Hom}}(i,p')}
& \underline{\mathit{Hom}}(\Delta[n],X)  \ar[d]^{\underline{\mathit{Hom}}(i,p)} \\
& P' \ar[r]
& P
}$$
where
$$P' := \underline{\mathit{Hom}}(\Delta[n],Y') \times_{\underline{\mathit{Hom}}(\partial \Delta[n],Y')} \underline{\mathit{Hom}}(\partial \Delta[n],X')$$
and
$$P := \underline{\mathit{Hom}}(\Delta[n],Y) \times_{\underline{\mathit{Hom}}(\partial \Delta[n],Y)} \underline{\mathit{Hom}}(\partial \Delta[n],X)$$
Hence every fibre of $\underline{\mathit{Hom}}(i,p')$ is isomorphic to a fibre of $\underline{\mathit{Hom}}(i,p)$.
\end{proof}

\begin{lem}\label{minimalEqIso}
Let $p: X \to Y$ and $q: Z \to Y$ be fibrations of simplicial sets, and that $q$ is a minimal fibration. Suppose $f,g : X \to Z$ are two maps of simplicial sets over $Y$ and suppose $H: X \times \Delta[1] \to Z$ be a homotopy from $f$ to $g$ such that $qH = p \pi_1$.\\
If $g$ is an isomorphism so is $f$.
\end{lem}
\begin{proof}
See lemma 3.5.6 in \cite{Hov}.
\end{proof}

\begin{thm}\label{minimalbundle}
Let $p:X \to Y$ be a minimal fibration of simplicial sets, then $p$ is locally trivial.
\end{thm}
\begin{proof}
Since minimal fibrations are stable under pullbacks we have that the pullback of $p$ along any singular simplex $y: \Delta[n] \to Y$ is a minimal fibration $y^* X \to \Delta[n]$. We know from corollary \ref{fibrecor} that this map is fibre homotopy equivalent to the product fibration $\Delta[n] \times F_n \to \Delta[n]$. By the previous theorem we obtain an isomorphism.
\end{proof}

\begin{thm}\label{minimalfibration}
Let $p: X \to Y$ be a fibration of simplicial sets. Then we can factor $p$ as $\displaystyle X \mathop{\to}^{r} X' \mathop{\to}^{p'} Y$, where $p'$ is a minimal fibration and $r$ is a retraction onto a subsimplicial set $X' \subseteq X$.
\end{thm}
\begin{proof}
We give a sketch of the proof: by the axiom of choice let $T$ be a set of simplices of $X$ containing one simplex from each $p$-equivalence class. It is easy to check that it contains every degenerate simplex (see lemma 3.5.8 in \cite{Hov}). Let $S$ denote the set of all subsimplicial sets of $X$ whose simplices lie in $T$. We put a partial order on $S$ induced by the one already present on subsimplicial sets of $X$ and by Zorn's lemma we get a maximal subsimplicial set $X'$.\\
If the restriction $p':X' \to Y$ is a fibration it will be automatically minimal. We will show that $p'$ is a retract of $p$ hence it will be a fibration.\\
We apply again Zorn's lemma, now to pairs $(Z,H)$ where $Z$ is a subsimplicial set of $X$ containing $X'$, and $H : \Delta[1] \times Z \to X$ is a homotopy such that its restriction to $\{0\} \times Z$ is the inclusion, it maps $\{1\} \times Z$ into $X'$, it is constant on $\Delta[1] \times X'$, and $pH$ is the constant homotopy of $p$ restricted to $Z$. Let $(Z,H)$ be a maximal pair, we must show that $Z = X$. If not consider a simplex $x : \Delta[n] \to X$ of minimal dimension that does not belongs to $Z$. We then construct the following pushout square:
$$\xymatrix@1{
& \partial \Delta[n] \ar[r] \ar[d]
& Z  \ar[d] \\
& \Delta[n] \ar[r]^{x}
& Z'
}$$
where $Z'$ is the subsimplicial set of $X$ generated by $Z$ and $x$. Finally, using the previous construction we can extend $H$ contradicting its maximality. For these remaining details see theorem 3.5.9 in \cite{Hov}.
\end{proof}

\begin{defn}
A map of topological spaces is a \textbf{Serre fibration}\index{fibration!topological}\index{fibration!Serre} iff it has the right lifting property with respect to all inclusions $D^n \to D^n \times I$ of the $n$-disk at ground zero of the cylinder over the $n$-disk.
\end{defn}

\begin{thm}\label{realisationSerre}
Let $p: X \to Y$ be a fibration of simplicial sets. Then $|p|$ is a Serre fibration of Kelley spaces.
\end{thm}
\begin{proof}
The statement is obtained adding at first the hypothesis that $p$ is locally trivial and then the hypothesis is dropped. See corollary 3.6.2 in \cite{Hov}.
\end{proof}

\chapter{Model Categories}

\lettrine{T}{here} are many different possible settings for an abstract theory of homotopy, perhaps the easiest way to define a homotopical structure on a category is to distinguish a suitable class of "weak equivalences" that contains the isomorphisms and is closed under some basic constructions, is then possible to form its \textit{homotopy category}\index{homotopy!category} localising with respect to this family of arrows.\\
But in a such general framework is difficult to manage the homotopy category and the weak equivalences; a way to face this problem is to add more structure taking topological spaces and simplicial sets as guiding examples. The definition of model category adds two other classes of morphisms to the structure, namely a class of fibrations and a class of cofibrations. 
Moreover, the presence of these maps allows to define well-behaved objects and the interaction of these two classes of maps gives a finer control on the homotopy category that turns out to be a quotient of the initial model category under a homotopy equivalence relation.\\
The aim of this chapter is to give some basics in the theory of model categories and an outline of the proof that the category of simplicial sets admits a model structure. The proof uses the axiom of choice because it relies on the theory of minimal fibrations in particular on theorem \ref{minimalfibration}.

\section*{Definition of Model Category}

We start with some introductory definitions, motivated by the theory of simplicial sets that we have seen so far.

\begin{defn}
\begin{enumerate}[label=(\alph*)]
\item[]

\item Recall that given a category $\mathcal C$ its \textbf{category of arrows}\index{category!of arrows}, written $Ar(\mathcal C)$, has arrows of $\mathcal C$ as objects and commutative squares of $\mathcal C$ as morphisms.

\item Given a category $\mathcal C$, a \textbf{functorial factorisation}\index{functorial factorisation} is a pair $(\alpha , \beta)$ of functors $Ar(\mathcal C) \to Ar(\mathcal C)$ such that for any arrow $f$ of $\mathcal C$ we have a factorisation as $f = \beta(f) \circ \alpha(f)$.
\end{enumerate}
\end{defn}

\begin{defn}
A \textbf{subcategory of weak equivalences}\index{subcategory of weak equivalences} of a given category is a subcategory that contains all the isomorphisms and satisfies the \textit{2-out-of-3 property} i.e. given any pair of composable maps $f$ and $g$ if two elements of the set $\{ f , g , f \circ g \}$ are weak equivalences so is the third.
\end{defn}

\nind
Following Hovey's book we require that model structures has fixed functorial factorisations, instead of simply asking for the existence of a factorisation. This is needed in order to perform some constructions in a canonical way, which otherwise would depend on the choice of the factorisation.

\begin{defn}
A \textbf{weak factorisation system}\index{weak factorisation system} is a quadruple $(\mathit A , \mathit B , \alpha , \beta)$ where the first two are classes of maps and $(\alpha , \beta)$ is a functorial factorisation such that $\alpha(f) \in \mathit A$ and $\beta(f) \in \mathit B$. Moreover, $\mathit A$ is the class of maps which have the left lifting property with respect to $\mathit B$ and conversely $\mathit B$ is the class of maps which have the right lifting property with respect to $\mathit A$.
\end{defn}

\nind
Notice that each of the two classes of a weak factorisation system determines the other.\\
Also, $\mathit A \cap \mathit B$ is exactly the class of isomorphisms since any map $f$ such that $f$ has the left lifting property with respect to to itself is an isomorphism.

\begin{defn}
A \textbf{model structure}\index{model structure} on a category consists in three subcategories $\mathit{Fib}$, $\mathit{Cof}$ and $\mathit W$, and two functorial factorisations $(\alpha , \beta)$ and $(\delta , \gamma)$ such that $\mathit W$ is a subcategory of weak equivalences, $(\mathit{Cof} \cap \mathit W , \mathit{Fib} , \alpha , \beta)$ and $(\mathit{Cof} , \mathit W \cap \mathit{Fib} , \delta , \gamma)$ are two weak factorisation systems.\\
Maps in $\mathit W$ are called \textbf{acyclic}\index{acyclic map} of \textbf{weak equivalences}\index{weak equivalence!in a model category}. Maps in $\mathit{Cof}$ are called \textbf{cofibrations}\index{cofibration!in a model category} and maps in $\mathit{Fib}$ are called \textbf{fibrations}\index{fibration!in a model category}.\\
A \textbf{model category}\index{model category} is a complete and cocomplete category together with a model structure.
\end{defn}

\nind
Every complete and cocomplete category admits three trivial examples of model structure, the ones where one of the three classes $\mathit{Cof}$ , $\mathit{Fib}$ and $\mathit W$ is the class of isomorphisms and the other two are all maps.\\
Since Quillen axioms for model categories are self-dual, if a category $\mathcal C$ has a model structure then also the opposite category admits one. The cofibration of $\mathcal C^{op}$ are the fibration of $\mathcal C$ and vice-versa.\\
If $\mathcal C$ is a model category, then every slice $\mathcal C / X$ inherits a model structure, where a map is a fibration, cofibration or weak equivalence iff it is in $\mathcal C$.\\
Given two model categories $\mathcal C$ and $\mathcal D$ we put on the product $\mathcal C \times \mathcal D$ the model structure where a map $(f,g)$ is a fibration [cofibration, weak equivalence] iff both $f$ and $g$ are fibrations [cofibrations, weak equivalences].\\
The definitions of fibrant and cofibrant objects are the same as in the case of simplicial sets. Note that if we factor the map $0 \to X$ into a cofibration followed by a trivial fibration we get a functor $X \mapsto QX$ and a map $q_X : QX \to X$ such that $QX$ is cofibrant and $q_X$ is a trivial fibration.\\
Dually, when we factor $X \to RX \to 1$ as a trivial cofibration followed by a fibration.

\begin{defn}
The two functors defined above are called respectively \textbf{cofibrant replacement functor}\index{cofibrant replacement functor} and \textbf{fibrant replacement functor}\index{fibrant replacement functor}.
\end{defn}

\begin{lem}[the retract argument]
If we have a factorization $f=pi$ in a category such that $f$ has the left lifting property with respect to $p$, then $f$ is a retract of $i$.\\
Dually, if $f$ has the right lifting property with respect to $i$, then it is a retract of $p$.
\end{lem}
\begin{proof}
First suppose $f$ has the left lifting property with respect to $p$. Let consider the diagram with the lifting $r$:
$$\xymatrix@1{
& A   \ar[r]^{i} \ar[d]_{f}
& B  \ar[d]^{p} \\
& C \ar@{=}[r] \ar@{-->}[ru]^{r}
& C
}$$
Hence $f$ is a retract of $i$ with retraction $r$.
\end{proof}

\begin{lem}
In a model category the three classes $\mathit{Cof}$, $\mathit{Fib}$ and $\mathit W$ are closed under retracts.
\end{lem}
\begin{proof}
Fibrations and cofibrations are closed under retracts because of the previous lemma.\\
For weak equivalences see proposition A.3.1 in \cite{JT}.
\end{proof}

\nind
Note that thanks to their characterisation in terms of lifting properties cofibrations are closed under pushouts and fibrations are closed under pullbacks.

\begin{lem}[Ken Brown's lemma]
Let $\mathcal C$ a model category and $\mathcal D$ a category with a distinguished class of "weak equivalences" satisfying the 2-out-of-3 property. If $F: \mathcal C \to \mathcal D$ is a functor which takes trivial cofibrations between cofibrant objects to weak equivalences. Then $F$ takes all weak equivalences between cofibrant objects to weak equivalences.
\end{lem}
\begin{proof}
We suppose that $f:A \to B$ is a weak equivalence of cofibrant objects and factor $(f , 1_B) : A + B \to B$ into a cofibration followed by a trivial fibration $\displaystyle A + B  \mathop{\longrightarrow}^{q} C \mathop{\longrightarrow}^{p} B$. The pushout diagram:
$$\xymatrix@1{
& 0  \ar[r] \ar[d]
& A  \ar[d] \\
& B  \ar[r] 
& A  + B
}$$
shows that the inclusions $i_1 : A \to A + B$ and $i_2 : B \to A + B$ are cofibrations. By the 2-out-of-3 axiom we have that $q \circ i_1$ and $q \circ i_2$ are weak equivalences hence trivial cofibrations. By hypothesis we have that both $F(q \circ i_1)$ and $F(q \circ i_2)$ are weak equivalences. Since $F(p \circ q \circ i_2) = F(1_B)$ is also a weak equivalence we conclude from the 2-out-of-3 property that $F(p)$ is a weak equivalence, and hence that $F(f) = F(p \circ q \circ i_1)$ is a weak equivalence as required. 
\end{proof}

\begin{defn}
A model category is \textbf{right proper}\index{model category!right proper} iff the pullback of a weak equivalence along a fibration is again a weak equivalence.\\
It is \textbf{left proper}\index{model category!left proper} iff the pushout of a weak equivalence along a cofibration is again a weak equivalence.\\
Finally, a model category is \textbf{proper}\index{model category! proper} iff it is both left and right proper.
\end{defn}

\nind
We now define morphisms and equivalences of model categories.

\begin{defn}
Let $\mathcal C$ and $\mathcal D$ be model categories.
\begin{enumerate}
\item A functor $F: \mathcal C \to \mathcal D$ is a \textbf{left Quillen functor}\index{left Quillen functor} iff $F$ is a left adjoint and preserves cofibrations and trivial cofibrations.

\item A functor $U: \mathcal D \to \mathcal C$ is a \textbf{right Quillen functor}\index{right Quillen functor} iff it is a right adjoint and preserves fibrations and trivial fibrations.

\item An adjunction $(F,U, \phi) : \mathcal C \to \mathcal D$ is a \textbf{Quillen adjunction}\index{Quillen!adjunction} iff $F$ is a left Quillen functor iff $U$ is a right Quillen functor.

\item A Quillen adjunction $(F,U, \phi) : \mathcal C \to \mathcal D$ is a \textbf{Quillen equivalence}\index{Quillen!equivalence} iff for all cofibrant $X \in Ob(\mathcal C)$ and all fibrant $Y \in Ob(\mathcal D)$ we have that $f \in \mathit{Hom}_{\mathcal D}(FX , Y)$ is a weak equivalence iff $\phi(f) \in \mathit{Hom}_{\mathcal D}(X , UY)$ is a weak equivalence.
\end{enumerate}
\end{defn}

\section*{Homotopy in Model Categories}

\begin{defn}
Given a category $\mathcal C$ with a subcategory of weak equivalences $\mathit W$ we define its \textbf{homotopy category}\index{homotopy category}, written $Ho(\mathcal C)$\index{$Ho(\mathcal C)$} as the localisation $\mathcal C [\mathit W ^{-1}]$, obtained inverting formally all the arrows in $\mathit W$.
\end{defn}

\nind
Observe that for arbitrary categories with a subcategory of weak equivalences a size issue emerges. Indeed, due to the size of $\mathit W$ the homotopy category may have proper classes as $\mathit{Hom}$ even if the category $\mathcal C$ has $\mathit{Hom}$-sets.\\
We will state a theorem which guarantees that for model categories this phenomenon disappears, in fact the homotopy category turns out to be a quotient of the model category.

Now we focus on homotopies. Observe that for simplicial sets and topological spaces, homotopies admit two dual descriptions: as maps from a cylinder $H: X \times I \to Y$ or as paths in a path space $X \to Y^I$. Moreover, any cylinder retracts onto its base so that in particular we have a weak equivalence $I \times X \to X$, similarly for the path space because paths can be contracted to their starting point. Hence we split the notion of homotopy into left and right homotopies expressing some key properties in terms of fibrations, cofibrations and weak equivalences, using these two dual descriptions as guidance.

\begin{defn}
Let $\mathcal C$ be a model category and $f,g : B \to X$ two maps.
\begin{enumerate}[label=(\alph*)]

\item A \textbf{cylinder object}\index{cylinder object} for $B$ is a factorisation of the codiagonal $\nabla : B + B \to B$ into a cofibration $i_0 + i_1 : B + B \to IB$ followed by a weak equivalence $s:IB \to B$.

\item A \textbf{path object}\index{path object} for $X$ is a factorisation of the diagonal $\Delta : X \to X \times X$ into a weak equivalence $r:X \to PX$ followed by a fibration $(p_0 , p_1) : PX \to X \times X$.\\
A \textbf{very good path object}\index{path object!very good} is a factorisation of the diagonal as a trivial cofibration followed by a fibration.

\item A \textbf{right homotopy}\index{homotopy!right} from $f$ to $g$ is a map $K : B \to PX$ into some path object for $X$ such that $p_0K=f$ and $p_1K=g$.
We say that $f$ is \textbf{right homotopic} to $g$, written $\displaystyle f \mathop{\sim}^{r} g$\index{$\displaystyle \mathop{\sim}^{r}$} iff there exists a right homotopy from $f$ to $g$.

\item A \textbf{left homotopy}\index{homotopy!left} from $f$ to $g$ is a map $H : IB \to X$ from some cylinder object for $B$ such that $Hi_0=f$ and $Hi_1=g$.
We say that $f$ is \textbf{left homotopic} to $g$, written $\displaystyle f \mathop{\sim}^{l} g$\index{$\displaystyle \mathop{\sim}^{l}$} iff there exists a left homotopy from $f$ to $g$.

\item $f$ and $g$ are \textbf{homotopic}\index{homotopy!abstract}, written $f \sim g$ iff $f$ is both left and right homotopic to $g$.

\item $f$ is a \textbf{homotopy equivalence}\index{homotopy!equivalence} iff there is a map $h:X \to B$ such that $fh \sim 1_X$ and $hf \sim 1_B$.
\end{enumerate}
\end{defn}

\nind
Notice that in any model category both cylinder and path objects exist, simply applying the factorisation to the codiagonal and the diagonal respectively.\\
Recalling the case of simplicial sets, we expect that some requirement on objects is needed in order to make homotopy relations into equivalence relations.

\begin{lem}
Let $f, g: A \to X$ be two maps in a model category $\mathcal C$. We have the following:
\begin{enumerate}[label=(\alph*)]
\item If $A$ is cofibrant, then left homotopy is an equivalence relation on $\mathit{Hom}(A , X)$. Dually, if $X$ is fibrant, then right homotopy is an equivalence relation on $\mathit{Hom}(A , X).$

\item If $A$ is cofibrant, $\displaystyle f \mathop{\sim}^{l} g $ implies $\displaystyle f \mathop{\sim}^{r} g$. Dually, if $X$ is fibrant $\displaystyle f \mathop{\sim}^{r} g$ implies $\displaystyle f \mathop{\sim}^{l} g$.

\item Let $\mathcal C_{cf}$\index{$\mathcal C_{cf}$} be the full subcategory of fibrant and cofibrant objects of $\mathcal C$. The homotopy relation on the maps of $\mathcal C_{cf}$ is an equivalence relation compatible with composition.
\end{enumerate}
\end{lem}
\begin{proof}
See proposition 1.2.5 in \cite{Hov}.
\end{proof}

\begin{lem}
Let $\mathcal C$ be a model category and $A$ a cofibrant object. If $f:X \to Y$ is either a trivial fibration or a weak equivalence between fibrant objects, then the map $\displaystyle \mathit{Hom}(A , X)/ \mathop{\sim}^{l} \to \mathit{Hom}(A , Y) / \mathop{\sim}^{l}$ induced by $f$ is a bijection.\\
Dually, if $X$ is a fibrant object and $f:A \to B$ is either a trivial cofibration or a weak equivalence between cofibrant objects, then the map $\displaystyle \mathit{Hom}(B , X)/ \mathop{\sim}^{r} \to \mathit{Hom}(A , X) / \mathop{\sim}^{r}$ induced by $f$ is a bijection.
\end{lem}
\begin{proof}
Let $f:X \to Y$ be a trivial fibration and consider the map induced by $f$. Since $A$ is cofibrant, any map $A \to Y$ has a lifting to $X$, so the map is surjective even before passing to left homotopy. On the other hand, let $\displaystyle A + A  \mathop{\longrightarrow}^{(i_0 , i_1)} IA \mathop{\longrightarrow}^{s} A$ be a cylinder for $A$, and $H:IA \to Y$ a left homotopy between $fg , fh: A \to Y$. Then the diagram:
$$\xymatrix@1{
& A + A  \ar[r]^{(g,h)} \ar[d]_{(i_0 , i_1)}
& X  \ar[d]^{f} \\
& IA  \ar[r] ^{H} \ar@{-->}[ru]^{K}
& Y 
}$$
has a diagonal filler $K:IA \to X$, which is a left homotopy between $g,h:A \to X$. Thus the map induced by $f$ is also injective. The case when $f:X \to Y$ is a weak equivalence between fibrant objects follows from the first part of Ken Brown's lemma using the functor $\mathit{Hom}(A , -) / \mathop{\sim}^{l}$ and bijections of sets as "weak equivalences".
\end{proof}

\begin{thm}[Whitehead's theorem]\label{White}
If $\mathcal C$ is a model category, then a map of $\mathcal C_{cf}$ is a weak equivalence iff it is a homotopy equivalence.
\end{thm}
\begin{proof}
Let $X$, $Y$ and $A$ be objects in $\mathcal C_{cf}$ and $f: X \to Y$ be a weak equivalence. Then by the previous lemma the map $\mathit{Hom}_{\mathcal C_{cf}}(A , X) / \hspace{-1pt} \sim \, \to \mathit{Hom}_{\mathcal C_{cf}}(A , Y) / \hspace{-1pt} \sim$ induced by $f$ is a bijection. Taking $A=Y$ we find a map $f':Y \to X$ such that $ff' \sim 1_Y$. Since $ff'f \sim f$ it follows that $f'f \sim 1_X$.\\
Conversely, given $X$ and $Y$ objects in $\mathcal C_{cf}$ and $p:X \to Y$ which is a homotopy equivalence, we consider at first the case when it is a fibration. We will show that $p$ is a weak equivalence. So let $p':Y \to X$ be a homotopy inverse for $p$ and $H:IY \to Y$ be a left homotopy $pp' \sim 1_Y$. Since $p$ is a fibration there is a diagonal filler $H': IY \to X$ in the diagram:
$$\xymatrix@1{
& Y  \ar[r]^{p'} \ar[d]_{i_0}
& X  \ar[d]^{p} \\
& IY  \ar[r] ^{H} \ar@{-->}[ru]^{H'}
& Y 
}$$
We define $q := H'i_1$, then $pq = 1_Y$ and $H' : p' \sim q$. We then have $1_X \sim p'p \sim qp$, so let $K:IX \to X$ be a left homotopy such that $Ki_0 = 1_X$ and $Ki_1 = qp$. By the 2-out-of-3 property, $K$ is a weak equivalence, making $qp$ a weak equivalence as well. From the diagram:
$$\xymatrix@1{
& X  \ar@{=}[r] \ar[d]_{p}
& X  \ar@{=}[r] \ar[d]^{qp}
& X  \ar[d]^{p} \ar[d]^{p} \\
& Y  \ar[r]^{q}
& X  \ar[r]^{p}
& Y 
}$$
we see that $p$ is a retract of $pq$ hence it is a weak equivalence.\\
Now we remove the additional hypothesis so let $f:X \to Y$ be an arbitrary homotopy equivalence. Factor $f$ as $f = pi: X \to Z \to Y$ with $i$ a trivial cofibration and $p$ a fibration. Then $Z$ is cofibrant and fibrant, so by the first part $i$ is a homotopy equivalence. It follows from the 2-out-of-3 property for the homotopy equivalences of $\mathcal C_{cf}$ that $p$ is a homotopy equivalence. Then by the argument above $p$ is a weak equivalence, hence $f$ is a weak equivalence.
\end{proof}

\begin{thm}
Let $\pi \mathcal C_{cf}$ be the quotient category modulo homotopy. There is an isomorphism of categories $\pi \mathcal C_{cf} \to Ho(\mathcal C_{cf})$.
\end{thm}
\begin{proof}
See corollary 1.2.9 in \cite{Hov}.
\end{proof}

\begin{thm}\label{lemmaRP}
In a model category, the pullback of a weak equivalence between fibrant objects along a fibration is a weak equivalence. Therefore, a model category in which all objects are fibrant is right proper.
Dually for left properness.
\end{thm}
\begin{proof}
Let $w:X \to Y$ be a weak equivalence between fibrant objects, in the following pullback square:
$$\xymatrix@1{
& X'  \ar[r]^{f'} \ar[d]^{w'}
& X  \ar[d]^{w} \\
& Y'  \ar[r]^{f} 
& Y 
}$$
we want to show that $w'$ is a weak equivalence if $f$ is a fibration. First, consider a factorisation of $\displaystyle (1_X , f): X \mathop{\longrightarrow}^{i_X} Pf \mathop{\longrightarrow}^{(p_X , p_Y)} X \times Y$ as a weak equivalence followed by a fibration. $p_X i_X = 1_X$ and $w=p_Y i_X$. Since $X$ and $Y$ are fibrant, $p_X$ and $p_Y$ are trivial fibrations. The pullback of a trivial fibration is a weak equivalence, so we may suppose $w:X \to Y$ is a weak equivalence that has a retraction $r:Y \to X$ such that $rw = 1_X$ and $r$ is a trivial fibration.
Now consider the diagram:
$$\xymatrix@1{
& X'  \ar[r]^{w'} \ar[d]^{w'}
& Y'  \ar[d]^{u} \ar[dr]^{1_{Y'}} \\
& Y'  \ar[r]^{v} \ar[d]^{rf}
& Y  \times_X Y' \ar[r]^{p_2} \ar[d]^{p_1}
& Y' \ar[d]^{rf} \\
& X  \ar[r]^{w}
& Y  \ar[r]^{r}
& X
}$$
where $p_1u=f$, $p_1v = wrf$ and $p_2v = 1_{Y'}$. Now we have that $p_1uw' = fw'$ and $p_1vw' = wrfw' = wrwf' = wf'$. Also, $p_2uw' = w'$ and $p_2vw' = w'$ so the top left-hand square commutes. The lower left-hand square is a pullback, since the lower right-hand is and the two together are. similarly, the top left-hand square is a pullback since the lower left-hand square is and the two together are. It follows that $w' = w^*(u)$, where $w^* : \mathcal C / Y \to \mathcal C / X$ is the pullback functor. $r$ is a trivial fibration, hence so is $p_2$. Since $p_2u = 1_{Y'}$, $u$ is a weak equivalence in $\mathcal C / Y$ between fibrant objects. Since the pullback functor preserves fibrations and trivial fibrations $w'$ is a weak equivalence by Ken Brown's lemma.
\end{proof}


\section*{A Model Structure on Simplicial Sets}

As we have seen in the previous section the axioms of a model category are powerful. Hence we expect that in exchange it is difficult to build a model structure. This is indeed the case, in particular for simplicial sets.

\begin{lem}[recognition lemma]
Let $\mathcal C$ be a complete and cocomplete category provided with a subcategory $\mathit W$ which has the 2-out-of-3 property. Moreover, given other subcategories $\mathit{Fib}$ and $\mathit{Cof}$ and two other classes of maps $\mathit{C_{W}}$ and $\mathit{F_W}$, such that $(\mathit{C_W} , \mathit{Fib})$ and $(\mathit{Cof} , \mathit{F_W})$ are weak factorisation systems. If the following hold:
\begin{enumerate}[label=(\alph*)]
\item $\mathit{C_W} \subseteq \mathit{Cof} \cap \mathit W$ and $\mathit{F_W} \subseteq \mathit{Fib} \cap \mathit W$;

\item either $\mathit{Cof} \cap \mathit W \subseteq \mathit{C_W}$ or $ \mathit{Fib} \cap \mathit W \subseteq \mathit{F_W}$.

\end{enumerate}
Then $ \mathit{Fib}$, $\mathit{Cof}$ and $\mathit W$ determine a model structure on $\mathcal C$.
\end{lem}
\begin{proof}
Consider the case $\mathit{Fib} \cap \mathit{W} \subseteq \mathit{F_W}$. We want to show $\mathit{Cof} \cap \mathit{W} \subseteq \mathit{C_W}$, so let $i:A \to B$ be a trivial cofibration. We factor $i$ as $i=pj : A \to E \to B$, with $j \in \mathit{C_W}$ and $p \in \mathit{Fib}$. Since $i$ is a weak equivalence and $j \in \mathit{C_W} \subseteq \mathit{Cof} \cap \mathit{W}$ it follows that $p$ is a trivial fibration. Since $\mathit{Fib} \cap \mathit{W} \subseteq \mathit{F_W}$ and $i$ is a cofibration there is a diagonal filler $d$ in the diagram:
$$\xymatrix@1{
& A  \ar[r]^{j} \ar[d]^{i}
& E  \ar[d]^{p} \\
& B  \ar@{=}[r] \ar@{-->}[ru]^{d} 
& B 
}$$
It follows that $i$ is a retract of $j$ so $i \in \mathit{C_W}$. Thus $\mathit{C_W} = \mathit{Cof} \cap \mathit{W}$.
\end{proof}

\nind
In order to prove that the category of simplicial sets admits a model structure we choose $\mathit{Fib}$, $\mathit{Cof}$ and $\mathit{W}$ to be the classes of Kan fibrations, cofibrations and weak equivalences respectively. Moreover, $\mathit{C_W}$ is the class of anodyne extensions and $\mathit{F_W}$ the one of maps with the right lifting property with respect to cofibrations.\\
The following lemmas are quite easy to obtain.

\begin{lem}
Any simplicial map can be factored as a cofibration followed by a map which has the right lifting property with respect to cofibrations. Moreover, this factorisation is functorial.
\end{lem}
\begin{proof}
It is enough to repeat the argument of theorem \ref{factor}.
\end{proof}

\begin{lem}
Every anodyne extension is a trivial cofibration of simplicial sets
\end{lem}
\begin{proof}
See proposition 3.2.3 in \cite{Hov}.
\end{proof}

\begin{lem}
If $f$ is a map of simplicial sets which has the right lifting property with respect to cofibrations, then it is a trivial fibration.
\end{lem}
\begin{proof}
See proposition 3.2.6 in \cite{Hov}.
\end{proof}

\nind
The hard part is condition $(b)$ in the recognition lemma.

\begin{thm}
If $p$ is a trivial Kan fibration, then it has the right lifting property with respect to all cofibrations.
\end{thm}
\begin{proof}
Instead of breaking this proof into a chapter of lemmas we give here a sketch.\\
The idea of the proof is to start with simple cases and generalise step by step: first of all we observe that if $X$ is a non-empty Kan complex with no non-trivial homotopy groups, then the map $X \to 1$ has the RLP with respect to cofibrations. For the proof see proposition 3.4.7 in \cite{Hov}.\\
Next, if $p:X \to Y$ is a locally trivial fibration such that every fibre is non-empty and has no non-trivial homotopy groups, then $p$ has the RLP with respect to cofibrations. Indeed, since $p$ is locally trivial, a lift in the diagram:
$$\xymatrix@1{
& \partial \Delta[n]  \ar[r] \ar[d]
& X  \ar[d]^{p} \\
& \Delta[n]  \ar[r]^{v} 
& Y 
}$$
is equivalent to a lift in the diagram:
$$\xymatrix@1{
& \partial \Delta[n]  \ar[r]^{f} \ar[d]
& \Delta[n] \times F  \ar[d]^{\pi_1} \\
& \Delta[n]  \ar@{=}[r] 
& \Delta[n] 
}$$
A lift in this square is equivalent to an extension of the map $\pi_2f : \partial \Delta[n] \to F$ to $\Delta[n]$. By the hypotheses $F$ is a Kan complex with no non-trivial homotopy groups so that we conclude by the first step.\\
Now we drop the hypothesis that $p$ is locally trivial: by theorem \ref{minimalfibration} we factor $p = p'r$ as a minimal fibration followed by a retraction. $r$ can be chosen to have the RLP with respect to cofibrations, see theorem 3.5.9 in \cite{Hov}. Since $p'$ is a retract of $p$, its fibres are retracts of the fibres of $p$. Every minimal fibration is locally trivial by theorem \ref{minimalbundle}, hence we can apply the previous step.\\
Finally, we show that every trivial Kan fibration has non-empty fibres which have no non-trivial homotopy groups. Let $F$ be the fibre of $p$ over the vertex $v$. Thanks to theorem \ref{realisationSerre} and the exactness properties of the geometric realisation functor, we have that $|F|$ is the fibre of $|p|$ over the vertex $|v|$. Since $|p|$ is a weak equivalence, $|F|$ is non-empty and has no non-trivial homotopy groups. We conclude recalling that $\pi_n (X,v) \cong \pi_n(|X|, |v|)$.
\end{proof}

\nind
Hence we have obtained the following:

\begin{cor}
The category of simplicial sets admits a model structure where fibrations, cofibrations and weak equivalences defined as in the first chapter.
\end{cor}

\nind
We cannot end this section without a mention to the standard model structure on topological spaces. As we have anticipated in the first chapter simplicial sets are combinatorial models for nice topological spaces. Here we make this statement precise.

\begin{thm}
The category of topological spaces admits a model structure whose fibrations are Serre fibrations and weak equivalences are the usual weak equivalences.\\
Every object is trivially fibrant; moreover, the cofibrant objects are the retracts CW-complexes.
\end{thm}
\begin{proof}
See chapter 2 section 4 in \cite{Hov}.
\end{proof}

\begin{thm}[Milnor]
The geometric realisation $|\cdot |  : \textbf{SSet} \to \textbf{Ke}$ and the inclusion functor $i : \textbf{Ke} \to \textbf{Top}$ are Quillen equivalences.\\
Their composition gives a Quillen equivalence from the model category of simplicial sets to the one of topological spaces, with their standard model structures.
\end{thm}
\begin{proof}
See corollary 2.4.24 and theorem 3.6.7 in \cite{Hov}.
\end{proof}

We conclude the section with this result that will play an important role in the interpretation of type theory in simplicial sets.

\begin{thm}
The model structure on simplicial sets is proper.
\end{thm}
\begin{proof}
Since every object is cofibrant it is left proper thanks to theorem \ref{lemmaRP}.\\
For right properness just consider the diagram:
$$\xymatrix@1{
& X'  \ar[r]^{f'} \ar[d]^{w'}
& X  \ar[d]^{w} \\
& Y'  \ar[r]^{f} 
& Y 
}$$
where $w$ is a weak equivalence and $f$ a fibration. We apply the geometric realisation functor, thanks to theorem \ref{realisationSerre} and the definition of weak equivalence in simplicial sets it is enough to show that the model structure on topological spaces is right proper, but this is immediate, since every topological space is fibrant.
\end{proof}

\section*{Monoidal Model Categories}

When dealing with model categories we are mainly interested in their homotopy category, therefore is important to give definitions which descends reasonably in the homotopy category. For this reason we need to impose more than the usual structure on a model category to make a "good notion" of monoidal model category; the key ingredient turns out to be closedness of the monoidal category.

\begin{defn}
Given $\mathcal C$, $\mathcal D$ and $\mathcal E$ categories, an \textbf{adjunction of two variables}\index{adjunction of two variables} from $\mathcal C \times \mathcal D$ to $\mathcal E$ is a quintuple $(\otimes , \underline{\mathit{Hom}}_r, \underline{\mathit{Hom}}_l , \phi_r , \phi_l)$ where $\otimes : \mathcal C \times \mathcal D \to \mathcal E$, $\underline{\mathit{Hom}}_r : \mathcal D^{op} \times \mathcal E \to \mathcal C$ and $\underline{\mathit{Hom}}_l : \mathcal C^{op} \times \mathcal E \to \mathcal D$ are functors and $\phi_r$ and $\phi_l$ are natural isomorphisms:

$$\mathcal C( C , \underline{\mathit{Hom}}_r(D,E)) \mathop{\longrightarrow}^{\phi_r^{-1}} \mathcal E(C \otimes D , E) \mathop{\longrightarrow}^{\phi_l} \mathcal D (D , \underline{\mathit{Hom}}_l(C , E)) $$

\end{defn}

\begin{defn}
A \textbf{closed monoidal structure}\index{category!closed monoidal} on a category is an octuple $(\otimes , a , l , r, \underline{\mathit{Hom}}_r, \underline{\mathit{Hom}}_l , \phi_r , \phi_l)$ where $(\otimes , a , l , r)$ is a monoidal structure on $\mathcal C$ and $(\otimes , \underline{\mathit{Hom}}_r, \underline{\mathit{Hom}}_l , \phi_r , \phi_l) : \mathcal C \times \mathcal C \to \mathcal C$ an adjunction of two variables.
\end{defn}

\nind
We have built the definition of closed monoidal category from an adjunction of two variables; by analogy we want to define a convenient notion of Quillen adjunction of two variables. We motivate the following definition using the examples provided in the chapter on simplicial sets by the importance of the pushout product.

\begin{defn}
Given $\mathcal C$, $\mathcal D$ and $\mathcal E$ model categories, an adjunction of two variables $(\otimes , \underline{\mathit{Hom}}_r, \underline{\mathit{Hom}}_l , \phi_r , \phi_l) : \mathcal C \times \mathcal D \to \mathcal E$ is a \textbf{Quillen adjunction of two variables}\index{Quillen!adjunction of two variables} iff given a cofibration $f:U \to V$ in $\mathcal C$ and a cofibration $g: W \to X$ in $\mathcal D$, the induced map $f \Box g : P(f,g)= (V \otimes W) \coprod_{U \otimes W} (U \otimes X) \to V \otimes X$ is a cofibration in $\mathcal E$ which is trivial if either $f$ or $g$ is.\\
The left adjoint $\otimes$ of a Quillen adjunction of two variables is simply called \textbf{Quillen bifunctor}\index{Quillen!bifunctor}.\\
By extension of the already established terminology we call $f \Box g$\index{$\Box$} the \textbf{pushout product}\index{pushout product} of $f$ and $g$.
\end{defn}

\nind
The following lemma characterize the property of being a Quillen bifunctor for a given adjunction of two variables.

\begin{lem}\label{Quillenbif}
Let $\mathcal C$, $\mathcal D$ and $\mathcal E$ be model categories and $(\otimes , \mathit{Hom}_r, \mathit{Hom}_l , \phi_r , \phi_l) : \mathcal C \times \mathcal D \to \mathcal E$  an adjunction of two variables. Then the following are equivalent:
\begin{enumerate}[label=(\roman*)]
\item $\otimes$ is a Quillen bifunctor;

\item given a cofibration $g : W \to X$ in $\mathcal D$ and a fibration $p: Y \to Z$ in $\mathcal E$, the induced map:
$$\underline{\mathit{Hom}}_r(g,p) : \underline{\mathit{Hom}}_r(X , Y) \to \underline{\mathit{Hom}}_r(X , Z) \times_{\underline{\mathit{Hom}}_r(W, Z)} \underline{\mathit{Hom}}_r(W , Y) $$
\nind
is a fibration in $\mathcal C$ which is trivial if either $g$ or $p$ is;

\item given a cofibration $f:U \to V$ in $\mathcal C$ and a fibration $p: Y \to Z$ in $\mathcal E$, the induced map:
$$\underline{\mathit{Hom}}_l(f,p) : \underline{\mathit{Hom}}_l(V , Y) \to \underline{\mathit{Hom}}_l(V , Z) \times_{\underline{\mathit{Hom}}_l(U, Z)} \underline{\mathit{Hom}}_l(U , Y) $$
\nind
is a fibration in $\mathcal D$ which is trivial if either $f$ or $p$ is.
\end{enumerate}
\end{lem}
\begin{proof}
It is a standard argument using the definition of adjunction.
\end{proof}

Now we give the promised definition of monoidal model category, as before we impose a technical condition which is needed in order to obtain a well behaved homotopy category. We anticipate that this additional condition is automatic if the unit of the monoidal category is cofibrant, as it is in the main examples of simplicial sets and topological spaces.

\begin{defn}
A \textbf{monoidal model category}\index{model category!monoidal} is a closed category $\mathcal C$ which is also a model category such that the two following conditions hold:
\begin{enumerate}[label=(\roman*)]
\item the monoidal structure $\otimes : \mathcal C \times \mathcal C \to \mathcal C$ is a Quillen bifunctor;

\item let $q: QS \to S$ be the cofibrant replacement for the unit $S$ of the monoidal structure, then the standard map $q \otimes 1 : QS \otimes X \to S \otimes X$ is a weak equivalence for every cofibrant $X$. Similarly the standard map $1 \otimes q: X \otimes QS \to X \otimes S$ is a weak equivalence for every cofibrant $X$.
\end{enumerate}
\end{defn}

\begin{defn}
A \textbf{symmetric monoidal model category}\index{model category!symmetric monoidal} is just a symmetric monoidal category and a model category whose monoidal and model structures satisfy the compatibility conditions of the above definition.
\end{defn}

\begin{thm}
The standard model structure on simplicial sets forms a symmetric monoidal model category.
\end{thm}
\begin{proof}
The symmetric monoidal structure on $\textbf{SSet}$ is the one given by the cartesian product, the adjoint being the function complex $\underline{\mathit{Hom}}(X , Y)$. All objects are cofibrant, hence we just need to check that the product is a Quillen bifunctor. Thanks to the previous lemma and to theorem \ref{homfib} we can conclude.
\end{proof}

\chapter{Type Theory}
\lettrine{I}{n} mathematics there are some questions that we don't want to answer, for example: "is $1 \in 3$?" or "can a finite simple group be a zero of the Riemann zeta function?". These questions arise from the nature of material set theory used as implicit foundation for everyday mathematics like ZF where all mathematical entities are sets and a global membership predicate is available. It is not important to know if $1 \in 3$, because the answer depends on the particular set-theoretic construction used to build the model of $\mathbb N$, whereas we are interested in the structural properties of natural numbers, i.e. the ones expressible from the constant $0$ and the successor function. Similarly, is common in mathematical papers and books to write "let $G$ be a finite group and $z$ a complex number", meaning that the two variables range over fixed domains which cannot conflict. Moreover, this is exactly what is done in computer programming when variables are declared.\\
Types are an answer for the need to a formalism adherent to informal reasoning. They can be conceived as intensional sets, meaning that they are characterised by a global, holistic, essence or nature instead of being reduced to their elements.
Type theory has three main actors: terms, types and contexts. Every term \textit{inhabits}\index{inhabitation of a type} types, written $a:A$, whereas \textit{contexts}\index{context} are finite ordered sets of inhabitation judgements like $\Gamma = (  x_1 : A_1 , \dots , x_n : A_n)$ such that for each $k = 1, \dots , n$ the expression $x_k : A_k$ can be inferred from the previous ones. \\
Type theories differ mainly for the type constructors which are admitted and by the rules, more or less powerful, that are used to manage these constructors. Examples of common type constructors are function and product types: given two types $A$ and $B$ we can form the type of functions $A \to B$ and the type of paris $A \times B$. We shall discuss in detail the constructors allowed in Martin-L\"of type theory.\\
As we have mentioned in the introduction a key feature of many type theories is the \textit{Curry-Howard isomorphism} also called \textit{propositions-as-types paradigm}, its starting point is an analogy between the syntactical rules used to define type-theoretic constructors and the ones used in natural deduction for intuitionistic logic. This yields to the interpretation of a proposition as the type of its proofs. 

We cannot omit to mention the relationship between type theory and computer science. The first and most basic example is the one given by the \textit{lambda calculus}\index{lambda-calculus} used by Church and his school to formalise the notion of computable function. Moreover, the Curry-Howard correspondence can be further extended to programs and games. Indeed, program specifications can be interpreted as types and programs verifying the specification as terms inside the type. Similarly, games can be interpreted as the types of their winning strategies.\\
As we shall see in the next section the Curry-Howard isomorphism can be seen as a type-theoretic implementation of the Brouwer-Heyting-Kolmogorov interpretation of intuitionistic logic.\\
For an example of more general type theories, without the proposition-as-types correspondence see \cite{AG}.\\
In the following table we sketch the most significant articulations of the correspondence:

\begin{center}
\begin{tabular}{|c|c|}
\hline
logic & type theory \\
\hline
\hline
formula & type \\
\hline
proof & term\\
\hline
type constructors & connectives\\
\hline
implication & function type \\
\hline
conjunction & product type \\
\hline
disjunction & sum type\\
\hline
absurdity & empty type\\
\hline
existential quantifier & dependent sum\\
\hline
universal quantifier & dependent product\\
\hline
equality & identity type\\
\hline
normal term & normal proof\\
\hline
\end{tabular}
\end{center}

\nind
Negation is defined as usual in intuitionistic logic as $\phi \Rightarrow \bot$.

The Curry-Howard isomorphism allows a subtle and important distinction: the one between propositions and judgements. In natural language, propositions are syntactical constructions whereas judgements are acts of speech. $\phi$ is a proposition, "$\phi$ is true" is a judgement. Observe that the statement that a proposition is well-formed like "$\phi$ is a proposition" is itself  a kind of judgement, this is one of the main sources of confusion regarding this distinction.\\
Under the Curry-Howard correspondence propositions are types, whereas judgements are treated formally by a system of rules. There are six kinds of judgements, two for each of the following: contexts, types and term.
\begin{itemize}
\item $\vdash \Gamma \; \mbox{context}$

\item $\vdash \Gamma = \Delta \; \mbox{context}$

\item $\Gamma \vdash A \; \mbox{type}$

\item $\Gamma \vdash A=B \; \mbox{type}$

\item $\Gamma  \vdash a:A$

\item $\Gamma \vdash a=b : A$
\end{itemize}

\nind
Judgements are expressed by the symbol $\vdash$, meaning that the expression on the right can be inferred from the context on the left. Sometimes the symbol $\vdash$ is dropped when the empty context appears.\\
We give now as an example the rules for contexts, postponing the discussion of the rules governing syntactic operations (the so-called \textit{structural rules}\index{structural rules}) to the sequel, when the main ideas of Martin-L\"of type theory will be already fixed. Contexts are given by two rules (where the first is a rule with empty premiss i.e. an axiom):

\begin{prooftree}
\AxiomC{}
\UnaryInfC{$\vdash () \;  \mbox{context}$}
\end{prooftree}

\begin{prooftree}
\AxiomC{$\Gamma \vdash A \; \mbox{type}$}
\UnaryInfC{$\vdash (\Gamma , x:A) \; \mbox{context}$}
\end{prooftree}

\nind
Which express the fact that contexts are inductively constructed finite list of judgements of the form $x:A$. Notice again the similarity with computer programming.

In Martin-L\"of theory the distinction between propositions and judgements induces an analogous distinction between \textit{propositional equality}\index{equality!propositional} and \textit{definitional}\index{equality!definitional} or \textit{judgemental equality}\index{equality!judgemental}. In fact the former is the one present in the basic judgements, whereas the latter is a type constructor written $Id_A(a,b)$ for $a,b : A$, motivated by the Curry-Howard isomorphism.\\
We cannot ask if a judgemental equality is true or not because it is not a proposition, although we may derive it from other judgements. On the other hand, we can derive the judgement $p:Id_A(a,b)$ asserting that there is a proof of the proposition "$a$ and $b$ are equal in $A$" so that we know that the proposition $Id_A(a,b)$ is true.\\
In fact Martin-L\"of type theory can be seen as a formalisation of the meta-theory of judgements of an object theory which deals with propositions (i.e. types).\\
We will use the symbol $=$ for generic definitional equalities and the specific $:=$ when we will introduce new symbols inside type theory, thus we should distinguish between them, the latter would deserve the name definitional equality and the former judgemental equality, but we will not insist on this point.\\
We emphasize that the two symbols: $\vdash$ and the long horizontal line used for syntactical rules have very different meanings, in spite of the apparent similarity, in fact they both express some kind of deduction. But the first is a formal symbol of the theory used to express judgements, whereas the second is a meta-theoretical rule used to define the meaning of the former.

Finally, we introduce \textit{dependent types}\index{dependent!types}, that are families of types parametrized by another type. Type dependencies are expressed formally as particular judgements, like: $x:A \vdash b(x) : C(x)$ and they are interpreted as hypothetical judgements i.e. judgements made under assumptions. Type dependencies are familiar in computer programming, for example the type of strings depends on the type of natural numbers; indeed, for every natural number $n$ we have the well-formed type of strings of length $n$. As we shall see in the next chapter type dependencies are interpreted homotopically as fibrations.

The sources for this chapter are \cite{ML1}, \cite{ML2} and \cite{War} for Martin-L\"of type theory and \cite{KLV},\cite{UFP} for homotopy type theory.\\
A good introductory book on type theory and the Curry-Howard isomorphism is \cite{SU}.

\section*{Martin-L\"of Type Theory}

Martin-L\"of type theory can be synthetically described as a dependent type theory with products, sums, types of well-founded trees, identity types, a type of natural numbers, for the standard finite sets and for one universe. Martin-L\"of proposed two versions of his theory, one with extensional identity types and another with intensional ones. In the sequel Martin-L\"of type theory will always mean the version with intensional identity types.\\
Thanks to the Curry-Howard correspondence, these constructors are defined by syntactical rules following a pattern similar to the one used for natural deduction, every type constructor has four rules defining it; namely: formation, introduction, elimination and computation.
Let start with \textbf{dependent products}\index{dependent!product}. The formation rule tells us when the type constructor is well defined.

\begin{prooftree}
\AxiomC{$A \; \mbox{type}$}
\AxiomC{$x:A \vdash B(x) \; \mbox{type}$}
\BinaryInfC{$(\Pi x:A )B(x) \; \mbox{type}$}
\end{prooftree}

\nind
The introduction rules specifies the so called \textit{canonical terms}\index{canonical term} of the type constructor:

\begin{prooftree}
\AxiomC{$x:A \vdash b(x):B(x)$}
\UnaryInfC{$\lambda x.b(x) : (\Pi x:A)B(x)$}
\end{prooftree}
\index{$\lambda x.b(x)$}

\nind
Notice that we should add in the premiss also judgements for the well definition of the types involved as $A \; \mathit{type}$, but for a better readability we will leave these judgements implicit.\\
The introduction and formation rules can be seen as rules for the well definition of the canonical terms. Their meaning is explained by the elimination and the computations rules.

\begin{prooftree}
\AxiomC{$a:A$}
\AxiomC{$c : (\Pi x:A)B(x)$}
\BinaryInfC{$\textsf{ap}(c,a) : B(a)$}
\end{prooftree}
\index{$\textsf{ap}$}

\nind
And the computation rule which tells us what happens if we introduce and then eliminate a canonical term, they are analogous to the metarules in natural deduction that express how we can convert detours given by an introduction followed by an elimination. Note that the "computation" is truly a syntactical computation thanks to the definitional equality that appears in the conclusion.

\begin{prooftree}
\AxiomC{$a:A$}
\AxiomC{$x:A \vdash b(x):B(x)$}
\BinaryInfC{$\textsf{ap}(\lambda x.b(x),a)=b(a) : B(a)$}
\end{prooftree}

\nind
Observe that thanks to the correspondence with natural deduction we would like to read the defintional equality above from the left to the right, otherwise we would introduce a detour instead of elimnating it.\\
We can now say, after having inspected these rules, that the canonical terms of the dependent product deserve the name of functions, which are obtained "abstracting" the variable $x$ from the dependent expression $b(x)$. Furthermore, the eliminator $\textsf{ap}$ has the meaning of the application of a function to an element.\\
\textbf{Function types}\index{function type} are defined to be products of a family of types not depending over the base type i.e. $A \to B := (\Pi x:A)B$.

Now we move to \textbf{dependent sums}\index{dependent!sum}, sums are types of pairs. As usual we start with the formation and the introduction rules:

\begin{prooftree}
\AxiomC{$A \; \mbox{type}$}
\AxiomC{$x:A \vdash B(x) \; \mbox{type}$}
\BinaryInfC{$(\Sigma x:A )B(x) \; \mbox{type}$}
\end{prooftree}

\begin{prooftree}
\AxiomC{$a:A$}
\AxiomC{$b:B(a)$}
\BinaryInfC{$(a,b) : (\Sigma x:A )B(x)$}
\end{prooftree}

\nind
The elimination rule says that if we know how to prove a proposition $C$ for pairs, then we can prove it for any term in the sum type:

\begin{prooftree}
\AxiomC{$x:A , y:B(x) \vdash d(x,y) : C((x,y))$}
\AxiomC{$c:(\Sigma x:A )B(x)$}
\BinaryInfC{$\textsf{E}(c , d(x,y)) : C(c)$}
\end{prooftree}

\nind
Sometimes in the sequel we will write $\textsf{E}(c , (x,y)d(x,y))$ if the term $d$ will have more free variables. And finally the computation rule:

\begin{prooftree}
\AxiomC{$a:A$}
\AxiomC{$b:B(a)$}
\AxiomC{$x:A , y:B(x) \vdash d(x,y) : C((x,y))$}
\TrinaryInfC{$\textsf{E}((a,b) , d(x,y)) = d(a,b) : C(c)$}
\end{prooftree}

\nind
In some cases - like for the dependent sum - the computation rules can be seen as recursion principles, saying that in order to define the eliminator it suffices to know how to define it on canonical terms. Then, morally speaking, a type contains only canonical terms. Intuitively, we can interpret this feature recalling the constructive and computational content of type theory; indeed, we can conceive generic terms as the result of iterated constructions based on canonical terms, we know in principle that every term is built on canonical ones, but we don't know what are the constituents and how to effectively build it.\\
Note that a normalisation theorem i.e. the metatheorem asserting that every term is definitionally equal to a canonical one and that these canonical forms are well defined, is highly nontrivial and has consistency as one of its consequences.\\
Again, consider the constant case: if the family of types is not dependent over the base, we define the \textbf{product type}\index{product type} as $A \times B := (\Sigma x : A)B$.\\
Ordinary products are equipped with projections; this is the case, in fact we can define the \textbf{left projection}\index{left projection} as $\textsf{p}(c) := \textsf{E}(c, (x,y)x)$\index{$\textsf{p}$}. It is straightforward to obtain the derived rules:

\begin{prooftree}
\AxiomC{$c: (\Sigma x:A)B(x)$}
\UnaryInfC{$\textsf{p}(c) : A$}
\end{prooftree}

\begin{prooftree}
\AxiomC{$a:A$}
\AxiomC{$b:B(a)$}
\BinaryInfC{$\textsf{p}((a,b)) = a : A$}
\end{prooftree}

\nind
Furthermore, we can define a \textbf{right projection}\index{right projection} as $\textsf{q}(c) := \textsf{E}(c, (x,y)y)$\index{$\textsf{q}$}. Similarly, we obtain the two derived rules for the right projection:

\begin{prooftree}
\AxiomC{$c: (\Sigma x:A)B(x)$}
\UnaryInfC{$\textsf{q}(c) : B(\textsf{p}(c))$}
\end{prooftree}

\begin{prooftree}
\AxiomC{$a:A$}
\AxiomC{$b:B(a)$}
\BinaryInfC{$\textsf{q}((a,b)) = b : B(a)$}
\end{prooftree}

\nind
In addition to the logical interpretation given by the Curry-Howard isomorphism, sums allow to interpret the constructive notion of \textit{such that}\index{such that}: for example a construction of a Cauchy sequence of rational numbers is a pair, the first component being a function $s: \mathbb N \to \mathbb Q$, and the second a proof that $s$ is Cauchy. Thanks to this interpretation the dependent sum $(\Sigma x:A)B(x)$ can also be thought as the set of all $a$ in $A$ such that the proposition $B(a)$ holds, so that the rules for the sum play the role of the set-theoretic separation axiom.\\
Then we present the rules for the \textbf{identity types}\index{identity types}:

\begin{prooftree}
\AxiomC{$A \; \mbox{type}$}
\AxiomC{$a,b : A$}
\BinaryInfC{$Id_A(a,b) \; \mbox{type}$}
\end{prooftree}

\begin{prooftree}
\AxiomC{$A \; \mbox{type}$}
\AxiomC{$a : A$}
\BinaryInfC{$\textsf{refl}_A(a) : Id_A(a,a)$}
\end{prooftree}
\index{$\textsf{refl}$}

\nind
The introduction rule is stated with just one term, but allowing two definitionally equal term it can be seen as a "reflection rule" giving us a propositional equality every time we have a definitional one.\\
The elimination and the computation rules need a little explanation in order to become readable, we can explain them as follows: if we have a property $C$ depending on a couple of terms in $A$ and on a proof of the identity between these two, and we know how to prove this property for definitionally equal terms, then we know how to prove it for every couple of propositionally equal terms. So it can be thought as an instance of the principle of \textit{indiscernibility of identicals}\index{indiscernibility of identicals}:

\begin{prooftree}
\AxiomC{$x,y : A, u:Id_A(x,y) \vdash C(x,y,u) \; \mbox{type}$}
\AxiomC{$z: A \vdash d(z) : C(z,z, \textsf{refl}_A(z))$}
\BinaryInfC{$\textsf{J}_{z,d}(x,y,u) : C(x,y,u)$}
\end{prooftree}
\index{$\textsf{J}$}

\begin{prooftree}
\AxiomC{$x,y : A, u:Id_A(x,y) \vdash C(x,y,u) \; \mbox{type}$}
\AxiomC{$z: A \vdash d(z) : C(z,z, \textsf{refl}_A(z))$}
\BinaryInfC{$\textsf{J}_{z,d}(x,x,\textsf{refl}_A(x)) = d(x) : C(x,x,\textsf{refl}_A(x))$}
\end{prooftree}

\nind
It is common to find the expression \textit{Leibniz's law}\index{Leibniz's law} used sometimes for the principle of indiscernibility of identicals and other times for the principle of identity of indiscernibles, which is the converse. The latter is much more controversial and was one of the cornerstones of Leibniz's metaphysics, thus would be better to confine to it the use of this expression. Anyway, in order to be clear we will avoid the expression "Leibniz's law".\\
The \textbf{disjoint sum}\index{disjoint sum} of two types could be introduced as a dependent sum over the type of Booleans $\textbf N_2$, which we will introduce later. For a clearer understanding we present now the rules of this type:

\begin{prooftree}
\AxiomC{$A \; \mbox{type}$}
\AxiomC{$B \; \mbox{type}$}
\BinaryInfC{$A+B \; \mbox{type}$}
\end{prooftree}

\nind
For this type constructor we have two introduction rules:

\begin{prooftree}
\AxiomC{$a:A$}
\UnaryInfC{$\textsf{i}(a):A+B$}
\end{prooftree}
\index{$\textsf{i}$}

\begin{prooftree}
\AxiomC{$b:B$}
\UnaryInfC{$\textsf{j}(b):A+B$}
\end{prooftree}
\index{$\textsf{j}$}

\nind
And the usual elimination and computation rules, explaining that proofs about a disjoint sum can be performed by cases:

\begin{prooftree}
\AxiomC{$x:A \vdash d(x) : C(\textsf{i}(x))$}
\AxiomC{$y:B \vdash e(y) : C(\textsf{j}(y))$}
\AxiomC{$c:A+B$}
\TrinaryInfC{$\textsf{D}(c, d(x), e(y) ):C(c)$}
\end{prooftree}
\index{$\textsf{D}$}

\begin{prooftree}
\AxiomC{$a:A$}
\AxiomC{$x:A \vdash d(x) : C(\textsf{i}(x))$}
\AxiomC{$y:B \vdash e(y) : C(\textsf{j}(y))$}
\TrinaryInfC{$\textsf{D}(\textsf{i}(a), d(x), e(y) ) = d(a):C(c)$}
\end{prooftree}

\begin{prooftree}
\AxiomC{$b:B$}
\AxiomC{$x:A \vdash d(x) : C(\textsf{i}(x))$}
\AxiomC{$y:B \vdash e(y) : C(\textsf{j}(y))$}
\TrinaryInfC{$\textsl{D}(\textsf{j}(b), d(x), e(y) ) = e(b):C(c)$}
\end{prooftree}

\nind
The \textbf{types of finite sets}\index{type!of finite sets} are given by the following rules. Notice that we could take as primitive just $\textbf N_0$, $\textbf N_1$ and $\textbf N_2$, defining the others as iterated disjoint unions.

\begin{prooftree}
\AxiomC{}
\UnaryInfC{$\textbf N_k \; \mbox{type}$}
\end{prooftree}
\index{$\textbf N_k$}

\begin{prooftree}
\AxiomC{}
\UnaryInfC{$m_k : \textbf N_k \quad (m=0, \dots , k-1)$}
\end{prooftree}

\nind
The elimination and computation rules tell us that in order to prove a proposition for the terms of a finite set we can reason by cases checking separately each term.

\begin{prooftree}
\AxiomC{$c:\textbf N_k$}
\AxiomC{$c_m : C(m_k) \quad (m=0, \dots , k-1)$}
\BinaryInfC{$\textsf{R}_k(c,c_0, \dots, c_{k-1}) : C(c)$}
\end{prooftree}
\index{$\textsf{R}_k$}

\begin{prooftree}
\AxiomC{$c_m : C(m_k) \quad (m=0, \dots , k-1)$}
\UnaryInfC{$\textsf{R}_k(m_n,c_0, \dots, c_{k-1}) = c_m : C(m_k)$}
\end{prooftree}

\nind
Notice that $\textbf N_0$ has no introduction rule, therefore no terms.

Let move to the \textbf{type of natural numbers}\index{type!of natural numbers} again we could define them using $W$-types and the type of Booleans, but a direct treatment is clearer for an introduction and allows in turn to use natural numbers in order to understand $W$-types.

\begin{prooftree}
\AxiomC{}
\UnaryInfC{$\textbf N \; \mbox{type}$}
\end{prooftree}
\index{$\textbf N$}

\nind
We have two introduction rules:

\begin{prooftree}
\AxiomC{}
\UnaryInfC{$0:\textbf N$}
\end{prooftree}

\begin{prooftree}
\AxiomC{$n:\textbf N$}
\UnaryInfC{$s(n):\textbf N$}
\end{prooftree}

\nind
The other rules encapsulate proofs by mathematical induction defining proofs by recursion. The last two are the computational ones:

\begin{prooftree}
\AxiomC{$d:C(0)$}
\AxiomC{$x:\textbf N, y:C(x) \vdash e(x,y): C(s(x))$}
\AxiomC{$c:\textbf N$}
\TrinaryInfC{$\textsf{R}(c,d,e(x,y)):C(c)$}
\end{prooftree}
\index{$\textsf{R}$}

\begin{prooftree}
\AxiomC{$d:C(0)$}
\AxiomC{$x: \textbf N, y:C(x) \vdash e(x,y): C(s(x))$}
\BinaryInfC{$\textsf{R}(0,d,e(x,y)) = d:C(0)$}
\end{prooftree}

\begin{prooftree}
\AxiomC{$a: \textbf N$}
\AxiomC{$d:C(0)$}
\AxiomC{$x:\textbf N, y:C(x) \vdash e(x,y): C(s(x))$}
\TrinaryInfC{$\textsf{R}(s(a),d,e(x,y)) = e(a, \textsf{R}(a,d,e(x,y))):C(s(a))$}
\end{prooftree}

\nind
Finally, we present \textbf{$W$-types}\index{type!of well-founded trees}, that are a generalisation of types of natural numbers, lists and binary trees. For every dependent type we can form the appropriate $W$-type, written $(W x:A)B(x)$, its terms can be thought as well-founded trees. The base type has the meaning of the type of labels for the nodes, and given a label for a node $a:A$, we have that $B(a)$ contains names for the branches of the tree at the node $a$. They can be represented by a function $b:B(a) \to (W x:A)B(x)$ so that $b(v)$ for $v:B(a)$ is the subtree obtained from the $v$-th branch of the node labelled by $a$.\\
Moreover, intuitionistic ordinals are no longer totally ordered, hence they can be represented as well-founded trees, and recursion over a $W$-type is the type-theoretic analogue of transfinite recursion over ordinals.\\
As usual the first two rules are simple:

\begin{prooftree}
\AxiomC{$A \; \mbox{type}$}
\AxiomC{$x:A \vdash B(x) \; \mbox{type}$}
\BinaryInfC{$(W x:A)B(x) \; \mbox{type}$}
\end{prooftree}

\begin{prooftree}
\AxiomC{$a:A$}
\AxiomC{$b:B(a) \to (W x:A)B(x)$}
\BinaryInfC{$\textsf{sup}(a,b) : (W x:A)B(x)$}
\end{prooftree}
\index{$\textsf{sup}$}

\nind
We may think of $\textsf{sup}(a,b)$ as the supremum i.e. the least ordinal greater than all the ordinals $b(v)$ for $v$ in $B(a)$. \\
We would like to have also a bottom term $0$ as starting points for recursion, but we can obtain it by taking one of the types $B(x)$ to be $\textbf N_0$. Indeed, if $a_0 : A$ is such that $B(a_0) = \textbf N_0$, then $\textsf{R}_0(y) : (W x:A)B(x)$ given any $y:B(a_0)$, so that by $\lambda$-abstraction we can form a function in $B(a_0) \to (W x:A)B(x)$ and by the above introduction rule $\textsf{sup}(a_0 ,\lambda y. \textsf{R}_0(y)) : (W x:A)B(x)$.\\
The elimination rule has the meaning of transfinite recursion or induction. Recall that $\textsf{ap}(b,v)$ is the eliminator term for dependent products, hence it can be conceived as $b(v)$ i.e. the application of the function $b$ to the term $v$.\\
Notice that the first long premiss in the elimination rule says that $C$ holds for $\textsf{sup}(a,b)$ if it holds for all predecessors $\textsf{ap}(b,v)$, i.e. $C$ is an inductive property:

\begin{prooftree}
\alwaysNoLine
\AxiomC{$c: (W x:A)B(x)$}
\UnaryInfC{$x:A, y:B(x) \to (W x:A)B(x) , z: (\Pi v:B(x))C(\textsf{ap}(y,v)) \vdash d(x,y,z) : C(\textsf{sup}(x,y))$}
\alwaysSingleLine
\UnaryInfC{$\textsf{T}(c , d(x,y,z)) : C(c)$}
\end{prooftree}
\index{$\textsf{T}$}

\nind
The conclusion of the computation rule, as in the case of the type of natural numbers, contains the expression for defining functions by recursion.

\begin{prooftree}
\AxiomC{$a:A$}
\AxiomC{$b:B(a) \to (W x:A)B(x) \quad $}
\alwaysNoLine
\BinaryInfC{$x:A, y:B(x) \to (W x:A)B(x) , z: (\Pi v:B(x))C(\textsf{ap}(y,v)) \vdash d(x,y,z) : C(\textsf{sup}(x,y))$}
\alwaysSingleLine
\UnaryInfC{$\textsf{T}(\textsf{sup}(a,b) , d(x,y,z)) = d(a,b, \lambda v. \textsf{T}(\textsf{ap}(b,v) , d(x,y,z))): C(\textsf{sup}(a,b))$}
\end{prooftree}

\nind
In order to fix the ideas, we sketch the construction of the type of natural numbers as a $W$-type. The type of labels is $A := \textbf N_2$ because natural numbers are zero or the successor of another number. The labels are $B(0_{\textbf 2}) := \textbf N_0$ because zero is a constant so that it has arity $0$, and $B(1_{\textbf 2}) := \textbf N_1$ because the successor function is a function with arity $1$. Hence we can define $\textbf N := (W x:\textbf N_2) B(x)$, it is straightforward to check that the rules for the type of natural numbers are the rules of this $W$-type.

Universes and their rules are analogous to the replacement axiom in classical ZF; moreover, we expect to be able to form the type of all propositions and by the Curry-Howard isomorphism we are induced to consider a "type of types". The first formulation of Martin-L\"of type theory included a literal formulation of this need, i.e. a type of all types that is a type $U$ such that for any type $A$ we have $A:U$, included $U:U$. This formulation was recognized to be inconsistent with the discovery of \textit{Girard's paradox}\index{Girard's paradox} which is a type-theoretic version of the standard set-theoretic paradoxes.\\
Therefore, other rules are used and a universe is conceived as a type of small types, with reflection rules that allow to perform the type constructions inside it. We discuss now the two most important kinds of universes: à la Russell and à la Tarski. 
The first takes its name because of their similarity with the ones used in the Principia Mathematica, the second for its resemblance with Tarski's definition of truth in first order logic.\\
\textbf{Russell style universes}\index{universe!Russell style} are types of small types, with the following formation rules:

\begin{prooftree}
\AxiomC{}
\UnaryInfC{$U \; \mbox{type}$}
\end{prooftree}

\begin{prooftree}
\AxiomC{$A:U$}
\UnaryInfC{$A \; \mbox{type}$}
\end{prooftree}

\nind
With introduction rules, asserting that the universe is closed under the type constructors:

\begin{prooftree}
\AxiomC{$A:U$}
\AxiomC{$x:A \vdash B(x) : U$}
\BinaryInfC{$(\Pi x:A)B(x) : U$}
\end{prooftree}

\nind
And similarly for the other type constructors.\\
One remark is now needed: usually the universe has no elimination rules, the main reason for leaving the universe without elimination rules is that we want universes to be \textit{open}, in the sense that we may want to add other type constructors in the future without changing the universe. Moreover, the constructive content of the universe would be lost adding elimination rules for it; indeed, a universe is better conceived as the part of the ideal mathematical world that we have explored, constructed, observed, conceived with our intuition and so on depending on the flavour of intuitionism that we want to adopt, and therefore is necessarily partial and open to further constructions or explorations.\\
Russell style universes have the defect of breaking the clear distinction between terms and types, which was one of the motivations we used to introduce type theory. Instead we prefer the slightly more involved, but conceptually much deeper \textbf{Tarski style}\index{universe!Tarski style} ones. They are types of names for small types and they are given with a constructor function called $El$, suggesting that the constructed type is the one of elements with the given name. The formation rules are:

\begin{prooftree}
\AxiomC{}
\UnaryInfC{$U \; \mbox{type}$}
\end{prooftree}

\begin{prooftree}
\AxiomC{$a:U$}
\UnaryInfC{$El(a) \, \mbox{type}$}
\end{prooftree}
\index{$El$}

\nind
The introduction rules come with coherence rules asserting that the names for the type constructors which live inside the universe, correspond under the function $El$, to the actual external type constructors:

\begin{prooftree}
\AxiomC{$a:U$}
\AxiomC{$x:El(a) \vdash b(x):U$}
\BinaryInfC{$\pi(a,b(x)) : U$}
\end{prooftree}

\begin{prooftree}
\AxiomC{$a:U$}
\AxiomC{$x:El(a) \vdash b(x):U$}
\BinaryInfC{$\pi(a,b(x)) = (\Pi x:El(a))El(b(x))$}
\end{prooftree}

\nind
The rules for the reflection of identity types:

\begin{prooftree}
\AxiomC{$a:U$}
\AxiomC{$b:El(a)$}
\AxiomC{$c:El(a)$}
\TrinaryInfC{$i_a(b,c) : U$}
\end{prooftree}

\begin{prooftree}
\AxiomC{$a:U$}
\AxiomC{$b:El(a)$}
\AxiomC{$c:El(a)$}
\TrinaryInfC{$El(i_a(b,c)) = Id_{El(a)}(b,c)$}
\end{prooftree}

\nind
The rules for the reflection of the type of natural numbers:

\begin{prooftree}
\AxiomC{}
\UnaryInfC{$n : U$}
\end{prooftree}

\begin{prooftree}
\AxiomC{}
\UnaryInfC{$El(n) = \textbf N$}
\end{prooftree}

\nind
And similarly for the other type constructors.

Finally, now that the presentation of the type constructor is sufficient to grasp the idea of Martin-L\"of type theory, we present the remaining rules i.e. the structural ones.\\
Recall that the first two structural rules for contexts were already presented at the beginning of the section. 
The following is the \textbf{variable declaration}\index{variable declaration} rule:

\begin{prooftree}
\AxiomC{$x:A$}
\AxiomC{$\Gamma \; \mbox{context}$}
\BinaryInfC{$\Gamma , x:A \vdash x:A$}
\end{prooftree}

\nind
Let $\mbox{J}$ be a basic judgement, then we have the \textbf{weakening rule}\index{weakening rule}, which can be thought as a form of monotonicity for the judgement relation $\vdash$, it means that, if a judgement $\mbox{J}$ can be inferred in a certain context, it can be inferred as well adding more information to the context:

\begin{prooftree}
\AxiomC{$A \; \mbox{type}$}
\AxiomC{$\Gamma \vdash \mbox{J}$}
\BinaryInfC{$\Gamma , x:A \vdash \mbox{J}$}
\end{prooftree}

\nind
We cannot omit to say that the variable declaration and the weakening rule are deeper than their appearance, in fact they are linked to an idealised conception of information, so that thanks to these rules informations can be freely duplicated and deleted. Logic and type theory can be generalised without it and this yields to the world of linear logic.

Now we present the \textbf{substitution rule}\index{substitution rule}:

\begin{prooftree}
\AxiomC{$a:A$}
\AxiomC{$\Gamma , x:A \vdash \mbox{J}$}
\BinaryInfC{$\Gamma[a/x] \vdash \mbox{J}[a/x]$}
\end{prooftree}

\nind
We finish with the \textbf{rules for the definitional equality}\index{rules for the definitional equality}, we have reflexivity, symmetry and transitivity for definitionally equal terms and types respectively:

\begin{prooftree}
\AxiomC{$A \; \mbox{type}$}
\UnaryInfC{$A=A \; \mbox{type}$}
\end{prooftree}

\begin{prooftree}
\AxiomC{$A=B \; \mbox{type}$}
\UnaryInfC{$B=A \; \mbox{type}$}
\end{prooftree}

\begin{prooftree}
\AxiomC{$A=B \; \mbox{type}$}
\AxiomC{$B=C \; \mbox{type}$}
\BinaryInfC{$A=C \; \mbox{type}$}
\end{prooftree}

\begin{prooftree}
\AxiomC{$a:A$}
\UnaryInfC{$a=a:A$}
\end{prooftree}

\begin{prooftree}
\AxiomC{$a=b:A$}
\UnaryInfC{$b=a:A$}
\end{prooftree}

\begin{prooftree}
\AxiomC{$a=b:A$}
\AxiomC{$b=c:A$}
\BinaryInfC{$a=c:A$}
\end{prooftree}

\nind
And the last two rules expressing the stability of judgements $a:A$ or $a=b:A$ under under the change of definitionally equal types:

\begin{prooftree}
\AxiomC{$a:A$}
\AxiomC{$A=B \; \mbox{type}$}
\BinaryInfC{$a:B$}
\end{prooftree}

\begin{prooftree}
\AxiomC{$a=b:A$}
\AxiomC{$A=B \; \mbox{type}$}
\BinaryInfC{$a=b:B$}
\end{prooftree}

Perhaps surprisingly, Martin-L\"of theory validates a type-theoretic version of the axiom of choice.\\
In standard set-theoretic foundations the axiom of choice has been criticized as a prototypical example of nonconstructive reasoning. Nevertheless, some kind of choice principle is already present in the Brouwer-Heyting-Kolmogorov interpretation of intuitionistic logic where a construction of a universal quantifier is interpreted as a function and one of the existential quantifier as a couple of constructions.

\begin{thm}[type-theoretic AC]\label{choice}
The following type is inhabited:
$$(\forall x:A)(\exists y:B(x))C(x,y) \to \sum_{ f: (\Pi x:A)B(x))}(\forall x:A)C(x, \textsf{ap}(f,x)) $$
where the quantifiers are interpreted with the appropriate type constructor following the Curry-Howard correspondence.
\end{thm}

\begin{proof}
We first give an informal argument which shall be translated in formal derivations in type theory. Suppose that the antecedent holds, hence we have a method to transform each x in proof of $(\exists y)C(x,y)$ i.e. a pair of a term $y$ and a proof of the proposition $C(x,y)$. Let $f$ be a method to get from every $x$ the first component of the pair, hence $C(x,f(x))$ holds.\\
Now the formal translation: assume $z:(\Pi x:A)(\Sigma y:B(x))C(x,y) $, if $x$ is an arbitrary term of type $A$, then by the elimination rule for products we get $\textsf{ap}(z,x) : (\Sigma y:B(x))C(x,y)$, we write $z(x)$ for the application.\\
Now apply the two projections to obtain $\textsf{p}(z(x)) : B(x)$ and $\textsf{q}(z(x)) : C(x, \textsf{p}(z(x)))$. Then we discharge the assumption $x:A$ by a $\lambda$-abstraction on $x$ so that $\lambda x. \textsf{p}(z(x)) : (\Pi x:A)B(x)$.\\
By the computation rule for products we have $\textsf{ap}(\lambda x. \textsf{p}(z(x)) , x) = \textsf{p}(z(x))$, hence by substitution $C(x , \textsf{ap}(\lambda x. \textsf{p}(z(x)) , x) )  = C \big( \, x, \, \textsf{p}(z(x)) \, \big)$. By this equality we have that the term $\textsf{q}(z(x))$ inhabits the type $C(x , \textsf{ap}(\lambda x. \textsf{p}(z(x)) , x) )$.\\
Finally, we use another $\lambda$-abstraction on $x$ to get $\lambda x. \textsf{q}(z(x))$ in the appropriate type and by the introduction rule for sum we have the desired term $( \lambda x. \textsf{p}(z(x) ) ,  \lambda x. \textsf{q}(z(x)) ): \sum_{ f: (\Pi x:A)B(x)}(\Pi x:A)C(x, f(x))$.
\end{proof}

\section*{Homotopy Type Theory}
In this section we present and discuss the three further rules that we need to add to Martin-L\"of type theory in order to justify a honest homotopical interpretation. We also prove some lemmas for the last chapter.
 
First of all we give a notion of "proof-irrelevant types", i.e. types for which all proofs are identical.

\begin{defn}
A type $A$ is a \textbf{mere proposition}\index{mere proposition} iff the type $(\Pi x,y:A) Id_A(x,y)$ is inhabited.
\end{defn}
 
\nind
We apologize for the clash of terminology between this definition and the one induced by the Curry-Howard correspondence.\\
Next we study briefly some basic aspects of identity types in Martin-L\"of type theory which can already suggest that terms inside an identity type can be interpreted as paths. For this reason sometimes we will abuse the terminology and use the term "path" in a type-theoretic context, meaning a term in an identity type. For this reason, the elimination rule for identity types will be sometimes called \textit{path induction}\index{path induction}.

\begin{lem}
For every type $A$ and every $x,y:A$ there is a function $Id_A(x,y) \to Id_A(y,x)$, denoted as $p \mapsto p^{-1}$ such that $\textsf{refl}^{-1} = \textsf{refl}$. We call $p^{-1}$ the \textbf{inverse}\index{inverse path} of $p$.
\end{lem}

\begin{proof}
Under the Curry-Howard correspondence, prove a theorem means finding a term in an appropriate type, thus we search a term in the type $(\Pi x,y:A)(Id_A(x,y) \to Id_A(y,x))$.\\
The idea is to use the elimination and computation rules for identity types in order to find a term in $Id_A(x,y) \to Id_A(y,x)$ and to use $\lambda$-abstraction to get the desired function. Recall the elimination rule:

\begin{prooftree}
\AxiomC{$x,y : A, p:Id_A(x,y) \vdash C(x,y,p) \; \mbox{type}$}
\AxiomC{$z: A \vdash d(z) : C(z,z, \textsf{refl}_A(z))$}
\BinaryInfC{$\textsf{J}_{z,d}(x,y,p) : C(x,y,p)$}
\end{prooftree}

\nind
we choose the family of types to be $C(x,y,p) := Id_A(y,x)$, so that the conclusion of the rule will give us the desired "inverse path" if we know how to prove (i.e. find a term inside) $C(x,x, \textsf{refl}_A(x)) = Id_A(x,x)$. We choose the term to be simply the canonical one $\textsf{refl}_A(x)$, so that $p^{-1} := \textsf{J}_{x, \textsf{refl}(x)}(x,y,p) : Id_A(y,x)$.\\
The computation rule gives the desired definitional equality $\textsf{refl}^{-1}(x) = \textsf{refl}(x)$.
\end{proof}

\begin{lem}
For every type $A$ and every $x,y,z:A$, there is a function $Id_A(x,y) \to Id_A(y,z) \to Id_A(x,z)$, written $p \mapsto q \mapsto p \cdot q$ called the \textbf{concatenation} or \textbf{composite}\index{composite path}, such that $\textsf{refl}(x) \cdot \textsf{refl}(x)= \textsf{refl}(x)$.
\end{lem}

\begin{proof}
The type corresponding to the statement is $(\Pi x,y,z:A)(Id_A(x,y) \to Id_A(y,z) \to Id_A(x,z))$. Again we will perform the construction of a term in $Id_A(x,y) \to Id_A(y,z) \to Id_A(x,z)$ and finally use $\lambda$-abstraction.\\
The family of types needed for the elimination rule is $C(x,y,p) := (\Pi z:A)(Id_A(y,z) \to Id_A(x,z))$. Thus, in order to apply the induction principle we need for every $x:A$ a term of type $C(x,x,\textsf{refl}(x)) = (\Pi z:A)(Id_A(x,z) \to Id_A(x,z))$.\\
At this point we may wish to conclude using the identity function $Id_A(x,z) \to Id_A(x,z)$, but we want that our computation rule gives the desired definitional equality $\textsf{refl}(x) \cdot \textsf{refl}(x)= \textsf{refl}(x)$ which would not be the case. If we stopped here after an induction on $p$ over the identity function we would have $\textsf{refl}(y) \cdot q= q$, for $q:Id_A(x,z)$.\\
Hence an induction on $q$ is needed. Now let $D(x,z,q) := Id_A(x,z)$. Note that $D(x,x, \textsf{refl}(x)) = Id_A(x,x)$, then we can use $\textsf{refl}(x) : Id_A(x,x)$ to get by elimination a term in $D(x,z,q) = Id_A(x,z)$ and by $\lambda$-abstraction one in $Id_A(x,z) \to Id_A(x,z)$. Finally, applying the induction principle for identity type to $C$ we get the desired term.

\end{proof}

\nind
The need of a double induction in the previous proof is a consequence of \textit{proof-relevance}\index{proof-relevance} of type theory i.e. the fact that is not only important to know that a type is inhabited (or a proposition is provable) but to know what is the proof. It can matter the kind of specific term we have found in a type in order to perform other constructions.\\
For this reason we cannot stop after the proofs of reflexivity and transitivity but we need to analyse these operations $(-)^{-1}$ and $(-\cdot -)$. The following properties correspond to the usual behaviour of paths.

\begin{lem}
Let $x,y,z,w:A$ and $p:Id_A(x,y)$, $q:Id_A(y,z)$ and $r:Id_A(z,w)$. The following types are inhabited:
\begin{enumerate}[label=(\roman*)]
\item $Id(p, p \cdot \textsf{refl}(y))$ and $Id(p, \textsf{refl}(x) \cdot p)$;

\item $Id(p^{-1} \cdot p , \textsf{refl}(y))$ and $Id(p \cdot p^{-1} , \textsf{refl}(x))$;

\item $Id((p^{-1})^{-1} , p)$;

\item $Id(p \cdot (q \cdot r) , (p \cdot q) \cdot r)$.

\end{enumerate}
All these identity types are of the form $Id_{Id_A}(- , -)$.
\end{lem}

\begin{proof}
As usual we use the induction principle for identity.
\begin{enumerate}[label=(\roman*)]

\item Let $C$ be the family of types $C(x,y,p) := Id_{Id_A}(p, p \cdot \textsf{refl}(y))$. Then $C(x,x,\textsf{refl}(x)) = Id_{Id_A}(\textsf{refl}(x) \cdot \textsf{refl}(x) , \textsf{refl}(x))$. We know that this equality holds definitionally, hence propositionally and we conclude by induction. Similarly for the other identity.

\item Let $C$ be the family of types $C(x,y,p) := Id(p^{-1} \cdot p , \textsf{refl}(y))$. Then $C(x,x, \textsf{refl}(x)) = Id(\textsf{refl}(x)^{-1} \cdot \textsf{refl}(x) , \textsf{refl}(x))$.

\item Let $C$ be the family of types $C(x,y,p) := Id((p^{-1})^{-1} , p)$. Then $C(x,x, \textsf{refl}(x)) = Id((\textsf{refl}(x)^{-1})^{-1} , \textsf{refl}(x))$. Since $\textsf{refl}(x)^{-1} = \textsf{refl}(x)$ we have $C(x,x, \textsf{refl}(x)) = Id((\textsf{refl}(x) , \textsf{refl}(x))$ so that we obtain a term in the desired type by the elimination rule.

\item For this last point and in the sequel we will use a less formal style, leaving the application of the appropriate induction principle (i.e. elimination rule) to the reader. By induction it suffices to assume $p$, $q$ and $r$ are all $\textsf{refl}(x)$. In this case we have $p \cdot ( q \cdot r) = \textsf{refl}(x) \cdot (\textsf{refl}(x) \cdot \textsf{refl}(x) ) = \textsf{refl}(x) = (\textsf{refl}(x) \cdot \textsf{refl}(x)) \cdot \textsf{refl}(x) = (p \cdot q) \cdot r$. Thus we have $\textsf{refl}( \textsf{refl}(x) )$ inhabiting the desired type.
\end{enumerate}
\end{proof}

\begin{lem}
Let $f:A \to B$ be a function, then for any $x,y:A$ there is an operation $\textsf{ap}_f : Id_A(x,y) \to Id_B(f(x),f(y))$\index{$\textsf{ap}_f$}. Moreover, for each $x:A$ we have $\textsf{ap}_f(\textsf{refl}(x)) = \textsf{refl}(f(x))$. We will write $f(p) := \textsf{ap}_f(p)$.
\end{lem}
\begin{proof}
By induction it suffices to assume $p$ is $\textsf{refl}(x)$. In this case we may define $\textsf{ap}_f (p ) := \textsf{refl}(f(x)) : Id_B(f(x) , f(x))$.
\end{proof}

\nind
The following lemma guarantees that we can transport a proof of a dependent type along a path in the base type. In the homotopical interpretation it is analogous to the path-lifting property of fibrations.

\begin{lem}[transport lemma]
Let $B(x)$ be a family of types over $A$, then for every $p:Id_A(x,y)$ there is a function $p_* : B(x) \to B(y)$\index{$p_*$}.
\end{lem}

\begin{proof}
By induction it suffices to suppose that $p = \textsf{refl}(x)$, in this case we take the transport to be the identity function $1_{B(x)}: B(x) \to B(x)$ and we conclude by identity elimination.
\end{proof}

\begin{defn}
A \textbf{type-theoretic homotopy}\index{homotopy!type-theoretic} between the maps $f,g: ( \Pi x:A )B(x)$ is a term in the type $(f \sim g)=( \Pi x:A ) Id_{B(x)}(f(x),g(x))$\index{$(f \sim g)$}.
\end{defn}

\begin{defn}
Given $f:A \to B$ a \textbf{quasi-inverse}\index{quasi-inverse} is a term in the type $\textsf{qinv}(f) := (\Sigma g:B \to A)( (f \circ g \sim 1_B) \times (g \circ f \sim 1_A) )$\index{$\textsf{qinv}(f)$}.
\end{defn}

\nind
The following theorem states that the concept of quasi-inverse, regarded as a type, can have a nontrivial structure.

\begin{thm}
There exist types $A$ and $B$ and a function $f:A \to B$ such that $\textsf{qinv}(f)$ is not a mere proposition.
\end{thm}

\begin{proof}
See theorem 4.1.3 in \cite{UFP}, page 127.
\end{proof}

\nind
For the proof of the generalised type-theoretic interpretation it will not be important if equivalences are mere propositions or not, but the search for alternative notions of equivalence seems essential in order to have a well-behaved notion of univalence, explicitly we look for a type $\textsf{isequiv}(f)$ such that the type $\textsf{isequiv}(f) \leftrightarrow \textsf{qinv}(f)$ is inhabited, and that $\textsf{isequiv}(f)$ is a mere proposition.\\
There are several ways to define equivalences corresponding to these desiderata. Following the article \cite{KLV} we choose:

\begin{defn}
A function $f:A \to B$ is an \textbf{equivalence}\index{equivalence!} or a \textbf{bi-invertible}\index{bi-invertible map} map iff the type 
$$\textsf{isequiv}(f) := (\Sigma g:B \to A) (g \circ f \sim 1_A) \times (\Sigma h:B \to A) (f \circ h \sim 1_B) $$\index{$\textsf{isequiv}(f)$} is inhabited.\\
The \textbf{equivalence type}\index{equivalence!type} is $\mathit{Equiv}(A,B)= (\Sigma f:A \to B) \, \textsf{isequiv}(f) $\index{$\mathit{Equiv}$}.
\end{defn}

\nind
We shall see in the sequel that in homotopy type theory this definition satisfies the requirements to be a mere proposition with maps from and to the type $\textsf{qinv}(f)$.\\

Now we study the structure of the identity types of a sum type. Observe that for every term $i: Id_{(\Sigma x:A) B(x)}(x,y)$ we get $\textsf{p}(i) : Id_A(\textsf{p}(x), \textsf{p}(y))$ where $\textsf{p}$ and $\textsf{q}$ are the two canonical projections. Moreover, we get by path induction a term in the type $Id_B(\textsf{p}(i)_*\textsf{q}(x), \textsf{q}(y))$.\\
In this way we can give a homotopical interpretation: paths in a sum type as pairs where one is a path in the base type and the second a path in the fibre of $\textsf{p}(y)$. The following theorem asserts that this data determine completely paths in a sum type.

\begin{thm}\label{idsum}
Let $B(x)$ be a family of types over $A$ and let $w,w' : (\Sigma x:A)B(x)$. Then there is an equivalence between $Id_{(\Sigma x:A)B(x)}(w,w')$ and $\displaystyle \sum_{i:Id(\textsf{p}(w) , \textsf{p}(w') )} Id_{B(\textsf{p}(w'))} (i_*(\textsf{q}(w)) , \textsf{q}(w'))$.
\end{thm}

\begin{proof}
We define for any $w,w' : (\Sigma x:A)B(x)$ a function 
$$f: Id_{(\Sigma x:A)B(x)}(w,w') \to \displaystyle \sum_{i:Id(\textsf{p}(w) , \textsf{p}(w') )} Id_{B(\textsf{p}(w'))} (i_*(\textsf{q}(w)) , \textsf{q}(w'))$$
by the elimination rule for identity types, with $f(w,w,\textsf{w}) := ( \textsf{refl}(\textsf{p}(w)) , \textsf{refl}(\textsf{q}(w)) )$.\\
For the sake of readability we will write simply $Id(x,y)$ for an identity type when the ambient type will be clear from the context.\\
We want to show that $f$ is an equivalence. In the opposite direction we define:

$$\displaystyle g:  \prod_{w,w' : \Sigma_{x:A}B(x)} \left( \sum_{i: Id_A(\textsf{p}(w) , \textsf{p}(w'))} Id(i_*(\textsf{q}(w)) , \textsf{q}(w'))   \to Id (w,w') \right)$$

\nind
by a first induction on $w$ and $w'$ we split them into pairs $(w_1, w_2)$ and $(w_1' , w_2')$, so it suffices to show that the type is inhabited:
$$\sum_{i: Id_A(w_1 , w_1')} Id_{B(w_1')} (i_*(w_2) , w_2') )  \to Id((w_1 , w_2), (w_1' , w_2'))$$

\nind
Next, by induction on the sum type we take a pair $\displaystyle (i_1 , i_2) : \sum_{i: Id_A(w_1 , w_1')} Id_{B(w_1')} (i_*(w_2) , w_2')  $, then by induction on $i_1 : Id_A(w_1 , w_1')$ we have $i_2 : Id (\textsf{refl}_* (w_2) , w_2')$, and it suffices to show $Id((w_1 , w_2) , (w_1 , w_2'))$, so an induction on $i_2$ reduces to $Id((w_1 , w_2) , (w_1 , w_2))$.\\
Next we show that $f \circ g$ is homotopic to the identity map i.e. that its application is identical to the identity map for every $w$, $w'$ and for every $\displaystyle r : \sum_{i:Id(\textsf{p}(w) , \textsf{p}(w') )} Id (i_*(\textsf{q}(w)) , \textsf{q}(w'))$. First we break these three terms as pairs, and then use two path inductions to reduce both components of $r$ to $\textsf{refl}$. Then it suffices to show that $fg( (\textsf{refl} , \textsf{refl}) ) = \textsf{refl}$, which is true by definition.
Similarly in the opposite direction.\\
We have proved that $f$ has a quasi-inverse, which is a sufficient condition for being an equivalence.
\end{proof}

\begin{cor}
For $z: (\Sigma x:A)B(x)$ we have $Id (z , (\, \textsf{p}(z) , \, \textsf{q}(z)) \,)$, which is a kind of propositional canonicity for terms in a sum type.
\end{cor}

\begin{proof}
We have $\textsf{refl}(\textsf{p}(z)) : Id( \; \textsf{p}(z) , \; \textsf{p} ( \, (  \textsf{p}(z) ,\, \textsf{q}(z) ) \, ) \; )$, so that by the previous theorem it suffices to exhibit a term in $Id(\; \textsf{refl}(\textsf{p}(z))_*(\textsf{q}(z)) , \; \textsf{p}( \, (\textsf{p}(z) , \textsf{q}(z)) \, ) \;)$. But both sides are judgementally equal to $\textsf{q}(z)$.
\end{proof}

\nind
Notice that from the beginning of this section we have worked entirely in Martin-L\"of type theory.\\
It is natural to ask for a similar result for product types, but the rules of Martin-L\"of theory leave the question open. This is a motivation for the introduction of further rules that lead to homotopy type theory.\\

In the previous definitions (as in the whole previous section) we have interpreted terms in a product type as functions, but now is the time for a more careful analysis, in fact without further rules we can only say that terms in a product type can be interpreted as \textit{algorithms} because we are unable to prove that two terms are equal in a product type if they are equal pointwise. Indeed, two algorithms can be different even if they produce the same outputs for the same inputs, for example they can have extremely different computation time.\\
We want to develop some sort of homotopy type theory, therefore we are interested in actual mathematical functions and not in algorithms. For this reason we introduce and discuss two extensional principles that seem needed for a homotopical interpretation of type theory. As a firs extensional principle we may ask that every term in a product type is equal to a $\lambda$-abstraction, so that it is determined by its values. It is called \textbf{$\eta$-rule}\index{$\eta$-rule}. It can come into two guises, propositional or definitional, depending on the kind of equality imposed on terms. We consider the definitional version, which is stronger:\\

\begin{prooftree}
\AxiomC{$f: (\Pi x:A)B(x)$}
\UnaryInfC{$\lambda x. \textsf{ap}(f,x) = f : (\Pi x:A)B(x)$}
\end{prooftree}

\nind
Notice that under the Curry-Howard correspondence this rule says that we can perform conversions of proof and avoid every detour produced by an elimination followed by an introduction. To be honest we have to underline that the interpretation of this rule as an extensional principle require to read the definitional equality from the right to the left, i.e. as an introduction of a proof detour. This ambiguity is present in the literature where the rule is called \textit{$\eta$-reduction}\index{$\eta$-reduction} or \textit{$\eta$-expansion}\index{$\eta$-expansion}, respectively if it read from the left to the right or in the opposite direction.\\
In addition, we recall the interpretation of terms in a type as iterated constructions based on canonical terms which we do not know \textit{a priori} how to calculate. The $\eta$-rule intuitively guarantees that we can perform an inspection and write every term in a product type as a $\lambda$-abstraction.\\
However, this is a quite weak extensional principle and still does not give a statement for product types similar to the one for sum types \ref{idsum}. For this reason we add a suitable axiom.\\
First of all, observe that there is a function $\textsf{happly} : Id_{(\Pi x:A)B(x)} (f,g) \to (\Pi x:A) Id_{B(x)}(f(x) , g(x))$ easily defined by the elimination rule for identity types.\\
Now we state the \textbf{function extensionality axiom}\index{axiom!function extensionality}, in the sequel we shall reformulate it as a rule.

\begin{ax}[function extensionality]
The map $\textsf{happly}$ is an equivalence, i.e. there is an inhabitant $$\textsf{funext} : \mathit{Equiv}(Id_{(\Pi x:A)B(x)} (f,g) , \; (\Pi x:A) Id_{B(x)}(f(x) , g(x)))$$\index{$\textsf{funext}$}
\end{ax}

\nind
Observe that function extensionality and the $\eta$-rule refer to different kinds of extensionality: the former is interpreted internally inside type theory using the Curry-Howard correspondence, whereas the latter refers to the behaviour of application at the level of rules and judgements.\\
Notice that if we interpret types as spaces and terms inside an identity type as paths, we want that a homotopy between functions is the same as a path in a function space, which is exactly the content of function extensionality.\\

The last but not the least is the \textbf{univalence axiom}\index{axiom!univalence}, we give three versions of the axiom: the first one is modelled on Russell style universes and is maybe the easiest way to introduce univalence. Given two small types $A,B :U$ we can form the identity type as well as the equivalence type, $Id_U(A, B)$ and $\mathit{Equiv}(A,B)$. Moreover, we have a map from the former to the latter:

\begin{lem}
Given types in a Russell style universe $A, B : U$, there is a term in the type $Id_U(A,B) \to \mathit{Equiv}(A,B)$.
\end{lem}

\begin{proof}
Note that the identity function of the universe $1_U : U \to U$ can be regarded as a family of types. Thus for every path $p:Id_U(A,B)$ we have that the transport function $p_* : A \to B$. We prove that it is an equivalence.\\
By induction suppose $p = \textsf{refl}(A)$ in which case $p_* = 1_A$ which is an equivalence. Indeed, a quasi-inverse is given by $1_A$ itself together with the trivial homotopies $\textsf{refl}(x)$ and $\textsf{refl}(y)$. Every map with a quasi inverse is in particular an equivalence in our definition.
\end{proof}

\nind
In mathematical practice, an isomorphism between two structured objects allow to transfer the "structural properties" from one object to another. In the categorical context, equivalences have the same behaviour. \\
This can motivate the following:

\begin{ax}[univalence axiom for Russell universes\index{axiom!univalence for Russell universes}]
The function in $Id_U(A,B) \to \mathit{Equiv}(A,B)$, defined in the previous lemma is an equivalence.
\end{ax}

\nind
In particular the axiom says that the canonical map has left and right inverses, hence from a proof of equivalence we can get a (noncanonical) proof of identity. The reader may suspect that we have imposed some sort of "skeletality" so that equivalent objects turn out to be equal; but this is not the case, in fact we have changed material equality into a type-theoretic intensional equality expanding its meaning to fit better the notion of equivalence.\\
We are interested in Tarski style universes, therefore we formulate univalence for them.\\
As we have done in the previous lemma, given a family of types $x:A \vdash B(x) \; \mathit{type}$ we can derive by the elimination rule for identity types the judgement $x,y:A \vdash w_{x,y} : Id_A(x,y) \to \mathit{Equiv}(B(x) , B(y))$.

\begin{defn}
We say that a family of types is \textbf{univalent}\index{univalent family of types} iff for each $x,y : A$ the map $w_{x,y}$ is itself an equivalence, i.e. iff we have a term inhabiting the type $(\Pi x,y:A) \textsf{isequiv}(w_{x,y})$.
\end{defn}

\begin{ax}[univalence axiom for Tarski universes\index{axiom!univalence for Tarski universes}]
The family of types $El$ over the universe $U$ is univalent.
\end{ax}

\nind
Univalence is a statement about a universe, hence we say that a \textbf{univalent universe}\index{universe!univalent} is a universe in which the univalence axiom holds.\\

Equivalence types are mere propositions, as we have anticipated before. Note that this theorem uses function extensionality.

\begin{thm}
For any $f:A \to B$ the equivalence type $\mathit{Equiv}(f)$ is a mere proposition.
\end{thm}

\begin{proof}
See theorem 4.3.2 in \cite{UFP}.
\end{proof}

Now we reformulate the univalence and function extensionality axioms as rules. We remark that it is still an open problem to give a consistency proof for homotopy type theory by means of a normalization theorem, on the other hand it is well known that axioms destroy the normalization procedure (see \cite{Gir} page 125 for a simple counterexample), a first step in this direction is to express axioms as rules.\\
We can interpret this phenomenon observing that a normalization proof for a system of rules is a first step in the study of the conditions of possibility of the new rules, in fact it gives a better explanation of the meaning analysing their mutual syntactical interactions, whereas axioms are apodictic statements, which need to be justified. Usually they are justified semantically by means of objects and considerations external to the theory.\\
We give just the Russell style version of the rules the other being similar, and refer to \cite{UFP}, corollary 5.8.5 and 5.8.6 for the proofs. We have an elimination and a computation rule, their structure correspond to the one of identity types, the identity map takes the place of the canonical reflexivity term $\textsf{refl}$. These rules say that equivalences between small types behaves like paths and are called \textbf{equivalence induction}\index{equivalence!induction}. Extending the interpretation of identity types we may say that these rules express a principle of \textit{indiscernibility of equivalents}\index{indiscernibility of equivalents}:

\begin{prooftree}
\AxiomC{$A,B:U , e:\mathit{Equiv}(A,B) \vdash D(A,B,e) \; \mbox{type}$}
\AxiomC{$A,B:U , e:\mathit{Equiv}(A,B) \vdash d(A):D(A,A,1_A)$}
\BinaryInfC{$A,B:U , e:\mathit{Equiv}(A,B) \vdash f(A,B,e):D(A,B,e)$}
\end{prooftree}

\begin{prooftree}
\AxiomC{$A,B:U , e:\mathit{Equiv}(A,B) \vdash D(A,B,e) \; \mbox{type}$}
\AxiomC{$A,B:U , e:\mathit{Equiv}(A,B) \vdash d(A):D(A,A,1_A)$}
\BinaryInfC{$A,B:U , e:\mathit{Equiv}(A,B) \vdash f(A,A,1_A) = d(A):D(A,A,1_A)$}
\end{prooftree}

\nind
Similarly we reformulate as rules the function extensionality axiom, which is named \textbf{homotopy induction}\index{homotopy!induction} in this form. Again the structure is the same as the one for the elimination and computation of the identity types, thus they say that homotopies behaves like paths:

\begin{prooftree}
\AxiomC{$s,t: (\Pi x:A)B(x) , h:(s \sim t) \vdash D(s,t,h) \; \mbox{type}$}
\AxiomC{$f:(\Pi x:A)B(x) \vdash d(f):D(f,f, \lambda x.\textsf{refl}(f(x)))$}
\BinaryInfC{$s,t: (\Pi x:A)B(x) , h:(s \sim t) \vdash k(s,t,h) : D(s,t,h)$}
\end{prooftree}

\begin{prooftree}
\AxiomC{$s,t: (\Pi x:A)B(x) , h:(s \sim t) \vdash D(s,t,h) \; \mbox{type}$}
\AxiomC{$f:(\Pi x:A)B(x) \vdash d(f):D(f,f, \lambda x.\textsf{refl}(f(x)))$}
\BinaryInfC{$s,t: (\Pi x:A)B(x) , h:(s \sim t) \vdash k(f,f, \lambda x.\textsf{refl}(f(x))) = d(f) : D(f,f,\lambda x.\textsf{refl}(f(x)))$}
\end{prooftree}

We end with some remarks: we may ask if some of the two axioms and the $\eta$-rule can be derived from he others, in fact it can be shown that univalence and the $\eta$-rule together imply function extensionality. See section 4.9 in \cite{UFP}.\\
Anyway we have decided to add and discuss separately all these additional principles in order to give a better exposition and also because someone may wish to change the $\eta$-rule in future, maintaining function extensionality.

Now in the end of this section we can look back to the extensional principles introduced and observe that the search of an equivalence that is a mere proposition, and the two axioms of univalence and function extensionality can be seen as the need for a balance between intensional and extensional concepts in type theory in order to make it adherent to homotopy theory.

We present briefly the idea behind \textit{higher inductive types}\index{higher inductive types}. Type constructors in Martin-L\"of theory can be conceived as "inductive types" because of the inductive structure of elimination rules. Under a homotopical interpretation we may add types defined inductively using not only terms, but also paths and higher paths. For example we can introduce rules to define and manage synthetically the spheres, and calculate synthetically some homotopy groups (see chapter 6 of \cite{UFP}).\\

Now we can finally give the definition of \textbf{weak Tarski universe}\index{universe!weak Tarski}. The idea is to weaken the definitional equalities present in the computation rules for the universe substituting them with propositional equalities. However, we do not have identity types for arbitrary types outside the universe if we do not have a second bigger universe containing that types.\\
Then we ask for an even weaker kind of universes: we replace the definitional equalities with equivalences asking for a canonical term in the appropriate equivalence type.\\
Let start as usual with the formation rule:

\begin{prooftree}
\AxiomC{}
\UnaryInfC{$U \; \mbox{type}$}
\end{prooftree}

\begin{prooftree}
\AxiomC{$a:U$}
\UnaryInfC{$El(a) \; \mbox{type}$}
\end{prooftree}

\nind
Then we have the introduction rules:

\begin{prooftree}
\AxiomC{$a:A$}
\AxiomC{$x:El(a) \vdash b(x)$}
\BinaryInfC{$\pi(a,b(x)):U$}
\end{prooftree}

\nind
The computation rule for the product types expresses the presence of a canonical term in the equivalence type between $El( \pi(a,b(x)) )$ and $(\Pi x:El(a)) El(b(x))$, as follows:

\begin{prooftree}
\AxiomC{$a:A$}
\AxiomC{$x:El(a) \vdash B(x)$}
\BinaryInfC{$\textsf{ceq}_{\pi} : \mathit{Equiv}( \, El( \pi(a,b(x)) ) , \, (\Pi x:El(a)) El(b(x)) \, )$}
\end{prooftree}

\nind
The rules for the reflection of identity types are:

\begin{prooftree}
\AxiomC{$a:U$}
\AxiomC{$b:El(a)$}
\AxiomC{$c:El(a)$}
\TrinaryInfC{$i_a(b,c) : U$}
\end{prooftree}

\begin{prooftree}
\AxiomC{$a:U$}
\AxiomC{$b:El(a)$}
\AxiomC{$c:El(a)$}
\TrinaryInfC{$\textsf{ceq}_{id} : \mathit{Equiv}( El(i_a(b,c)) ,  Id_{El(a)}(b,c) )$}
\end{prooftree}

\nind
The rules for the reflection of the type of natural numbers:

\begin{prooftree}
\AxiomC{}
\UnaryInfC{$n : U$}
\end{prooftree}

\begin{prooftree}
\AxiomC{}
\UnaryInfC{$\textsf{ceq}_{n} : \mathit{Equiv}(El(n), \textbf N )$}
\end{prooftree}

\nind
And similarly for the other type constructors.\\

\chapter{Homotopy Type Theory in Simplicial Sets}

\lettrine{S}{implicial} sets form a locally cartesian closed category, and a model category. We shall see that these structures allow to interpret respectively dependent types and intensional identity types.\\
We emphasize that in order to give a clear and simple exposition at first we will ignore important coherence issues that we shall discuss briefly in the sequel. In fact the na\"ive interpretation of type theory give rise to problems with substitution which is interpreted as pullback, so that it is in general stricter than its semantical counterpart, for example substitution is strictly associative, whereas pullbacks associate only up to isomorphism, hence the need of coherence theorems asserting that pullbacks can be chosen in a way to fit the type-theoretic substitution. Other coherence issues arise from the computation rules for the universe.\\
Following \cite{KLV} we give a sketch of Voevodsky's answer to the coherence issues for simplicial sets, the source for $W$-types in simplicial sets is \cite{MvdB1} and in the end we give a proof of univalence using \cite{Moe}.\\
In this chapter we give sketches and outlines without going much into the details in order to maintain the size of this thesis under control because proofs are quite long and technical.\\
In this chapter we shall use heavily the axiom of choice, both for univalence and the coherence conditions, in the end we will use two inaccessible cardinals.

We give now a sketch of the interpretation of dependent sums and products in a locally cartesian closed category, the idea is that the slices allow to express type dependencies and that the right and left adjoint to the pullback functor give respectively dependent products and sums.  
The interpretation of dependent type theory in locally cartesian closed categories (lccc for short) was studied firstly in Seely's seminal paper \cite{See} with inaccuracies in the treatment of substitution, corrections in this sense can be found in \cite{Hofb}.\\
For a review on locally cartesian closed categories see appendix \ref{lccc}.

Given a locally cartesian closed category $\mathcal C$ the basic idea is to interpret a context $\Gamma$ as an object $\llbracket \Gamma \rrbracket$, in particular every type is interpreted as an object of the category. The empty context is interpreted as the terminal object of $\mathcal C$. Types in context like  $\Gamma \vdash A \; \mathit{type}$ are interpreted as maps with the interpretation of $\Gamma$ as codomain $p_A : \llbracket \Gamma , A \rrbracket \to \llbracket \Gamma \rrbracket$, so that they are objects in the slice category $\mathcal C / \llbracket \Gamma \rrbracket$.\\
A substitution which yields the context $\Gamma$ from the context $\Delta$, is interpreted as a map $\llbracket \Delta \rrbracket \to \llbracket \Gamma \rrbracket$, and when this substitution is applied to a dependent type of the form $\Delta \vdash A \; \mathit{type}$ the resulting judgement is interpreted as the pullback of $\llbracket \Delta , A \rrbracket \to \llbracket \Delta \rrbracket$ along the substitution map $\llbracket \Gamma \rrbracket \to \llbracket \Delta \rrbracket$.\\
Each term in context like $\Gamma \vdash x:A$ is interpreted as a section of $p_A$. The unit type (in the empty context) is interpreted as a terminal object. For a dependent type $\Gamma , x:A \vdash B(x) \; \mathit{type}$
 we interpret the dependent sum $\Gamma \vdash (\Sigma x:A)B(x)$ as the composite $\llbracket \Gamma , A , B \rrbracket \to \llbracket \Gamma , A \rrbracket \to \llbracket \Gamma \rrbracket$ i.e. as the application of the left-adjoint to the pullback $\Sigma_{p_A} (\llbracket \Gamma , A , B \rrbracket  \to \llbracket \Gamma , A \rrbracket) : \mathcal C / \llbracket \Gamma , A \rrbracket \to \mathcal C / \llbracket \Gamma \rrbracket$.\\
Similarly, for the dependent products we use the right adjoint $\Pi_{p_A} (\llbracket \Gamma , A , B \rrbracket \to \llbracket \Gamma , A \rrbracket) : \mathcal C / \llbracket \Gamma , A \rrbracket \to \mathcal C / \llbracket \Gamma \rrbracket$.

Let us now consider the problem with this na\"ive interpretation: for example the standard choice of pullbacks in the category of sets does not work, indeed if $f:A \to B$, $g:B \to C$ and $h:D \to C$ are three maps, then the pullback of $h$ along $g \circ f$ is the set $\{ ( a,d ) \, | \,  a \in A , d \in D , g(f(a)) = h(d) \}$ whereas the iterated pullback of $h$ along $g$ and then along $f$ is $\{ ( a , ( b , d ) ) \, | \, a \in A , b \in B , d \in D , f(a) = b \;  \mbox{and} \; g(b) = h(d) \}$ which is isomorphic but not equal to the former.\\
The are several ways to solve this issue, but they share the common idea to perform constructions in the lccc in order to get more structure and then use it to build a semantic substitution operation which commutes with composition and all semantic type and term formers. Then the interpretation is changed in order to make coherent choices of the interpretations based on the semantical substitution.\\
Following \cite{KLV} we will use \textit{contextual categories}\index{category!contextual} (the standard sources are \cite{Car} and \cite{Str}). Other possibilities are \textit{categories with attributes}\index{category!with attributes} again in \cite{Car}, \textit{categories with families}\index{category!with families} \cite{Hof} or \textit{comprehension categories}\index{category!comprehension} \cite{Jac}.

Next we consider intensional identity types, which need a subtler treatment, in fact they bring homotopical content into type theory. For the interpretation of identity type we will use the model structure on simplicial sets.\\
The identity type $\Gamma , x:A , y:A \vdash Id_A(x,y) \; \mathit{type}$ is interpreted as a very good path object, i.e. one obtained as an acyclic cofibration followed by a fibration  $P_{\llbracket \Gamma \rrbracket} \llbracket A \rrbracket \to \llbracket A \rrbracket \times_{\llbracket \Gamma \rrbracket} \llbracket A \rrbracket$, with the reflexivity term interpreted by the acyclic cofibration $\llbracket A \rrbracket \to P_{\llbracket \Gamma \rrbracket} \llbracket A \rrbracket $. Moreover, the previously considered interpretation of dependent types need to be restricted to fibrations and fibrant objects. Hence some conditions on the lccc category in consideration should be imposed in order to guarantee that fibrations are stable under dependent products, the case of dependent sum is trivial because the composition of fibrations is again a fibration. Additionally, choices of all data including liftings need to be given such that they commute with pullbacks.\\
In the following lemmas we take care of some details for the category of simplicial sets:

\begin{lem}
Suppose $q : Z \to Y$ and $p:Y \to X$ are fibrations, then the dependent product $\Pi_p q$ is a fibration over $X$.
\end{lem}

\begin{proof}
Recall that $\textbf{SSet}$ is a right proper model category whose cofibrations are exactly the monomorphisms, hence the pullback functor $p_* : \textbf{SSet}/X \to \textbf{SSet}/Y$ preserves trivial cofibrations. Therefore by adjointness $\Pi_p$ preserves fibrations.
\end{proof}

\nind
Some notions of \textit{type-theoretic model category} (i.e. categories with minimal requirements for the structure needed to interpret Martin-L\"of type theory with identity types) have been proposed, they rely on some generalisation of these two key conditions: right properness, and the property that cofibrations are the monomorphisms (see for example \cite{Shu1}).

Now we focus on identity types showing how the monoidal model structure on simplicial sets can help with one of the many coherence issues involved.

\begin{defn}
Given a fibration $p:E \to B$ in $\textbf{SSet}$ we define the \textbf{fibred path object}\index{fibred path object} $P_B(E)$ as the pullback:

$$\xymatrix@1{
& P_B(E) \ar[r] \ar[d] 
& E^{\Delta[1]} \ar[d]^{p^{\Delta[1]}} \\
& B \ar[r]^{c}
& B^{\Delta[1]}
}$$

\nind
where $c$ is the map "constant path". Consider now the constant path map for $E$ that is $c:E \to E^{\Delta[1]}$, which factors through $P_B(E)$ so that we get $r:E \to P_B(E)$. Moreover, there are source and target maps $s,t:P_B(E) \to E$.

\end{defn}

\begin{lem}
For any fibration $p:E \to B$ in $\textbf{SSet}$ the maps $\displaystyle E \mathop{\longrightarrow}^{r} P_B(E) \mathop{\longrightarrow}^{( s,t )} E \times_B E$ give a factorisation of the diagonal $E \to E \times_B E$ over $B$ as a trivial cofibration followed by a fibration. Moreover, this construction is \textit{stable over $B$}, i.e. the pullback along any $B' \to B$ is again such a factorisation.
\end{lem}

\begin{proof}
It is clear that these maps give a factorisation of the diagonal. To see that they are a trivial cofibration and a fibration respectively, split the construction of $P_B(E)$ in two intermediate stages:

$$\xymatrix@1{
& P_B(E) \ar[r] \ar[d]_{( s , t )} 
& E^{\Delta[1]} \ar[d]^{( s , p^{\Delta[1]} , t )} \\
& E \times_B E \ar[r] \ar[d]_{\pi_1}
& E \times_B B^{\Delta[1]} \times_B E \ar[d]^{ ( \pi_1 , \pi_2 )} \\
& E \ar[r] \ar[d]
& E \times_B B^{\Delta[1]} \ar[d] \\
& B \ar[r]^{c}
& B^{\Delta[1]}
}$$

\nind
These three square are all pullbacks. Notice that the map $( s , p^{\Delta[1]} , t )$ is a fibration because of lemma \ref{Quillenbif} applied to the cofibration $1 + 1 \to \Delta[1]$ and to the fibration $p$. Hence $( s , t )$ is a fibration because it is pullback of a fibration.\\
Similarly, the source map $s$ is a trivial fibration since it is a pullback of $E^{\Delta[1]} \to E \times_B B^{\Delta[1]}$, which is one again by lemma \ref{Quillenbif}. Observe that $s$ is a retraction of $r$ so that $r$ is a weak equivalence by the 2-out-of-3 property and a monomorphism, therefore it is a trivial cofibration as desired.\\
Finally, the stability under pullback follows from the stability of the construction itself: for any $f: B \to B'$ there is a canonical isomorphism $P_{B'}(f^* E) \cong f^* P_B(E)$, commuting with the maps $r$, $s$ and $t$.
\end{proof}

\nind
Next we study the categorical analogues to $W$-types. This notion can be defined in every locally cartesian closed category.

\begin{defn}
Let $\mathcal C$ be a locally cartesian closed category, and $f : A \to B$ be any map, the \textbf{polynomial functor}\index{polynomial functor} $P_f$ associated to $f$ is the composite:

$$ \displaystyle  \mathcal C \mathop{\longrightarrow}^{- \times B} \mathcal C / B \mathop{\longrightarrow}^{\Pi_f} \mathcal C / A \mathop{\longrightarrow}^{\Sigma_A} \mathcal C$$ 

\nind
If exists, the initial algebra for a polynomial endofunctor is called the \textbf{$W$-type associated to $f$}\index{$W$-type associated to $f$} and denoted $W(f)$\index{$W(f)$}.
\end{defn}

\nind
The category of sets has all $W$-types. Indeed, for a map $f: B \to A$ in set theoretic notation we have $P_f(X) = \sum_{a \in A} X^{B_a}$ where $B_a = f^{-1}(a)$. Then $W(f)$ is the set of labelled well-founded trees with nodes labelled by the elements of $A$ and the edges labelled by the elements of $B$. The labelling is such that for a given node $a \in A$ the edges coming into it are labelled by elements in the fibre $b \in B_a$.\\
The algebra structure is given by the map $\mathit{sup}: P_f(W(f)) \to W(f)$ defined in the following way: given $a \in A$ and $t:B_a \to W(f)$ we can form a new tree with as root the node of the starting tree labelled by $a$, whose edges are labelled by the elements of $B_a$ and "subtrees" linked to the root the ones determined by $t(b)$ for $b \in B_a$.\\
For these trees we can give the following:

\begin{defn}
Define by recursion the notion of \textbf{rank}\index{rank of a well-founded tree} $rk:W(f) \to \mathit{Ord}$ as $rk(\mbox{sup}(a,t)) := \mbox{sup}\{ rk(t(b))+1 \, | \, b \in B_a \}$. In addition we define $W(f)_{< \alpha} := \{ w \in W(f) \, | \, rk(w) < \alpha \}$.
\end{defn}

\nind
Observe that $W(f)_{< 0} = \void$ and $W(f)_{< \alpha + 1} = P(f)(W(f)_{< \alpha})$. It is easy to check that for a regular cardinal $\kappa$ strictly bigger than all the $B_a$, then $W(f) = W(f)_{< \kappa}$.

Categories of presheaves have all $W$-types as well. Given a category $\mathcal C$ and a morphism of presheaves $f: B \to A$ we write $\hat{A} := \{ (C, a) \, | \, C \in Ob(\mathcal C), a \in A(C) \} $. For $(C , a) \in \hat A$ define $\hat B_{(C , a)} := \{ (\alpha , b) \, | \, \alpha : D \to C , b \in B(D) \; \mbox{and} \; f_D(b) = a \cdot \alpha \}$; finally, let $\hat f$ be the projection $\hat f : \sum_{(C , a) \in \hat A} \hat B_{(C , a)} \to \hat A$.\\
As a first step we take the corresponding $W$-type in sets: $W(\hat f)$. Now we give it a presheaf structure, to do so we declare that an element $\mbox{sup}((C , a) , t)$  lives in the fibre over $C \in Ob(\mathcal C)$ and for any $\alpha : D \to C$ its restriction is given by the formula: $\mbox{sup}((C , a) , t) \cdot \alpha = \mbox{sup}((D , a \cdot \alpha) , (t \cdot \alpha))$ where $(t \cdot \alpha)(\beta , b) := t(\alpha \beta , b)$.\\
As before we can assign a rank by transfinite recursion: $$rk(\mbox{sup}((C , a) , t)) := \mbox{sup}\{ rk(t(\beta , b)) + 1 \, | \, (\beta , b) \in \hat B_{(C , a)} \}$$

\begin{defn}
\begin{enumerate}[label=(\alph*)]
\item[]

\item A tree $\mbox{sup}((C , a) , t)$ is \textbf{composable}\index{tree!composable} iff for any $(\alpha , b) \in \hat B_{(C , a)}$ the tree $t(\alpha , b)$ lives in the fibre over $\mbox{dom}(\alpha)$.

\item A tree $\mbox{sup}((C , a) , t)$ is called \textbf{natural}\index{tree!natural} iff the map $t$ is a natural transformation i.e. for any $(\alpha , b) \in \hat B_{(C , a)}$ and $\beta : E \to D$ we have that $t(\alpha  \beta , b \cdot \beta) = t(\alpha , b) \cdot \beta$.

\item The collection of \textbf{subtrees}\index{tree!subtree} of $\mbox{sup}((C , a) , t)$ is defined recursively as the collection consisting of $\mbox{sup}((C , a) , t)$ itself and all subtrees of the $t(\alpha , b)$.

\item A tree $\mbox{sup}((C , a) , t)$ is \textbf{hereditarily natural}\index{tree!hereditarily natural} iff all its subtrees are natural.

\item The $W$-type in presheaves associated to $f$, written $W(f)$, is the subpresheaf of $W(\hat f)$ consisting of hereditarily natural trees.
\end{enumerate}
\end{defn}

\nind
This definition give rise to the desired initial algebra for the polynomial endofunctor.\\
In addition we put $W(f)_{< \alpha} := \{ w \in W(f) \, | \, rk(w) < \alpha \} $.

\begin{thm}
If $p: Y \to X$ is a Kan fibration, then the canonical map $W(f)_{<\alpha} \to X$ is a Kan fibration.
\end{thm}

\begin{proof}
See theorem 3.4 in \cite{MvdB1}
\end{proof}

Now is the time for a concise exposition of the strategy used by Voevodsky to interpret universes solve the coherence issues for them and for substitution.\\
As we have said before the idea is to put more structure on the locally cartesian closed category. The construction splits in some intermediate steps:

$$\displaystyle \textbf{SSet} \to \mbox{lccc with a category-theoretic universe} \to \mbox{contextual category} \to \mbox{type theory}$$

\nind
The first step is to construct a strict universal fibration in simplicial set. This category-theoretic universe is used to build a literal translation of type theory into category theory (a contextual category), which finally yields the desired strict interpretation of type theory.\\
In the exposition we will walk these arrows in the opposite direction, starting from the definition of contextual category.\\
This definition is a little bit involved, although the basic idea is simple i.e. that objects behave like contexts, so that they have a grade corresponding to the length of the context, and for any object there is a map which forgets the last judgement of the context projecting it into a lesser grade. In addition, there are chosen pullbacks added to the structure which satisfy by definition the strict functoriality condition.

\begin{defn}
A \textbf{contextual category}\index{category!contextual} is a category with the following structure:
\begin{itemize}
\item a grading of objects $ Ob(\mathcal C) = \coprod_{n \in \mathbb N} Ob_n(\mathcal C)$;

\item an object $1 \in Ob_0(\mathcal C)$;

\item forgetting maps $ft: Ob_{n+1}(\mathcal C) \to Ob_n(\mathcal C)$;

\item for every object $X \in Ob_{n+1}(\mathcal C)$ the \textit{canonical projection} $p_X : X \to ft(X)$;

\item for each $X \in Ob_{n+1}(\mathcal C)$ and $f : Y \to ft(X)$, there are a distinguished object $f^*(X)$ and morphism $q(f, X) : f^*(X) \to X$.

\end{itemize}
Such that the following conditions hold:
\begin{itemize}
\item $1$ is the unique object of $Ob_0(\mathcal C)$;

\item $1$ is a terminal object in $\mathcal C$;

\item for each $n> 0$, each $X \in Ob_n(\mathcal C)$ and each $f: Y \to ft(X)$, we have $ft(f^*X) = Y$ and the following square is a pullback called the \textit{canonical pullback} of $X$ along $f$:

$$\xymatrix@1{
& f^*X \ar[r]^{q(f,X)} \ar[d]_{p_{f^*X}} 
& X \ar[d]^{p_X} \\
& Y \ar[r]^{f}
& ft(X)
}$$

\item these canonical pullbacks are strictly functorial, i.e. for $X \in Ob_{n+1}(\mathcal C)$, we have that $1^*_{ft(X)} X = X$ and $q(1_{ft(X)} , X) = 1_X$; moreover, for $X \in Ob_{n+1}(\mathcal C)$, $f: Y \to ft(X)$ and $g: Z \to Y$, we have $(fg)^* X = f^*(g^*(X))$ and $q(fg , X) = q(f, X)q(g , f^*X)$.

\end{itemize}
\end{defn}

\begin{rem}
The intuitive explanation given before the definition is justified by the observation that every system of dependent types give rise to a contextual category $\mathcal C$, described as follows:
\begin{itemize}
\item $Ob_n(\mathcal C)$ is formed by contexts of length $n$ up to definitional equality and renaming of free variables;

\item maps of $\mathcal C$ are \textit{substitutions} or \textit{context morphisms} up to definitional equality and renaming of free variables. That is, a map $f: (x_1 : A_1 , \dots , x_n : A_n) \to (y_1 : B_1 , \dots , y_m : B(y_1, \dots , y_{m-1}))$ is represented by a sequence of terms $(f_1 , \dots , f_m)$ such that $x_1 : A_1 , \dots , x_n : A_n \vdash f_1 : B_1 $ and so on until $ x_1 : A_1 , \dots , x_n : A_n \vdash f_m : B_m(f_1 , \dots , f_{m-1})$; two such maps $(f_i)_i , (g_i)_i$ are equal iff for each $i$ we have $x_1 : A_1 , \dots , x_n : A_n \vdash f_i = g_i : B_(f_1 , \dots , f_{i-1})$;

\item composition of maps is given by substitution of terms and the identity $\Gamma \to \Gamma$ by the variables of $\Gamma$ considered as terms;

\item $1$ is the empty context;

\item $ft((x_1 :A_1, \dots , x_{n+1} : A_{n+1})) := (x_1 :A_1 , \dots , x_n : A_n)$, and the map $p_{\Gamma} : \Gamma \to ft(\Gamma)$ is simply the map forgetting the last judgement;

\item for contexts $\Gamma = (x_1 :A_1, \dots , x_{n+1} : A_{n+1}(x_1 , \dots , x_n))$ and $\Gamma' = (y_1 : B_1 , \dots , y_m : B_m(y_1 , \dots , y_{m-1}))$ and a map $f=(f_1(\underline{y}) , \dots , f_n(\underline{y})) : \Gamma ' \to ft(\Gamma)$ the canonical pullback $f^*\Gamma$ is the context $(y_1 : B_1 , \dots , y_m : B_m(y_1 , \dots , y_{m-1}) , y_{m+1} : A_{n+1}(f_1(\underline{y}) , \dots , f_m( \underline{y})))$; finally, $q(\Gamma , f) : f^* \Gamma \to \Gamma$ is the map $(f_1 , \dots , f_n , y_{n+1})$. 
\end{itemize}
\end{rem}

\begin{defn}
Given a category $\mathcal C$, a \textbf{category-theoretic universe}\index{universe!category-theoretic} is an object $U$ together with a morphism $p: \tilde{U} \to U$ such that each map $f: X \to U$ has a choice of pullback:

$$\xymatrix@1{
& (X ; f) \ar[r]^{Q(f)} \ar[d]_{P(X; f)} 
& \tilde{U} \ar[d]^{p} \\
& X \ar[r]^{f}
& U
}$$

\nind
For a sequence of maps $f_1 : X \to U$, $f_2 : (X;f_1) \to U$ and so on, we will write $(X ; f_1 , \dots , f_n)$ for the iteration $((\dots (X ; f_1); \dots ); f_n)$.
\end{defn}

\nind
In the following definition we explain how to build a contextual category from a category-theoretic universe.

\begin{defn}
Given a category $\mathcal C$ with a universe $U$ and a terminal object $1$, we define a contextual category $\mathcal C_U$ as follows:
\begin{itemize}
\item $Ob_n(\mathcal C_U) := \{ (f_1 , \dots , f_n) \, | \, f_i : (1 ; f_1 , \dots , f_{i-1}) \to U \}$;

\item $\mathit{Hom}_{\mathcal C_U}((f_1 , \dots , f_n) , (g_1 , \dots , g_m)) := \mathit{Hom}_{\mathcal C}((1; f_1 , \dots , f_n) , (1; g_1 , \dots , g_m))$;

\item $1_{\mathcal C_U} := ()$ the empty sequence;

\item $ft((f_1 , \dots , f_{n+1})) := (f_1 , \dots , f_n)$;

\item the projection $p_{(f_1 , \dots , f_{n+1})}$ is the map $P(X ; f_{n+1})$ provided by the universe structure on $U$;

\item given $(f_1 , \dots , f_{n+1})$ and a map $\alpha : (g_1 , \dots , g_m) \to (f_1 , \dots , f_{n+1})$ in $\mathcal C_U$, the canonical pullback $\alpha ^* (f_1 , \dots , f_{n+1})$ is given by $(g_1 , \dots , g_m , f_{n+1} \alpha )$ with projection induced by $Q(f_{n+1}\alpha)$:

$$
\xymatrix@1{
& (1; g_1 , \dots , g_m , f_{n+1} \alpha ) \ar[r] \ar@/^1pc/[rr]^{Q(f_{n+1}\alpha)} \ar[d]
& (1; f_1 , \dots , f_{n+1}) \ar[r] \ar[d]
& \tilde{U} \ar[d]^{p} \\
& (1; g_1 , \dots , g_m ) \ar[r]^{\alpha}
& (1; f_1 , \dots , f_n) \ar[r]^-{f_{n+1}}
& U
}$$

\end{itemize}
\end{defn}

\nind
We have given definitions only for the structure needed to manage substitution. The next step should be to give suitable definition of $\Pi$, $\Sigma$, $Id$, $W$ and type-theoretic universe structure on a contextual category and on a category-theoretic universe, so that from the latter we can easily reconstruct the former. Moreover, we have not checked that this definition is well-posed. For these details we refer the interested reader to \cite{KLV}.\\

The last step is to build a category-theoretic universe with all these structures inside the category of simplicial sets. The construction uses a well-ordering trick in order to get a universal fibration that classifies fibrations uniquely (not only uniquely up to homotopy).\\
We need also to impose a smallness condition if we want that our universal map is a set instead of a proper class. Hence we briefly recall few notions from set theory.

\begin{defn}
\begin{enumerate}[label=(\alph*)]
\item[]

\item Given an ordinal $\alpha$ its \textbf{cofinality}\index{cofinality} is the least ordinal that can be injected cofinally in $\alpha$. In symbols: $\mbox{cof}(\alpha) := \mbox{min}\{ \beta \, | \, \exists f: \beta \to \alpha \; \mbox{cofinal in } \alpha  \}$;\index{$\mbox{cof}(\alpha)$}

\item A cardinal $\alpha$ is \textbf{regular}\index{regular cardinal} iff it is equal to its cofinality. 

\item A cardinal $\alpha$ is \textbf{strongly inaccessible}\index{strongly inaccessible cardinal} (or simply inaccessible) iff it is regular, uncountable and $2^{\lambda} < \alpha$ for every $\lambda < \alpha$.
\end{enumerate}
\end{defn}

\nind
If ZFC is consistent it cannot prove the existence of inaccessible cardinals.\\
Now we use the notion of regular cardinal to give a notion of smallness for a fibration of simplicial sets.

\begin{defn}
Fix a regular cardinal $\alpha$, we say that a map $X \to Y$ is \textbf{$\alpha$-small}\index{$\alpha$-small map} iff it has fibres with cardinality less than $\alpha$.
\end{defn}

\begin{defn}
\begin{enumerate}[label=(\alph*)]
\item[]

\item A \textbf{well-ordered morphism}\index{well-ordered morphism of simplicial sets} of simplicial sets is a morphism $f: Y \to X$ together with a function assigning to each simplex $x \in X_n$ a well-ordering on the fibre $f^{-1}(x) \subseteq Y_n$.

\item Given two well-ordered morphisms $f:Y \to X$ and $g:Z \to X$ into a common base a \textbf{morphism of well-ordered morphisms}\index{morphism of well-ordered morphisms} from $f$ to $g$ is a morphism $Y \to Z$ respecting the fibres and the well-ordering on each fibre.
\end{enumerate}
\end{defn}

\begin{defn}
Given a regular cardinal $\alpha$ and a simplicial set $X$ we define $\textbf W_{\alpha}(X)$ as the set of isomorphism classes of $\alpha$-small well-ordered morphisms into $X$. We can define $\textbf W_{\alpha}$ on morphisms as the pullback action on isomorphism classes making it into a  functor $\textbf W_{\alpha} : \textbf{SSet}^{op} \to \textbf{SSet}$.\\
Furthermore, we define a simplicial set by composing with the Yoneda embedding of $\Delta$ into $\textbf{SSet}$, so that $W_{\alpha} := \textbf W_{\alpha} \cdot y^{op} : \Delta^{op} \to \textbf{SSet} $\index{$W_{\alpha}$}.
\end{defn}

\begin{lem}
The functor $\textbf W_{\alpha}$ is representable, represented by $W_{\alpha}$.
\end{lem}

\begin{proof}
The functors $\textbf W_{\alpha}$ and $\mathit{Hom}(- , W_{\alpha})$ agree up to isomorphism on the standard simplices by the Yoneda lemma. It is easy to check that $\textbf W_{\alpha}$ sends colimits into limits, hence they always agree because every simplicial set is canonically a colimit of standard simplices.
\end{proof}

\nind
Applying the natural isomorphism above to the identity map $1_{W_{\alpha}}$ we get a map $\widetilde{W}_{\alpha} \to W_{\alpha}$. By construction we have that every $\alpha$-small morphism of simplicial set can be obtained as a pullback of this projection. Indeed, by the axiom of choice we can fix a well-ordering on the fibres and use the universal property of $W_{\alpha}$.

\begin{defn}
Let $U_{\alpha} \subseteq W_{\alpha}$\index{$U_{\alpha}$} be the subobject consisting of isomorphism classes of $\alpha$-small fibrations and $p_{\alpha} : \widetilde{U}_{\alpha} \to U_{\alpha}$ as the pullback:

$$\xymatrix@1{
& \widetilde{U}_{\alpha} \ar[r] \ar[d]_{p_{\alpha}} 
& \widetilde{W}_{\alpha} \ar[d] \\
& U_{\alpha} \ar[r]
& W_{\alpha}
}$$

\end{defn}

\begin{thm}
The map $p_{\alpha} : \widetilde{U}_{\alpha} \to U_{\alpha}$ is a fibration. Moreover, the simplicial set $U_{\alpha}$ is a Kan complex.
\end{thm}

\begin{proof}
See theorems 2.1.10 and 2.2.1 in \cite{KLV}.
\end{proof}

\begin{thm}
Let $\alpha$ be a strongly inaccessible cardinal, then $U_{\alpha}$ carries $\Pi$, $\Sigma$, $Id$, $W$, $0$, $1$ and $2$ structure.\\
Moreover, if $\beta < \alpha$ is another inaccessible then $U_{\beta}$ gives an internal universe structure closed under all the other type constructors.
\end{thm}

\begin{proof}
See theorem 2.3.4 in \cite{KLV}.
\end{proof}

\nind
With the previous theorem we have finished our detour on coherence issues for Martin-L\"of type theory. Let now turn to homotopy type theory and univalence.

\begin{thm}
The $\eta$-rule and functional extensionality hold in the simplicial model.
\end{thm}

\begin{proof}
The proof for the $\eta$-rule is simply the observation that the two maps involved are forced to coincide thanks to the uniqueness in the universal property for exponentials in a locally cartesian closed category.\\
For function extensionality see theorem 2.3.6 in \cite{KLV}.
\end{proof}

\begin{defn}
Given a fibration $E \to B$ we can form the simplicial set $\mathit{Eq}(E)$\index{$\mathit{Eq}(E)$} whose $n$-simplices are $(x,y,p)$ where $x,y : \Delta[n] \to B$ and $p: x^*E \to y^* E$ is a weak equivalence.\\
Similarly we define $\mathit{Iso}(E)$\index{$\mathit{Iso}(E)$} whose $n$-simplices are $(x,y,i)$ where $x,y : \Delta[n] \to B$ and $i: x^*E \to y^* E$ is an isomorphism.
\end{defn}

\nind
Notice how the internal $\mathit{Hom}$-complex allows to express properties of maps and simplicial sets (like being an equivalence) as other simplicial sets (like $\mathit{Eq}(E)$) i.e. objects inside the theory. It is reasonable to expect some for of correspondence between the analogous constructions internal to type theory allowed by the Curry-Howard isomorphism.

\begin{defn}
A fibration $p:E \to B$ is \textbf{univalent}\index{fibration!univalent} iff the obvious map $B^{\Delta[1]} \to \mathit{Eq}(E)$ is a weak equivalence.\\
The \textbf{simplicial univalence axiom}\index{axiom!simplicial univalence} states that the universal fibration $\widetilde
U_{\alpha} \to U_{\alpha}$ is univalent.
\end{defn}

\nind
Observe that by the 2-out-of-3 property $p:E \to B$ is univalent iff the canonical diagonal map $\delta : B \to \mathit{Eq}(E)$ is a weak equivalence, which in turn is the same to say that $B \to \mathit{Eq}(E) \to B \times B$ is a (trivial cofibration, fibration) factorisation of the diagonal, because the first is a cofibration and the second a fibration in any case. Hence $p$ is univalent iff $\mathit{Eq}(E)$ is a very good path object for $B$.\\
Notice that this definition of simplicial univalence matches the basic na\"ive idea of the interpretation of type theory without any worry about coherence issues. For this reason we need the following:

\begin{lem}
Simplicial univalence and type-theoretic univalence in simplicial sets are equivalent.
\end{lem}

\begin{proof}
The proof requires several technical lemmas. See theorem 3.3.7 in \cite{KLV}.
\end{proof}

\nind
In view of the proof of simplicial univalence, we present here some definitions about simplicial principal bundles.

\begin{defn}
\begin{enumerate}[label=(\alph*)]
\item[]

\item Given a simplicial group $G$ and a Kan complex $E$ we say that an action $\rho : G \times E \to E$ is \textbf{principal}\index{principal simplicial action} iff for every fixed degree it is principal, i.e. the only elements $g \in G_n$ that have any fixed point $e \in E_n$ are the neutral elements.

\item Given $G$ a simplicial group, a morphism $P \to X$ of Kan complexes equipped with a $G$-action on $P$, is called a \textbf{$G$-simplicial principal bundle}\index{simplicial!principal bundle} iff the action is principal and the base is isomorphic to the quotient $E / G$. 
\end{enumerate}
\end{defn}

\begin{rem}
The classical construction of the universal principal $G$-bundle $\textbf EG \to \textbf BG$ generalises to the simplicial case. Recall that $\textbf E G$ is weakly contractible (i.e. it has all homotopy groups trivial) and that $G$ acts freely on $\textbf E G$, hence $\textbf B G \cong \textbf E G / G$.\\
In addition, the correspondence between fibre bundles and principal $G$-bundles generalises as well. From any principal $G$-bundle $P \to X$ we form the quotient $P \times_G F := P \times F / \sim$ where $(p,gy) \sim (pg,y)$ when $g$ runs over $G$. Conversely, given any fibre bundle with fibre $F$ we put $G=Aut(F)$ and consider the associated \textit{frame bundle}\index{frame bundle} $\mathit{Fr}(F,Y) \to X$ whose fibres are $\mathit{Fr}(F,Y)_x = \mathit{Iso}(F,Y_x)$.\\
For the details we refer the reader to \cite{May2}.\\
\end{rem}

\begin{thm}
Simplicial univalence holds, i.e. for the universal fibration $\widetilde U_{\alpha} \to U_{\alpha}$, the canonical map $U_{\alpha}^{\Delta[1]} \to \mathit{Eq}(\widetilde U_{\alpha})$ is a weak equivalence.
\end{thm}

\begin{proof}
The idea of the proof follows the one of the classification of fibration in simplicial sets.\\
By theorem \ref{minimalfibration} we extract from the universal fibration a retract which is a minimal fibration $\pi : M \to U_{\alpha}$. Every fibre bundle is a fibration so that it is a pullback of the universal fibration $p$, hence by a simple diagram chase in the pullback square we have that it is also a pullback of $\pi$, i.e. $\pi$ is universal for fibre bundles.

$$\xymatrix@1{
& Y \ar[r] \ar[d]
& \widetilde{U}_{\alpha} \ar[r]_{r} \ar[d]_{p_{\alpha}} 
& M \ar[dl]^{\pi} \ar@/_1pc/[l]_{i} \\
& X \ar[r]
& U_{\alpha}
}$$

\nind
Observe that every fibred weak equivalence $E \to E'$ over a base $X$ induces a fibred weak equivalence $\mathit{Eq}(E) \to \mathit{Eq}(E')$ over $X \times X$.\\
Moreover, it follows from theorem \ref{minimalEqIso} that for a minimal fibration $M \to X$, there is a fibred weak equivalence $\mathit{Eq}(M) \to \mathit{Iso}(M)$ over $X \times X$.\\
We want to prove that $U^{\Delta[1]} \to \mathit{Eq}(\widetilde U)$ is a weak equivalence, so we compose with other weak equivalences and prove that the composite is a weak equivalence:
$$U^{\Delta[1]} \to \mathit{Eq}(\widetilde U) \mathop{\longrightarrow}^{\simeq} \mathit{Eq}(M) \mathop{\longrightarrow}^{\simeq} \mathit{Iso}(M)$$
\nind
We suppose without loss of generality that $U$ is connected, otherwise we can reproduce the same argument for each connected component.\\
Let $F$ be the fibre over a fixed point $u_0 \in U$ and consider the principal bundle associated to $\pi$, which is the frame bundle $\mathit{Fr}(F,M) \to U$. By the universality of $\pi$ this principal bundle is a universal principal $Aut(F)$-bundle, i.e. $B$ is a  $\textbf{B} Aut(F)$.\\
Now we apply the long exact sequence of a fibration:

$$\dots \to \pi_n(\textbf E G , *) \to \pi_n(\textbf B G , *) \to \pi_{n-1}(G , *) \to \pi_{n-1}(\textbf E G , *) \to \dots$$

\nind
Using the weak contractibility of $\textbf E G$ we have $\pi_n(\textbf B G , *) \cong \pi_{n-1}(G , *)$ and recalling the construction of higher homotopy groups of spheres as $\pi_0$ of iterated loop spaces we have that $\pi_{n-1}(\Omega (\textbf B G) , *) \cong \pi_{n-1}(G , *)$. Note that these two are simplicial groups, then by Moore's theorem \ref{Moore} they are Kan complexes. Since they are fibrant and cofibrant objects this weak equivalence is a homotopy equivalence by Whitehead's theorem \ref{White}.\\
Then we have $\Omega (U , u_0) \simeq Aut(F)$. Now let consider the diagram:

$$\xymatrix@1{
& U^{\Delta[1]} \ar[rr] \ar[dr]
& 
& \mathit{Iso}(M) \ar[dl] \\
& 
& U \times U
& 
}$$

\nind
Pulling it back along the inclusion $\{ u_0 \} \times U \to U \times U$ it is easy to check that we get the diagram:

$$\xymatrix@1{
& P(U , u_0) \ar[rr] \ar[dr]
& 
& \mathit{Fr}(F,M) \ar[dl] \\
& 
& B
& 
}$$

\nind
where $P(U , b_0)$ is the pointed path space. Recall that by theorem \ref{fibrefibration}, for a map between fibrations over a connected base to be a homotopy equivalence, it is enough that the induced map between the fibres over just one base point is a homotopy equivalence, and this is indeed the case by what we have proved before. Since the fibres of the two projections over a point are the same before and after pulling back the diagram, we can conclude as well that $U^{\Delta[1]} \to \mathit{Iso}(M)$ is a weak equivalence.
\end{proof}

\nind
We can finally notice that simplicial univalence is easy to obtain from the classification result for fibrations and from the theory of minimal fibrations. The difficulties in treating univalence are mainly given by coherence conditions.\\
Observe that this proof may suggest that a universal fibration can give rise to a more direct interpretation of type theory, if the usual coherence issues for pullbacks can be solved. In particular the uniqueness up to homotopy of the classification theorem entails that we automatically have coherence up to equivalence, which corresponds to equivalence of types.\\
This is just one hint for a possible motivation for weak Tarski universes, the difficulty of explaining in full detail motivations for weak Tarski universes is due to the mix of coherence conditions: we still want coherence theorems for substitution, whereas we want to weaken the ones for the computation rules for the universe. As the reader have noticed in the previous pages the strategy used to manage substitution and the various type constructors uses a strongly inaccessible cardinal in order to build a suitable category-theoretic universe, adding a second inaccessible to the same construction solves automatically all the coherence conditions for the universe.\\
We underline that the forthcoming paper \cite{LW} should solve all the issues for substitution and type constructors on Grothendieck $(\infty , 1 )$-topoi, except for the coherence conditions for the universe, whose weakening is the main topic of this thesis.\\
For this reason a detour into the subject of $(\infty , 1 )$-topoi would be needed to adequately motivate weak Tarski universes, but the treatment of these notions is beyond the aim of this thesis. 

The coherence theorems quoted in this chapter have been generalised to a certain extent in \cite{Shu1}, \cite{Shu2} and \cite{Cis}, to cover the cases of some well-behaved simplicial presheaves and the one of cubical sets.

In conclusion, we underline that Voevodsky's proof takes place in a classical metatheory and makes heavy use of inaccessible cardinals and the axiom of choice. The search for a consistency proof obtained by means of finitistic methods, like the ones given by normalisation theorems, is still an open problem.

\chapter{Constructive Set Theory}

\lettrine{T}{he} intuitionistic and constructivist foundational program is not confined to logic. The development of intuitionistic logic is just the first step for a thoroughly rethinking of mathematics. \\
Type theory and category theory are deeply entangled with intuitionism, but also set theory can be rebuilt constructively. One of the advantage of this enterprise is to rely on the already familiar language of set theory, making constructive reasoning accessible to a wider audience. Moreover, sets and membership between sets are two of the most basic and fundamental concepts, then they doubtlessly deserve to be subject to a constructive analysis and clarification.\\
The sources for this chapter are the three articles by Peter Aczel on the type-theoretic interpretation \cite{Acz1}, \cite{Acz2},\cite{Acz3}, and the book draft on constructive set theory \cite{AR}.\\

An obvious requirement for an intuitionistic theory of sets is to have an intuitionistic underlying logic. It is immediate to see that this is not enough. Indeed, we have the following:

\begin{thm}
Consider the core theory given by the axioms of extensionality, separation, emptyset, and pair with intuitionistic logic.\\
The axiom of foundation in the form of a minimal $\in$-element: $\forall x \, [\exists u \in x \Rightarrow \exists y \, (y \in x  \land \forall z \, (z \in y \Rightarrow z \in x)) ]$, implies the law of excluded middle.\\
Moreover, if separation is limited to bounded formulae, excluded middle for bounded formulae can be derived.
\end{thm}

\begin{proof}
The idea is a common trick: we form by emptyset and pair the set $2:=\{ \void, \{ \void \} \}$ and by separation a suitable subset made with a given formula $\phi$ in a way such that this formulation of the foundation axiom gives directly the excluded middle.\\
Let define $A: =\{ y \in 2 \, | \, y = \{ \void \}  \lor (y= \void \land \phi) \}$, it is nonempty because $\{ \void \} \in A$, so that by foundations it admits a minimal element $Y$ which as an element of $A$ must be $Y= \{ \void \}$ or $Y= \void \land \phi$. If $Y= \{ \void \}$, then $ \void \in Y$ and by minimality $\void \notin A$, hence by definition $\neg (\void = \void \land \phi)$ therefore $\neg \phi$.\\
On the other hand if $Y= \void \land \phi$, we have trivially $\phi$.
\end{proof}

Therefore, a careful analysis of the various axioms is needed together with a search of constructive substitutions for non-constructive principles. Following Aczel's work we will not confine ourselves with the little changes needed to keep the excluded middle out of the theory, but we will seek a predicative theory, at least to a certain reasonable extent.\\

The language of CZF is a first order language with the following primitive logical symbols $\bot, \land, \lor, \Rightarrow \forall x, \exists x$, with the restricted quantifiers $\forall x \in y$ and $\exists x \in y$, and two nonlogical relational symbols $\in$ and $=$. As we have said before the underlying logic is intuitionistic.\\
some basic axioms are retained from ZF, as the defining schemes for the restricted quantifiers: 
\begin{itemize}
\item $(\forall x \in y) \, \phi(x) \Leftrightarrow \forall x \, (x \in y \Rightarrow \phi(x))$;

\item $(\exists x \in y) \, \phi (x) \Leftrightarrow \exists x \, (x \in y \land \phi(x))$.
\end{itemize}
We also have the usual equality axioms:
\begin{itemize}
\item $x = y \Leftrightarrow \forall z \, (z \in x \Leftrightarrow z \in y)$;

\item $x = y \land y \in z \Rightarrow x \in z$.
\end{itemize}
And finally pairing and union:
\begin{itemize}
\item pairing: $\exists z \, (x \in z \land y \in z)$;

\item union: $\exists z \, (\forall y \in x)(\forall u \in y)(u \in z)$.
\end{itemize}
The first step for a predicative theory is to restrict the separation axiom to restricted formulae:
\begin{itemize}
\item restricted separation: for each restricted $\phi$ we have $\exists z \, [(\forall y \in z) (y \in x \land \phi(y)) \land (\forall y \in x )(\phi(y) \Rightarrow y \in z)]$.
\end{itemize}
Moreover, a weakening of the powerset axiom is also needed; we still want to be able to form the set of functions between two arbitrary sets (the \textit{exponentiation axiom}\index{axiom!exponentiation}), which is enough for the construction of the real numbers based on Cauchy sequences, but is too weak for other constructions as Dedekind's one which needs to manage relations. For these kinds of reasons we will introduce a strengthening of the exponentiation axiom, namely the \textit{subset collection axiom}\index{axiom!subset collection}. We present here the formal statement of the axiom, but its explanation is not immediate so we postpone its informal presentation to the sequel.\\
For the sake of readability we introduce the following notation: given a formula $\phi(x,y)$ we define $\phi'(a,b)$ as $(\forall x \in a)(\exists y \in b) \, \phi(x,y) \land (\forall y \in b)(\exists x \in a) \, \phi(x,y)$;
\begin{itemize}
\item subset collection: $\exists c \, \forall u \, [(\forall x \in a)(\exists y \in b) \, \phi(x,y) \Rightarrow (\exists d \in c)\, \phi'(a,d)$, where $u$ may occour free in $\phi(x,y)$.
\end{itemize}
The axiom of foundation is replaced with its classically equivalent \textit{set induction axiom}\index{axiom!set induction} which allows to prove property for all sets, provided that it can be proved with the induction hypothesis that all the elements of the set in consideration already satisfy that property. This axiom informally states that sets are built inductively from previously constructed sets.
\begin{itemize}
\item set induction scheme: $\forall y \, [(\forall x \in y) \, \phi(x) \Rightarrow \phi(y)] \Rightarrow \forall x \, \phi (x)$.
\end{itemize}
In order to compensate the predicative weakening we strengthen some other axioms: the axioms of replacement is substituted with the \textit{strong collection axiom}\index{axiom!strong collection} which can be seen as a version of replacement for relations; it is immediate to see that strong collection implies replacement;
\begin{itemize}
\item strong collection: $(\forall x \in a) \exists y \, \phi(x,y) \Rightarrow \exists b \, \phi'(a,b)$.
\end{itemize}
Finally the infinity axiom is replaced with the \textit{strong infinity axiom}\index{axiom!strong infinity} which states directly the existence of the least inductive set, instead of constructing it impredicatively as the intersection of all inductive sets:
\begin{itemize}
\item let $\mathit{Zero}(x)$\index{$\mathit{Zero}(x)$} be the formula $(\forall y \in x ) \bot$, and $\mathit{Succ}(x,y)$\index{$\mathit{Succ}(x,y)$} be the formula $(\forall z \in y)(z \in x) \land (y \in x) \land (\forall z \in x)(z \in y \lor z = y)$. The the axiom of infinity can be stated as $\exists z \, \mathit{Nat}(z)$, where $\mathit{Nat}(z)$\index{$\mathit{Nat}(z)$} is the conjunction of $(\forall x \in z)(\mathit{Zero}(x) \lor (\exists y \in z) \, \mathit{Succ}(y,z))$ with $(\exists x \in z) \, \mathit{Zero}(x)$ and $(\forall y \in z)(\exists x \in z) \, \mathit{Succ}(y,x)$.
\end{itemize}
The presentation of some of these axioms is not canonical, for example the pair axiom here states simply the existence of a set having as elements the two constituents of the pair. Anyway with restricted separation we can easily recover the usual formulation, and similarly for the other cases. These choices are made with the aim to simplify the proof of the type-theoretic interpretation.\\
In the following tabular we summarize the basic axioms:

\begin{center}
\begin{tabular}{|c|c|c|}
\hline
ZF & CZF & changes \\
\hline
\hline
extensionality & extensionality & none \\
\hline
pair & pair & none \\
\hline
union & union & none \\
\hline
separation & bounded separation & weakened \\
\hline
powerset & subset collection & weakened \\
\hline
foundation & set induction & reformulated\\
\hline
replacement & strong collection & strengthened\\
\hline
infinity & strong infinity & strengthened\\
\hline
\end{tabular}
\end{center}

\nind
Set induction is a weakening or a strengthening depending on the classical variant chosen for the foundation axiom.

The basic classical set-theoretic constructions like, pairs, products, disjoint unions and so on can be easily performed in CZF in a straightforward way, sometimes minor changes are necessary to avoid non-constructive arguments. We discuss pairs because they give a simple example of constructive reformulation of a classical proof. As we have noticed before, applying bounded separation to the set given by our form of pairing, we get the usual $\exists y \, \forall x \, (x \in y \Leftrightarrow x = a \lor x=b)$. This set is unique by extensionality. We form the usual pair $ ( a,b ) := \{ a, \{ a, b \}  \}$ and prove the following:

\begin{lem}
If $ ( a,b ) = ( c,d )$, then $a = c$ and $b = d$.
\end{lem}

\begin{proof}
The usual classical proof uses the excluded middle to say that $a = c$ or not, and reason by cases. Instead we will argue focusing on the elements: $\{ a \}$ is an element of the left-hand set, so it is also an element of $( c,d )$, then either $\{ a \} = \{ c \}$ or $\{ a \} = \{ c,d\}$, and in either cases $a=c$. Similarly $\{ a,b \} = \{  c\}$ or $\{ a,b \} =  \{ c,d\}$. In either cases $b = c$ or $b = d$. If $b = c$, then $a = c = b$ so that $( a,b ) = \{ \{ a \} \}$ and hence $( c,d )$ must have a single element as well, so that $c=d$.
\end{proof}

Now we give some definitions about relations, and give an equivalent, version of subset collection.

\begin{defn}
\begin{enumerate}[label=(\alph*)]
\item[]

\item For a relation $R \subseteq A \times B$, we write $A \righttail B $\index{$\righttail \,$} iff $R$ is total, i.e. $(\forall x \in A)(\exists y \in B) \, ( x,y ) \in R$;

\item similarly, $R: A \leftrighttail B$\index{$\leftrighttail \,$} iff $(\forall x \in A)(\exists y \in B) \, ( x,y ) \in R \land (\forall y \in B)(\exists x \in A) \, (x,y) \in R$;

\item a set $C$ of subsets of $B$ is \textbf{$A$-full}\index{$A$-full set of subsets} iff $R: A \righttail B$ implies that $R: A \leftrighttail D $ for some $D \in C$.
\end{enumerate}
\end{defn}

\begin{thm}
The subset collection scheme is equivalent to the following axiom: for all sets $A$ and $B$, there exists a set $C$ that is an $A$-full set of subsets of $B$.
\end{thm}

\begin{proof}
This axiom is a special case of the subset collection scheme where $\phi(x,y)$ is chosen to be $( x,y ) \in R$. For the converse we will combine the axiom above with strong collection: let $C$ be an $A$-full set of subsets of $B$ and suppose that $(\forall x \in A) (\exists y \in B) \, \phi(x,y)$, we show that $\phi'(A,D)$ for some $D \in C$. Let $\psi(x,z)$ denote the formula $ (\exists y \in B)(\phi(x,y) \land ( x,y ) = z)$. Then by definition $(\forall x \in A) \exists z \, \psi(x,z)$, so that by strong collection there is a set $R$ such that $(\forall x \in A) (\exists z \in R) \, \psi(x,z) \land (\forall z \in R)(\exists x \in A ) \, \psi(x,z)$. Hence $R : A \righttail B$ and $\forall x \, \forall y \, (( x,y ) \in R \Rightarrow \phi(x,y))$. As $C$ is an $A$- full set of subsets of $B$ we can find a $D$ such that $R: A \leftrighttail B$. It follows that $\phi'(A,D)$.
\end{proof}

\begin{thm}
We have the following chain of implications: powerset $\Rightarrow$ subset collection $\Rightarrow$ exponentiation axiom.
\end{thm}

\begin{proof}
The first implication is trivial. For the second one we consider besides every function $f:A \to B$, the associated $f': A \to A \times B$ defined by $f'(x) := ( x, f(x) )$, and then we apply the axiom above to $A \times B$, in order to obtain $C$, an $A$-full set of subsets of $A \times B$. Now we want to ensure that any $f$ is in $C$, so that applying bounded separation we can form the set of functions. $f'$ is a total relation, therefore we can find a set $D \in C$ such that $f':  A \leftrighttail D$. As $f'$ is a function $D = \{ f'(x) \, | \, x \in A \} = f$, so that $f \in C$, as required.
\end{proof}
\nind
The two previous theorems show that the axiom of subset collection can be conceived as a strengthening of the exponentiation axiom that allows to manage total relations.

\begin{thm}
The exponentiation axiom, together with the statement that $\{ \emptyset\}$ has a powerset, is equivalent to the full powerset axiom.
\end{thm}

\begin{proof}
One direction is trivial, for the other the idea is that the set $B=\mathcal P(\{ \emptyset \})$ allows to define characteristic functions and therefore the set of all subsets.\\
For any set $A$ let $C := \{ \{ x \in A \, | \, \emptyset \in f(x) \} \, | \, f \in B^A \}$. This is a set by the exponentiation axiom, restricted separation and replacement. If $z \subseteq A$, let $f(x)$ be the set $\{ y \in \{ \emptyset \} \, | \, x \in z \}$ for $x \in A$. Then $f \in B^A$ and hence $z = \{ x \in A \, | \, \emptyset \in f(x) \} \in C$. Hence $C$ is the powerset of $A$.
\end{proof}
\nind
This theorem shows that in CZF, finite sets cannot have powersets. It may seem strange that the sets that are most accessible to our intuition lack powersets; especially because  finite sets provided to Cantor the intuitive background for the introduction of powersets for arbitrary infinite sets, and therefore full impredicativity. Probably, the best justification of this feature is again the type-theoretic interpretation of CZF; indeed, in Martin-L\"of type theory function types are available and the types for finite sets have fully justified rules which provide only canonical elements, and ways to introduce and eliminate them.\\
Now we want to prove that CZF with classical logic and ZF prove the same theorems, we start with the following two lemmas:

\begin{lem}\label{powlem}
The exponentiation axiom and restricted excluded middle, together imply the powerset axiom.
\end{lem}

\begin{proof}
By the previous theorem it is enough to prove that $\{ \void \}$ has a powerset. In fact we show that $\{ \void , \{ \void \} \}$ is its powerset. So let $x \subseteq \{ \void ,\{ \void \} \}$, by hypothesis $\void \in x \lor \void  \notin x$. In the first case $x = \{ \void \}$. Whereas in the second $x=\void$. In either case $x \in \{ \void , \{ \void \} \}$.
\end{proof}

\begin{lem}
The full separation scheme is equivalent to the scheme: $\exists x \, (\phi \Leftrightarrow \void \in x)$, where $x$ is not free in $\phi$.
\end{lem}

\begin{proof}
Given full separation and a formula $\phi$ let $x:=\{ y \in \{ \void \} \, | \, \phi \}$ where $y$ is not free in $\phi$. Then $\void \in x \Leftrightarrow \phi$.\\
Conversely, the idea is to build a function from the hypothesis by strong collection such that $\phi(y)$ is equivalent to $\void \in f(y)$ and then to apply restricted separation. Indeed, by assumption there is a set $x$ such that $\phi(y) \Leftrightarrow \void \in x$, for each $y \in A$. We may assume that $x \subseteq \{ \void \}$ in which case $x$ is uniquely determined by $y \in A$. By strong collection there is a function $f$ with domain $A$ such that $(\forall y \in A)(\phi(y) \Leftrightarrow \void \in f(y))$. By restricted separation we can form the set $\{ y \in A \, | \, \void \in f(y) \} = \{ y \in A \, | \, \phi(y) \}$.
\end{proof}

\begin{thm}
CZF with classical logic and ZF prove the same theorems.
\end{thm}

\begin{proof}
Clearly ZF contains all the theorems proved by CZF with classical logic. Conversely, by \ref{powlem} the powerset axiom holds. Full separation is also a theorem because for every formula $\phi$ we have $\phi \lor \neg \phi$, so choose $x=\{ \void \}$ if $\phi$ and $x=\void$ if $\neg \phi$. In either case $\phi \Leftrightarrow \void \in x$ and we get full separation by the previous lemma.
\end{proof}

Let $\omega$ be the unique set such that $\mathit{Nat}(x)$, which existence is given by the strong infinity axiom. As usual define $0:= \void$ and $x^+ := x \cup \{ x \}$ for every $x \in \omega$, which give rise to the successor function $s:\omega \to \omega$, defined by $s(x):=x^+$.

\begin{thm}
The structure $\mathbb N := (\omega, 0,s)$, satisfies the Peano axioms.
\end{thm}

\begin{proof}
The first two axioms - that $0$ is a number and that the successor of a number is again a number - hold by definition. The third axiom says that $0$ is not in the image of the successor function; indeed, $\void$ has no elements. Using set induction we easily obtain mathematical induction. \\
In order to prove the injectivity of $s$ we observe that every $x \in \omega$, $x$ is transitive (i.e. $(\forall y \in x) \, y \subseteq x$), and such that $x \notin x$. These properties follow easily by induction.\\
Let $x,y \in \omega$ be such that $x^+ = y^+$. As $x \in x^+$ we get $x \in y^+$, so that either $x \in y$ or $x = y$, and similarly with the roles of $x$ and $y$ shifted. If $x \in y$ then $x \in x$ for the transitivity, which is absurd. The only remaining possibility is that $x = y$.
\end{proof}

\begin{defn}
Given a set $B$ the set $C$ is called the \textbf{transitive closure}\index{transitive closure} iff $B \subseteq C$, $C$ is transitive and for every transitive set $X$ such that $B \subseteq X$, then $C \subseteq X$.
\end{defn}

\begin{thm}\label{trans}
In CZF every set has a transitive closure.
\end{thm}

\begin{proof}
Let $B$ be any set. We can then form the sets defined inductively $h(0) := B$ and $h(n+1) := h(n) \cup \bigcup h(n)$, let now $ C := \bigcup_{n \in \mathbb N} h(n)$. As $B=h(0)$ we have $B \subseteq C$, and given $x \in y \in C$, by definition $y \in h(n)$ for some $n$, thus $x \subseteq \bigcup h(n) \subseteq h(n+1) \subseteq C$, and hence $x \in C$. Finally, suppose that $ B \subseteq D$, where $D$ is a transitive set. By induction on $n$ one readily establish that $h(n) \subseteq D$, whence $C \subseteq D$.
\end{proof}

Let now discuss choice principles. It can be shown that the full axiom of choice implies unacceptable instances of the excluded middle. However, as we have seen in \ref{choice} Martin-L\"of theory validates the type-theoretic axiom of choice almost by definition.\\
A choice principle which can be justified constructively is the \textit{axiom of dependent choices}\index{axiom!of dependent choices}. It is usually stated in the following form: for every binary total relation $\rho$ on a given set $X$ (i.e. $(\forall x \in X) (\exists y \in X) $ such that $x \, \rho \, y$), there exists a sequence $\{ x_n \}_{n \in \omega}$ such that $x_n \, \rho \, x_{n+1}$ for every $n \in \omega$. It allows to define sequences of elements whose choices depend over the previously chosen ones.\\
We are interested in the axiom \textit{scheme}\index{axiom!scheme of dependent choices} (DC for short)\index{DC} of dependent choices, where the relation $x \in X$ is replaced by an arbitrary formula with a free variable, and similarly the binary relation is replaced with an arbitrary formula with two free variables: for the sake of readability, given formulae $\theta (x)$ and $\phi(x,y)$, we call $\psi(x,z)$ the formula expressing that $z$ is a function, whose domain is $\omega$, such that $z(0)=x$ and for every natural number $n \in \omega$, $\theta(z(n)) \land \phi(z(n),z(n+1))$ holds.\\
The axiom scheme of dependent choices is then the following:

$$ \forall x \, (\theta(x) \Rightarrow \exists y \, (\theta(y) \land \phi(x,y))) \Rightarrow \forall x \, (\theta (x) \Rightarrow \exists z \, \psi(x,z)) $$

\nind
It is easy to see that it is implied by the full axiom of choice (in presence of full separation), and that in turn it implies the axiom of countable choice.\\
Dependent choice is sufficient to develop much of the usual choice-based classical mathematics, included large parts of functional analysis. Indeed, the form with total relations is equivalent to the Baire's lemma (see \cite{Gol} for the Bair lemma in complete metric spaces and \cite{Bla} for its version in locally compact Hausdorff spaces).\\
In \cite{Acz2}, the axiom scheme of dependent choices is considered as a way to extend CZF, justified by the type-theoretic interpretation where it holds for the terms of the type of sets. We emphasize that the the type-theoretic axiom of dependent choices and the interpretation inside the type of sets of the set-theoretic one, are not the same, although the former is used in the proof of the latter. In the next chapter we will see the generalisation of this proof to homotopy type theory with a weak Tarski universe.\\
In the previous mentioned article by Peter Aczel other choice principles are considered, a \textit{base}\index{base} is defined to be a set such that choice functions defined on it can always be found. The \textit{presentation axiom}\index{axiom!presentation} states that every set is the surjective image of a base. The intended meaning is that surjective maps from a base - presentations - correspond to the concrete ways in which sets are given us. This axiom is again justified by the type-theoretic interpretation. We have not considered this kind of choice principles because they are not stable under category-theoretic constructions like taking sheaves (see \cite{MvdB2}).\\

Finally, we discuss inductive definitions in CZF and the role of the other axiom which can be possibly added: the regular extension axiom. We start with a discussion on classes in CZF and inductively defined classes, then we introduce regular sets and the regular extension axiom and state the theorem asserting that every bounded inductive definition, inductively defines a set.\\
As in ZF, classes can be treated implicitly as formulae with a free variable:

\begin{defn}
\begin{enumerate}[label=(\alph*)]
\item[]

\item Given a set $x$ and a class $\Phi$ we write $x \in \Phi$ iff $\Phi(x)$;

\item $\Phi$ is a \textbf{sublcass}\index{subclass} of $\Psi$ iff $(\forall x \in \Phi ) \, x \in \Psi$, which means $\forall x \, (\Phi(x) \Rightarrow \Psi(x))$;

\item the \textbf{union of classes}\index{union of classes} is defined as the logical disjunction: $\Phi \lor \Psi$;

\item similarly, the \textbf{intersection of classes}\index{intersection of classes} is the conjunction: $\Phi \land \Psi$;

\item the \textbf{product of classes}\index{product of classes} is $(\Phi \times \Psi) (z) := \exists x \, \exists y \, (\Phi(x) \land \Psi(y) \land z= ( x,y ))$;

\item the \textbf{powerclass}\index{powerclass} is defined as $\mathcal P(\Phi)(x) := (\forall y \in x) \, \Phi(y)$;

\item the \textbf{class of sets}\index{class of sets}, written $V$, is $x=x$.
\end{enumerate}
\end{defn}

\nind
In the sequel we will use the notation $\{ x \, | \, \Phi(x) \}$ for the formula $\Phi$ interpreted as a class.\\

Proofs by transfinite recursion on ordinals are often used in the classical treatment of inductive definitions. But intuitionistic ordinals have a quite different behaviour, moreover accordingly to the constructive paradigm is generally preferable to build directly inductively defined classes. We start with the following definition:

\begin{defn}
For any class $\Phi$, the class $X$ is \textbf{$\Phi$-closed} iff $A \subseteq X$ implies $a \in X$ for every order pair $( a,A ) \in \Phi$.
\end{defn}

\nind
Being the first encounter with a class defined as a formula we unwind the definition: $X$ is $\Phi$-closed iff $\forall a \, \forall A \, [\Phi(( a,A )) \Rightarrow ((\forall z \in A) \, X(z) \Rightarrow X(a))]$.

\begin{thm}
For any class $\Phi$, there is a smallest $\Phi$-closed class $I(\Phi)$\index{$I(\Phi)$}.
\end{thm}

\begin{proof}
Refer to the theorem in section 4.2 of \cite{Acz2}.
\end{proof}

\nind
Usually a $\Phi$-closed class is defined by a system of rules, it is straightforward to extract the class involved from the rules. For example, the class of natural numbers can be characterised as the smallest class $\omega$ closed under the rules: $\void \in \omega$ and $a \cup \{ a \} \in \omega$ if $a \in \omega$. The class generating $\omega$ is $\Phi = \{ ( a \cup \{ a\} , \{ a\} ) \, | \, a \in V \}$.\\
Because of the previous theorem, classes are sometimes called \textit{inductive definitions}\index{inductive definition}, accordingly to the use of the expression in the natural language.\\
We now present the regular extension axiom and show its how it can be used to manage inductive definitions.

\begin{defn}
A class $A$ is \textbf{regular}\index{regular class} iff it is transitive, i.e. every element of $A$ is a subset of $A$; moreover, for every $a \in A$ and relation $R \subseteq a \times A$ if $(\forall x \in a) \, \exists y \, ( x,y ) \in R$, then there is a set $b \in A$ such that $$ (\forall x \in a ) (\exists y \in b) \, ( x,y ) \in R \land (\forall y \in b)(\exists x \in a) \, ( x,y ) \in R$$
\end{defn}

\nind
In particular if $R: a \to A$ then ran$R \in A$.\\
For example it is easy to check that $\omega$ is a regular set.\\
The \textbf{regular extension axiom}\index{axiom!regular extension} (or REA\index{REA} for short) states that:
\begin{itemize}
\item  every set is a subset of a regular set.
\end{itemize}
As we have seen in the previous theorem the class $I(\Phi)$ always exists, we expect that $I(\Phi)$ is a set whenever $\Phi$ is a set, and in certain well-behaved cases when $\Phi$ is a proper class. This is the case, but in CZF, even when $\Phi$ is a set, the regular extension axiom is needed.

\begin{defn}
An inductive definition $\Phi$ is \textbf{bounded}\index{bounded inductive definition} iff 
\begin{enumerate}[label=(\roman*)]

\item for each set $A$ the class $\Phi_A := \{ x \, | \, ( x,A ) \in \Phi\}$ is a set, it is a condition of smallness on the fibres;

\item there is a set $B$ such that if $( a, A ) \in \Phi$, then $A$ is an image of a set in $B$, it is a condition of smallness for the elements in the image of $\Phi$. The set $B$ is called a \textbf{bound}\index{bound of an inductive definition} for $\Phi$.
\end{enumerate}
\end{defn}

\nind
Observe that if $\Phi$ is a set then it is automatically bounded with bound the image: $\{ A \, | \, \exists a \, ( a,A ) \in \Phi \}$. The class of natural numbers $\Phi=\{ ( \void , \void ) \} \cup \{ ( a \cup \{a \} , \{a\} ) \, | \, a \in V \}$ is bounded with bound $\{ \void , \{ \void \} \}$.

\begin{thm}
In CZF+REA, every bounded inductive definition, inductively defines a set.
\end{thm}

\begin{proof}
See theorem 5.2 in \cite{Acz3}.
\end{proof}

\chapter{Constructive Set Theory from a Weak Tarski Universe}

\lettrine{A}{s} we have seen in the previous chapters the type theoretic interpretation of CZF is the main conceptual justification of this kind of constructive set theory. In this chapter we shall generalise the standard interpretation to homotopy type theory with a weak Tarski universe following the three articles by Peter Aczel \cite{Acz1}, and selected parts of \cite{Acz2} and \cite{Acz3}.\\
The generalisation has two main issues: 
firstly we need to reformulate statements and proofs when the Russell style universe is replaced by a Tarski one, and secondly to adapt every definitional equality given by the computation rule of the universe to an equivalence and make the necessary lemmas work with equivalences. The first issue is straightforward although tedious, whereas the second needs some actual rethinking especially in some points like in the discussion about dependent choices and mainly in lemma \ref{mainlem} which uses the function extensionality axiom, and seems needed for the restricted separation axiom.\\
We have to remark that some additional reformulations were needed because the kind of type theory used in \cite{Acz2} (and subsumed by \cite{Acz3}) has extensional identity types and a closed universe, i.e. an elimination rule for the universe stating a recursion principle over type constructors.\\
In this chapter the set-theoretic symbol for equality will be denoted as $\doteq$\index{$\doteq$}.

The idea behind the type-theoretical interpretation is to organize all small types in a well-founded tree as a $W$-type building a type of iterative sets, analogous to the usual cumulative hierarchy. Then, to define a notion of extensional equality and use the type constructors to induce similar operations inside this type of sets.

\begin{defn}
The \textbf{type of iterative sets}\index{type!of iterative sets} is defined as $V := ( W x:U) El(x)$\index{$V$} where $U$ is the universe.
\end{defn}

\nind
Therefore we have the following rules which can be used as an alternative direct definition of the type of sets:

\begin{prooftree}
\AxiomC{$a:U$}
\AxiomC{$b:El(a) \rightarrow V$}
\BinaryInfC{$\textsf{sup}(a,b) : V$}
\end{prooftree}

\begin{prooftree}
\AxiomC{$c:V$}
\AxiomC{$x:U, y:El(x) \to V, z: (\Pi v:El(x)) C(y(v)) \vdash d(x,y,z) : C(\textsf{sup}(x,y))$}
\BinaryInfC{$\textsf T(c, d(x,y,z)) : C(c)$}
\end{prooftree}

\begin{prooftree}
\AxiomC{$a:U$}
\AxiomC{$f:El(a) \to V$}
\AxiomC{$x:U, y:El(x) \to V, z: (\Pi v:El(x)) C(y(v)) \vdash d(x,y,z) : C(\textsf{sup}(x,y))$}
\TrinaryInfC{$\textsf T(\textsf{sup}(a,f), d(x,y,z))=d(a,f, \lambda v.T(f(v), d(x,y,z))) : C(c)$}
\end{prooftree}

\nind
Where for the sake of readability we have written $f(v)$ instead of $\textsf{ap}(f,v)$. The elimination and the computation rules for this type simply express transfinite recursion over the cumulative hierarchy.

\begin{lem}
There is a function assigning $\overbar \alpha : U$\index{$\overbar \alpha$} and $\tilde{\alpha} : El(\overbar \alpha) \to V$\index{$\tilde{\alpha}$}, to any $\alpha : V$. Moreover, if $\alpha = \textsf{sup}(a,f)$, then $\overbar \alpha = a$ and $\tilde{\alpha} =f$.
\end{lem}

\begin{proof}
We define $\tau : V \to (\Sigma x:U)(El(x) \to V)$ by transfinite recursion on $V$ using the elimination rule for $W$-types: let $C$ be the constant family of types $(\Sigma x:U)(El(x) \to V)$, hence it suffices to derive $C(\textsf{sup}(x,y))$ from $x:U$ and $y:El(x) \to V$. We give the following definition: $\tau (\textsf{sup}(a,f))=(a,f) : (\Sigma x:U)(El(x) \to V)$. Now let $\overbar \alpha = \textsf{p}(\tau(\alpha))$ and $\tilde{\alpha} = \textsf{q}(\tau(\alpha))$.
\end{proof}

\nind
We may think at these $\overbar \alpha : U$ and $\tilde{\alpha}:El(\overbar \alpha) \to V$ as a presentation of the iterative set $\alpha$ as the supremum of the image of the function $\tilde{\alpha}$.

\begin{defn}
Let $\mathcal L$ be the language of set theory and $\mathcal L_V$\index{$\mathcal L_V$} the language obtained by adding to $\mathcal L$ a constant for each term $\alpha : V$.
\end{defn}

\nind
Now we give the definition of interpretation of a set theoretic formula in type theory. In order to do this we need to define a notion of extensional equality in $V$.

\begin{defn}
We define the \textbf{extensional equality}\index{equality!extensional}\index{extensional!equality} as a bisimulation relation by double transfinite recursion on the canonical elements of $V$. Explicitly $( \textsf{sup}(a,f) \doteq \textsf{sup}(b,g) )$ is defined to be:
\begin{center}
$(\Pi x:El(a)) (\Sigma y:El(b)) ( f(x) \doteq g(y) ) \times (\Pi y: El(b)) (\Sigma x:El(a)) ( f(x) \doteq g(y) )$\index{$\doteq$}
\end{center}

\nind
We can unwind this double transfinite recursion as the iteration of two simple recursions. We firstly define: $$F(\textsf{sup}(a,f)) := \lambda \beta . (\Pi x:El(a)) (\Sigma y:El(\overbar \beta )) F(f(x)) \times (\Pi y: El(\overbar \beta)) (\Sigma x:El(a)) F(f(x))(\tilde{\beta}(y))$$
and then $(\alpha \doteq \beta) := F(\alpha)(\beta)$.
\end{defn}

\begin{lem}
The extensional equality is an equality, i.e. for all $\alpha, \beta, \gamma : V$ we have a term inside the following types:
\begin{enumerate}[label=(\roman*)]

\item $\alpha \doteq \alpha$;

\item $(\alpha \doteq \beta) \to (\beta \doteq \alpha)$;

\item $(\alpha \doteq \beta) \times (\beta \doteq \gamma) \to (\alpha \doteq \gamma)$.
\end{enumerate}
\end{lem}

\begin{proof}
We discuss the first, the other two require respectively a double and a triple transfinite recursion. We want to find a term in the type: $$(\textsf{sup}(a,f) \doteq \textsf{sup}(a,f))= (\Pi x : El(a))(\Sigma y:El(a))(f(x) \doteq f(y)) \times (\Pi y:El(a))(\Sigma x:El(a))(f(x) \doteq f(y))$$
a proof of this type is constructed by a single transfinite recursion starting from a proof $d(x): (f(x) \doteq f(x))$, and is given by $r_0=(z_1,z_2)$ where $z_1 = \lambda x.(x,d(x)) (\Pi x:El(a))(\Sigma x:El(a))(f(x) \doteq f(x))$, and similarly for $z_2$.
\end{proof}

\begin{rem}\label{mainrem}
In order to manage the weakening of the universe we shall use often the principle of indiscernibility of identicals.\\
By the elimination rule for the identity types we have that a proof of the identity $p:Id_A(x,y)$ induces a proof of the extensional equality $f(x) \doteq f(y)$, for every $f:A \to V$.\\
Indeed, we define a family of types $x,y:A,p:Id_A(x,y) \vdash C(x,y,p)=(f(x) \doteq f(y))$. By the previous theorem we have a proof $c=\lambda z.r_0(z) : (\Pi z:A)(f(z) \doteq f(z))$, therefore by the elimination rule we obtain $h: (\Pi x,y:A)(\Pi p:Id_A(x,y))(f(x) \doteq f(y))$ such that $h(z,z,\textsf{refl}_z)=c(z)$.
\end{rem}

\begin{rem}
For any $\alpha : V$ we have a form of extensional canonicity: $\alpha \doteq \textsf{sup}(\overbar \alpha, \tilde{\alpha})$.\\
Indeed, let $g=\lambda x.\textsf{sup}(\overbar x, \tilde{x}): V \to V$, by construction $g(\alpha)= \alpha$ for all $\alpha = \textsf{sup}(a,f)$, then we have a canonical proof of the identity $\textsf{refl} : Id_V(g(\textsf{sup}(a,f)), \textsf{sup}(a,f))$, hence by transfinite recursion over $V$ we get an identity term defined on all $V$, namely $T(\alpha, \lambda x,y,z. \textsf{refl}(x,y,z)) : Id_V(g(\alpha), \alpha)$, that induces a proof of the extensional equality $g(\alpha) \doteq \alpha$, by the previous remark.
\end{rem}

\begin{defn}
We define the \textbf{type-theoretic interpretation}\index{type-theoretic interpretation of CZF} recursively on the structure of the formulae in $\mathcal L_V$:
\begin{enumerate}

\item $\llbracket \alpha \doteq \beta \rrbracket = (\alpha \doteq \beta)$;

\item $\llbracket \alpha \in \beta \rrbracket = (\Sigma y:El(\overbar \beta)) \llbracket \alpha \doteq \tilde{\beta}(y) \rrbracket$;

\item $\llbracket \bot \rrbracket = \mathbf{N_0}$;

\item $\llbracket \phi \Rightarrow \psi \rrbracket = \llbracket \phi \rrbracket \to \llbracket \psi \rrbracket$;

\item $\llbracket \phi \land \psi \rrbracket = \llbracket \phi \rrbracket \times \llbracket \psi \rrbracket$;

\item $\llbracket \phi \lor \psi \rrbracket = \llbracket \phi \rrbracket + \llbracket \psi \rrbracket$;

\item $\llbracket (\forall x \in \alpha)\phi(x) \rrbracket = (\Pi x: El(\overbar \alpha)) \llbracket \phi(\tilde{\alpha}(x)) \rrbracket$;

\item $\llbracket (\exists x \in \alpha) \phi(x) \rrbracket = (\Sigma x:El(\overbar \alpha)) \llbracket \phi(\tilde{\alpha}(x)) \rrbracket$;

\item $\llbracket \forall x \, \phi(x) \rrbracket = (\Pi \alpha:V) \llbracket \phi(\alpha) \rrbracket$;

\item $\llbracket \exists x \, \phi(x) \rrbracket = (\Sigma \alpha : V) \llbracket \phi(\alpha) \rrbracket$.
\end{enumerate}
\end{defn}

\begin{defn}
We say that a set theoretic formula $\phi(x_1, \dots , x_n)$ is \textbf{valid}\index{valid} in the interpretation iff the type of the interpretation of its universal closure $\llbracket \forall x_1, \dots, x_n \, \phi(x_1, \dots , x_n) \rrbracket$ is inhabited. 
\end{defn}

\nind
In order to prove our theorems we will perform the construction of a term in the type $\llbracket \phi(x_1, \dots , x_n) \rrbracket$ leaving the last step of $\lambda$-abstraction always implicit.

The first step for the construction of a model of CZF in type theory is to take care of the underlying logic, i.e. the Curry-Howard correspondence.

\begin{thm}
If $\phi_1, \dots \phi_n \vdash \phi$ in intuitionistic predicate logic and $\phi_1, \dots \phi_n$ are valid in the interpretation then so is $\phi$.
\end{thm}

\begin{proof}
See \cite{ML1} where it is shown without any use of the universe that the rules for intuitionistic natural deduction are particular cases of the rules for type-theoretic constructors.
\end{proof}

\nind
The reformulation of the following lemma is not easy as the others and seems to need function extensionality to be carried out.

\begin{lem}\label{mainlem}
For each restricted formula $\phi \in \mathcal L_V$ the type $\llbracket \phi \rrbracket$ is equivalent to a small type.
\end{lem}

\begin{proof}
We prove the statement by induction on the structure of the formula. Atomic restricted formulae are of the form $\alpha \doteq \beta$ or $\alpha \in \beta$. In the first case we proceed by double transfinite recursion: we consider the judgement $\alpha, \beta : V \vdash C(\alpha, \beta) = (\Sigma e:U) \mathit{Equiv}( \llbracket \alpha \doteq \beta \rrbracket , El(e))$, suppose to have canonical terms $\alpha = \textsf{sup}(a,f)$, $\beta = \textsf{sup}(b,g)$ and to already have terms $t(x,y):C(f(x),g(y))$. Now we want to find a term in the type: 
$$C(\textsf{sup}(a,f), \textsf{sup}(b,g)) = (\Sigma e:U) \mathit{Equiv}( \llbracket \textsf{sup}(a,f) \doteq \textsf{sup}(b,g) \rrbracket , El(e))$$
Recalling the recursive definition of extensional equality we see that it is enough to prove that if we have a family of equivalences $(\Sigma e : U) \mathit{Equiv} (B(x), El(e))$ parametrized by a type $El(\overbar \alpha)$, then forming the dependent product of this family give rise to a corresponding equivalence $(\Sigma e':U) \mathit{Equiv}((\Pi x:El(\overbar \alpha) ) B(x) , El(e'))$ and similarly for the sum.\\
So we have the following situation:

$$\xymatrix@1{
El(e) \ar[r]^{\tau}
& B(x)  \ar@/_1pc/[l]_{\theta}  \ar@/^1pc/[l]_{t} 
}$$

\nind
Form the dependent product of these two types over $El(\overbar \alpha)$, note that $El(e)$ is constant over $El(\overbar \alpha)$ so we get simply the function type for it. We can easily define maps $\tau'$, $\theta'$ and $t'$ by composition, i.e. $\tau'(f) = f \tau$ and similarly for the others. Hence:

$$\xymatrix@1{
El(exp(\overbar \alpha , e)) \ar[r] 
& (El(\overbar \alpha) \to El(e)) \ar[r]^{\tau'}  \ar@/^1pc/[l] \ar@/_1pc/[l]
& (\Pi x : El(\overbar \alpha))B(x) \ar@/^1pc/[l]^{t'} \ar@/_1pc/[l]_{\theta ' }}$$

\nind
Where the first is the equivalence given by the rule of the weak universe. Let $ P = (\Pi x : El(\overbar \alpha))B(x)$. We just need to check that the homotopies $\eta : (\Pi s: B(x))Id_{B(x)}(\tau \theta (s), s)$ and $\epsilon$ induce homotopies $\eta ' : (\tau ' \theta ' \sim id_P) = (\Pi f:P) Id_P (\tau ' \theta ' (f), f)$ and similarly for $\epsilon$.
But is just an application of function extensionality, in fact for every $s:B(x)$ we have a pointwise identification of the functions: $\eta (s) : Id_{B(x)}(\tau \theta (s) , s)$, by function extensionality we get a proof of the identity $Id_P(\tau \theta f , f)$ for every $f : P$. For $\epsilon$ the proof is analogous and simpler because we don't have to manage a dependent product but just a function type.\\
The case of dependent sums is similar and it relies on theorem \ref{idsum} in order to construct an term in the identity type of the dependent sum.\\
For the other kind of atomic restricted formula we need to show that $(\Sigma x:El(\overbar \beta)) \llbracket \alpha \doteq  \tilde{\beta}(x) \rrbracket$ is equivalent to a small type. This follows exactly as in the case of extensional equality since we know that $\llbracket \alpha \doteq \tilde{\beta}(x) \rrbracket$ is equivalent to a small type.\\
The rest of the induction is straightforward: the case of restricted quantification was already covered whereas the cases of connectives can be easily reconstructed following the kind of argument we needed to manage $\Sigma$ and $\Pi$ with function extensionality and theorem \ref{idsum}.
\end{proof}

\begin{defn}
In CZF we say that $\alpha$ is a \textbf{subset}\index{subset} of $\beta$, written $\alpha \subseteq \beta$ iff $\forall x \in \alpha \, (x \in \beta)$.
\end{defn}

\begin{rem}
The following are valid:
\begin{enumerate}
\item $u \doteq v \Leftrightarrow (\forall x \in u)(\exists y \in v)(x \doteq y) \land (\forall y \in v)(\exists x \in u)(x \doteq y) $;

\item $u \in v \Leftrightarrow (\exists y \in v)(u \doteq y)$;

\item $u \doteq v \Leftrightarrow (\alpha \subseteq \beta \land \beta \subseteq \alpha)$
\end{enumerate}
\end{rem}

\begin{defn}
A family of types over the type of iterative sets $x:V \vdash B(x) \; \mathit{type}$ is \textbf{extensional}\index{extensional!family of types} iff 
$$(\forall x,y  \in V)((x \doteq y) \land B(x) \Rightarrow B(y))$$
Similarly, a set-theoretic formula $\phi \in \mathcal L_V$ is \textbf{extensional}\index{extensional!formula} iff $(\forall x,y \in V)[(x \doteq y \land \phi(x)) \Rightarrow \phi(y)]$.
\end{defn}

\begin{lem}
Every formula in $\mathcal L_V$ is extensional in every variable.
\end{lem}

\begin{proof}
As usual we proceed by induction on the structure of the formula. First of all consider the case of atomic formulae: $(\alpha \doteq \beta)$ is extensional because of the transitivity. So consider $\beta \in \gamma$, we want to prove extensionality in the first variable, we have a term in $\llbracket \alpha \doteq \beta \rrbracket$ and a term in $\llbracket \beta \in \gamma \rrbracket = (\Sigma x : El(\overbar \gamma)) \llbracket \tilde{\gamma}(x) \doteq \beta \rrbracket$ and a simple application of transitivity is enough to obtain a term in $\llbracket \alpha \in \gamma \rrbracket = (\Sigma x : El(\overbar \gamma)) \llbracket \tilde{\gamma}(x) \doteq \alpha \rrbracket$. For extensionality in the other variable suppose that we have a term in the type $\llbracket \gamma \doteq \delta \land \alpha \in \delta \rrbracket = (\gamma \doteq \delta) \times (\Sigma y:El(\overbar \delta))(\alpha \doteq \tilde{\delta}(y))$, and we want a proof of $\llbracket \alpha \in \delta \rrbracket = (\Sigma x: El(\overbar \gamma))(\tilde{\gamma}(x) \doteq \alpha)$. Recalling the definition of extensional equality we have that for all $x: El(\overbar \delta)$ exists a $y: El(\overbar \gamma)$ such that $\tilde{\gamma}(y) \doteq \tilde{\delta}(x)$ and we have the statement.\\
The inductive steps for the non-atomic formulae are straightforward.
\end{proof}

\begin{lem}
If $\phi(x)$ is extensional in $x$, then the structural defining axioms for the restricted quantifiers are valid.
\end{lem}

\begin{proof}
We give the details just in the case of the universal quantifier, the existential one is similar. We have a proof of $t=(t_1,t_2): \llbracket (x \doteq y \land \phi) \Rightarrow \phi(y) \rrbracket \times \llbracket (\forall x \in y) \phi(x) \rrbracket = [ (x \doteq y) \times \llbracket \phi(x)  \rrbracket \to \llbracket \phi(y)\rrbracket ] \times (\Pi x : El(\overbar y)) \llbracket \phi(\tilde{y}(x)) \rrbracket$, and we need a proof of $(\Pi x : V) \llbracket x \in y \Rightarrow \phi(x) \rrbracket = (\Pi x : V)[(\Sigma z: El(\overbar y))(x \doteq \tilde{y}(z)) \to \llbracket \phi(x) \rrbracket ]$. Therefore we can suppose to have a proof given by a couple $(t,s)$ where $s=(s_1,s_2): (\Sigma z: El(\overbar y))(x \doteq \tilde{y}(z))$. Applying $t_2$ to $s_1$ we get a proof of $\llbracket \phi(\tilde{y}(s_1)) \rrbracket$, whereas $s_2$ is a proof of $(x \doteq \tilde{y}(s_1) )$, so we have the desired term inside $\llbracket \phi(x)\rrbracket$ and we conclude with a $\lambda$-abstraction on $x$.\\
Conversely, we have a term in $[(x \doteq y) \times \llbracket \phi(x) \rrbracket \to \llbracket \phi(y) \rrbracket] \times (\Pi x : V)[(\Sigma z : El(\overbar y))(x \doteq \tilde{y}(z)) \to \llbracket \phi(x) \rrbracket]$ and we want a term in $(\Pi x : El(\overbar y)) \llbracket \phi(\tilde{y}(x))\rrbracket$. Consider $x : El(\overbar y)$, then $\tilde{y}(x) : V$ and we know that for every $z: El(\overbar y)$ such that $\tilde{y}(x) \doteq \tilde{y}(z)$ we get a term in $\llbracket \phi(x) \rrbracket$. We make the trivial choice $z=x$ and then we find a term in $\llbracket \phi(\tilde{y}(x)) \rrbracket$.
\end{proof}

\nind
Before the main theorem we state a triviality: the formula $\llbracket (\forall x \in \alpha)(x \in \alpha) \rrbracket$ is valid.\\
Indeed, we already know that $\alpha \doteq \alpha$ is valid, in fact we have constructed a term $r_0(\alpha) : \llbracket \alpha \doteq \alpha \rrbracket$. Define $\alpha^* := \lambda x. (x, r_0(\tilde{\alpha}(x))) : \llbracket (\forall x \in \alpha ) (x \in \alpha) \rrbracket$. In fact $\alpha^*(a) : \llbracket \tilde{\alpha}(a) \in \alpha \rrbracket$ for each set $\alpha$ and each term $a:El(\overbar \alpha)$.\\
From this simple fact is clear that $El(\overbar \alpha)$ has the right to be thought as the type of elements of the set $\alpha$.

\begin{thm}\label{mainthm}
Every basic axiom of CZF is valid in the interpretation.
\end{thm}

\begin{proof}
We have already checked the defining schemes for the restricted quantifiers and the equality axiom $x \doteq y \land y \in z \Rightarrow x \in z$.
\begin{enumerate}[label=(\arabic*)]

\item Extensionality: the formula $x \in \alpha$ is extensional, so that by the previous lemma $\alpha \subseteq \beta \Leftrightarrow (\forall x \in \alpha)(x \in \beta) \Leftrightarrow \forall x \, (x \in \alpha \Rightarrow x \in \beta)$, and similarly $\beta \subseteq \alpha \Leftrightarrow \forall x \, (x \in \beta \Rightarrow x \in \alpha)$, hence $\alpha \doteq \beta \Leftrightarrow (\alpha \subseteq \beta \land \beta \subseteq \alpha ) \Leftrightarrow \forall x \, (x \in \alpha \Leftrightarrow x \in \beta)$.

\item Set induction: let $B= \llbracket \forall y \, (\forall x \in y) \, \phi(x) \to \phi(y)\rrbracket$, we need to define for every set $\alpha : V$ a function $h(\alpha): B \to \llbracket \phi(\alpha) \rrbracket$ in order to conclude by $\lambda$-abstraction. We define this function by transfinite recursion on the canonical elements of $V$: we take a term $b:B$ and apply it to our set $\alpha= \textsf{sup}(a,b)$ gaining a term in $(\forall x \in \alpha) \, \phi(x) \to \phi(\alpha)$. In order to obtain a desired term in $\llbracket \phi(\alpha)\rrbracket$ we just need a function in $\llbracket (\forall x \in y) \, \phi(x) \rrbracket$ which is given by the recursion hypothesis as follows $\lambda x.h(f(x))(b)$.

\item Pairing: given sets $\alpha, \beta : V$ we define $g': \textbf N_2 \to V$ by cases: $g'(1)= \alpha$ and $g'(2)=\beta$. Now we want to use the equivalence given by the weak Tarski universe in order to define an other function $g:El(n_2) \to V$, we will take its supremum to be the pair $\{ \alpha , \beta \}$. In fact let $\textsf{ceq}_{N_2}=(\tau, \theta, t, \eta, \epsilon)$ be the canonical term in the equivalence type

$$\xymatrix@1{
El(n_2) \ar[r]^{\tau} \ar@/_2pc/[rr]^{g}
& \mathbf N_2 \ar[r]^{g'}  \ar@<1ex>[l]  \ar@/_1pc/[l]_{\theta} 
& V }$$

Then we define $g=g'\tau : El(n_2) \to V$ and $\gamma = \textsf{sup}(n_2 , g) : V$. In order to conclude it suffices to find a term in $\llbracket \alpha \in \gamma \land \beta \in \gamma \rrbracket$, consider now $\gamma^*(\theta(1)) : \llbracket \tilde{\gamma}(\theta(1)) \in \gamma \rrbracket = \llbracket g(\theta(1)) \in \gamma \rrbracket = \llbracket g'(\tau \theta (1)) \in \gamma \rrbracket$. Thanks to the homotopy $\eta(1): Id_{N_2}(\tau \theta (1), 1)$ and by remark \ref{mainrem} we have that $g'( \tau \theta (1)) \doteq g'(1)$ so that from a term in $\llbracket g'(\tau \theta (1)) \in \gamma \rrbracket$ we get a term in $\llbracket g'(1) \in \gamma \rrbracket = \llbracket \alpha \in \gamma \rrbracket$ by an application of the transitivity of $\doteq$, therefore the term $(\gamma^*(\theta(a)), \gamma^*(\theta (2)))$ give rise to the desired term.

\item Union: the proof is similar to the previous one; indeed, for every $\alpha : V$ let $A=(\Sigma x : El(\overbar \alpha))El(\overline{ \tilde{\alpha}(x) })$. As before we want to apply the type constructor that matches the set theoretic operation, in this case it is clearly the dependent sum. Define $g': A \to V$ by the elimination rule on the pairs $(x,y)$ with $x : El(\overbar \alpha)$ and $y:El(\overline{\tilde{\alpha}(x)})$ as $g'((x,y))=\widetilde{\tilde{\alpha}(x)}(y)$, so that:

$$
\xymatrix@1{
El(\sigma(\overbar \alpha, \overline{\tilde{\alpha}(x)})) \ar[r]^-{\tau} \ar@/_2pc/[rr]^{g}
& A \ar[r]^{g'}  \ar@<1ex>[l]  \ar@/_1pc/[l]_{\theta} 
& V }$$

and we take the corresponding set inside $V$, namely $\gamma = \textsf{sup}(\sigma(\overbar \alpha , \overline{\tilde{\alpha}(x)}), g)$. In order to conclude we want a term inside the type $\llbracket (\forall x \in \alpha)(\forall y \in x)( y \in \gamma)\rrbracket$; unwinding the definition of interpretation this type is equal to $(\Pi x : El(\overbar \alpha))(\Pi y: El(\overline{\tilde{\alpha}(x)}) )(\widetilde{\tilde{\alpha}(x)}(y) \in \gamma)$. Now observe that $\widetilde{\tilde{\alpha}(x)}(y) = g'((x,y)) \doteq g'(\tau \theta(x,y)) = g(\theta(x,y)))$ so is enough to find a term inside $\llbracket g(\theta(x,y)) \in \gamma \rrbracket$; and the needed term is $\gamma^*(\theta(x,y)) : \llbracket g(\theta(x,y)) \in \gamma \rrbracket$.

\item Restricted separation: by a previous lemma every restricted formula $\phi \in \mathcal L_V$ give rise to a type equivalent to a small one. Let $\phi(x) \in \mathcal L_V$ be our formula restricted in all its variables except of $x$ which is free; next consider $(\exists x \in \alpha) \, \phi(x)$ which is bounded, so its interpretation $A = (\Sigma x : El(\overbar \alpha)) \llbracket\phi(\tilde{\alpha}(x)) \rrbracket$ is equivalent to a small type, in symbols $(\Sigma e:U) \mathit{Equiv}(El(e), A)$. Define $g':A \to V$ by $g'((x,v)) = \tilde{\alpha}(x)$ where $x : El(\overbar \alpha)$ and $v: \llbracket \phi(\tilde{\alpha}(x)) \rrbracket$ and the associated $g$ as usual:

$$\xymatrix@1{
El(e) \ar[r]^{\tau} \ar@/_2pc/[rr]^{g}
& A \ar[r]^{g'}  \ar@<1ex>[l]  \ar@/_1pc/[l]_{\theta} 
& V }$$

And we consider $\gamma = \textsf{sup}(e,g) : V$. Now we want a term in $\llbracket (\forall y \in \gamma)(y \in x \land \phi(y)) \rrbracket \times \llbracket (\forall y \in x)(\phi(y) \Rightarrow y \in \gamma) \rrbracket$. For the first half we define a function $h'_1$ on $A$ by recursion over pairs as follows $h'_1(x,v) = (\alpha^*(x), v) : \llbracket \tilde{\alpha}(x) \in \alpha \rrbracket \times \llbracket \phi(\tilde{\alpha}(x)) \rrbracket$, and $h_1 = \tau h'_1$. Hence $h_1 \theta$ give rise to the desired term for the first half of our axiom. The second half is similar: for $x : El( \overbar \alpha)$ and $v: \llbracket \phi(\tilde{\alpha}(x)) \rrbracket$ we define $h'_2(x)(v) = \gamma^*((x,v)) : \llbracket \tilde{\gamma}(x,v) \in \gamma \rrbracket = \llbracket g(x,v) \in \gamma \rrbracket = \llbracket \tilde{\alpha}(x) \in \gamma \rrbracket$, as usual $h_2 \theta$ determines the desired term.

\item Strong collection: we start with a premiss, given a formula $\phi(x,y) \in \mathcal L_V$ with at most $x$ and $y$ free, let $\alpha , \beta : V$ be sets such that $\overbar \alpha = \overbar \beta$ and for every $a: El(\overbar \alpha)$ we have this term $f(a) : \llbracket \phi(\tilde{\alpha}(a), \tilde{\beta}(a))\rrbracket$. This yields a term in $\llbracket \phi'(\alpha , \beta) \rrbracket$ where we recall that $\phi' (\alpha,\beta) = (\forall x \in \alpha)(\exists y \in \beta) \, \phi(x,y) \land (\forall y \in \beta)(\exists x \in \alpha) \, \phi(x,y)$. Indeed, this term is $K(f) = (\lambda x.(x, f(x)) , \lambda x. (x, f(x)))$.\\
Now turn to the strong collection axiom: let $\phi(x,y)$ as above and $\alpha : V$. We start with a term $a: \llbracket (\forall x \in \alpha) \exists y \, \phi(x,y) \rrbracket$ and we want to construct a term in the type $ \llbracket \exists z \, \phi'(\alpha, z)\rrbracket$. In order to achieve this consider the term $a(x)$ in the dependent sum and project it  $b=\lambda x.\textsf{p}(a(x)) : El(\overbar \alpha) \to V$ and $c=\lambda x.\textsf{q}(a(x)) : (\Pi x : El(\overbar \alpha)) \llbracket \phi(\tilde{\alpha}(x), b(x))\rrbracket$ where $\textsf{p}$ and $\textsf{q}$ are the canonical projections.\\
Now define $\beta = \textsf{sup}(\alpha, b)$, so that $K(c) : \llbracket \phi'(\alpha, \beta) \rrbracket$ and hence $d(a) = (\beta, K(c))$ is the desired term.

\item Subset collection: given sets $\alpha , \beta : V$ we use the equivalence given by the universe to perform type-theoretically the needed construction. Consider $z: El(\overbar \alpha) \to El(\overbar \beta)$ so that $z(x) : El(\overbar \beta)$ and $\tilde{\beta}(z(x))$; now define $G': (El(\overbar \alpha) \to El(\overbar \beta)) \to V$ as $G'=\lambda x.\textsf{sup}(\overbar \alpha , \tilde{\beta}z)$ and $G=G' \tau$, given by our usual picture: 

$$
\xymatrix@1{
El(\exp(\overbar \alpha , \overbar \beta)) \ar[r]^{\tau} \ar@/_2pc/[rr]^{G}
& (El(\overbar \alpha) \to El(\overbar \beta)) \ar[r]^-{G'}  \ar@<1ex>[l]  \ar@/_1pc/[l]_{\theta} 
& V }$$

Let $\phi_u(x,y) \in \mathcal L_V$ with at most $u$, $x$ and $y$ free and let $\psi_u(\alpha , \beta)$ denote $(\forall x \in \alpha )(\exists y \in \beta) \, \phi_u(x,y)$. Finally let $\delta : V$ be another set and take $a : \llbracket \psi_{\delta}(\alpha , \beta) \rrbracket$ to be a given term from which we wish to construct a term inside $\llbracket (\exists d \in \gamma ) \phi'_u(\alpha , d) \rrbracket$. Then consider $b= \lambda x.\textsf{p}(a(x)) : El(\overbar \alpha) \to El(\overbar \beta)$ and $c=\lambda x. \textsf{q}(a(x)) : (\Pi x:El(\overbar \alpha)) \llbracket \phi_{\delta}(\tilde{\alpha}(x) , \tilde{\beta}(b(x))) \rrbracket$, by which we construct $K(c) : \llbracket \phi'_{\delta}(\alpha , G'(b)) \rrbracket$. Hence $\llbracket (\exists z \in \gamma) \, \phi'_{\delta}(\alpha, z) \rrbracket = (\Sigma z:El(exp(\overbar \alpha, \overbar \beta))) \llbracket \phi'_{\delta}(\alpha , G(z)) \rrbracket$ and $G(\theta(b)) = G'(\tau \theta (b)) \doteq G'(b)$ so that $d(\delta , a) = (\theta(b), K(c))$ gives us the desired term.

\item Infinity: let $f'_0 : \textbf N_0 \to V$ and $g'_0 : \textbf N_0 \to \textbf N_0$ be the canonical functions and $f_0$, $g_0$ the associated functions defined on $El(n_0)$:

$$\xymatrix@1{
El(n_0) \ar[r]^{\tau_1} \ar@/_2pc/[rr]^{f_0}
& \mathbf N_0 \ar[r]^{f'_0}  \ar@<1ex>[l]  \ar@/_1pc/[l]_{\theta_1} 
& V 
}$$

and

$$\xymatrix@1{
El(n_0) \ar[r]^{\tau_2} \ar@/_2pc/[rr]^{g_0}
& \mathbf N_0 \ar[r]^{g'_0}  \ar@<1ex>[l]  \ar@/_1pc/[l]_{\theta_2} 
& \mathbf N_0
}$$

Let $\emptyset = \textsf{sup}(n_0 , f_0)$. We have $g_0 : \llbracket \mathit{Zero}(\emptyset) \rrbracket = (\Pi y:El(n_0)) N_0 = (El(n_0) \to N_0)$. Now for every set $\alpha : V$ define the function $h'(\alpha) : El(\overbar \alpha) + \textbf N_1 \to V$ by cases as $h'(\alpha)(\textsf{i}(x)) = \tilde{\alpha}(x)$ and $h'(\alpha)(\textsf{j}(1)) = \alpha$; consider the following diagram:

$$\xymatrix@1{
El(plus(\overbar \alpha , n_1)) \ar[r]^{\tau_3} \ar@/_2pc/[rrr]^{h(\alpha)}
& El(\overbar \alpha) + El(n_1) \ar[r]^-{\tau_4}  \ar@<1ex>[l]  \ar@/_1pc/[l]_{\theta_3} 
& El(\overbar \alpha)+ \mathbf N_1 \ar[r]^-{h'(\alpha)} \ar@<1ex>[l] \ar@/_1pc/[l]_-{\theta_4}
& V
}$$

and define $\tau = \tau_4 \tau_3$, $\theta = \theta_3 \theta_4$ and $h(\alpha) = h'(\alpha) \tau$. By the kind of argument of lemma \ref{mainlem} we know that $El(\overbar \alpha) + El(n_1)$ is equivalent to $El(\overbar \alpha) + \textbf N_1$, so that by composition we have an equivalence from $El(plus(\overbar \alpha , n_1))$ to $El(\overbar \alpha)+ \textbf N_1$.\\
We then form the supremum of this function $S(\alpha) = \textsf{sup}(plus(\overbar \alpha , n_1) , h(\alpha))$. Hence $S(\alpha)^*(\theta j(1)) : \llbracket \widetilde{S(\alpha)}(\theta j(1) \in S(\alpha)) \rrbracket = \llbracket h(\alpha)(\theta j(1)) \in S(\alpha) \rrbracket$ and by homotopy $h(\alpha)(\theta j(1)) = h'(\alpha)(\tau \theta j(1)) \doteq h'(\alpha)(j(1)) = \alpha$.\\
Now we prove that $\mathit{Succ}(\alpha , S(\alpha))$ is valid. recall that $\llbracket \mathit{Succ}(\alpha , S(\alpha)) \rrbracket = \llbracket (\forall z \in \alpha)(z \in S(\alpha)) \rrbracket \times \llbracket (\alpha \in S(\alpha)) \rrbracket \times \llbracket (\forall z \in S(\alpha))(z \in \alpha \lor z \doteq \alpha) \rrbracket$, we have already found a term in the type in the middle. Now we construct a term in the first type: consider firstly $g_1(\alpha) = \lambda x. S(\alpha)^*(\theta i (x))$; it yields a term in $\llbracket (\forall x \in \alpha) x \in S(\alpha) \rrbracket$.

Finally define $g'_2(\alpha) : (\Pi u: El(\overbar \alpha) + \textbf N_1 ) \llbracket h'(\alpha)(u) \in \alpha \lor h'(\alpha)(u) \doteq \alpha\rrbracket$ by cases as $g'_2(\alpha)(i(x)) = i(\alpha^*(x))$ and $g'_2(\alpha)(j(1)) = j(r_0(\alpha))$. What we need is a term in $(\Pi u:El(plus(\overbar \alpha , n_1))) \llbracket h(\alpha)(u) \in \alpha \lor h(\alpha)(u) \doteq \alpha \rrbracket = (\Pi u:El(\overline{S(\alpha)})) \llbracket \tilde{S(\alpha)}(u) \in \alpha \lor \tilde{S(\alpha)}(u) \doteq \alpha \rrbracket$ and this is easy since $g_2(\alpha) = g'_2(\alpha)\tau$ gives rise to a term in the desired type. Let $g(\alpha)$ be the term obtained in the type $\mathit{Succ}(\alpha , S(\alpha))$.\\
Then we define the set of natural numbers from the type $\textbf N$, so define the map $\Delta' : \textbf N \to V$ by recursion as $\Delta(0) = \emptyset$ and $\Delta'(s(n)) = S(\Delta'(n))$; consider the diagram:

$$\xymatrix@1{
El(n) \ar[r]^{\tau_4} \ar@/_2pc/[rr]^{\Delta}
& \mathbf N \ar[r]^{\Delta'}  \ar@<1ex>[l]  \ar@/_1pc/[l]_{\theta_4} 
& V 
}$$

Define the set of natural numbers as $\omega = \textsf{sup}(n,\Delta) : V$. First of all we check the validity of $(\exists x \in \omega) \, \mathit{Zero}(x)$ so we need a term in $(\Sigma x \in El(n)) \llbracket \mathit{Zero}(\Delta(x)) \rrbracket$ thanks to $\Delta (\theta_4 (0)) = \Delta'(\tau_4 \theta_4 (0)) \doteq \Delta'(0) = \emptyset$ we have that $g_0 : \llbracket \mathit{Zero}(\emptyset) \rrbracket$ together with $\theta_4(0)$ gives rise to the desired term.\\
Similarly for $\llbracket (\forall x \in \omega) \, \mathit{Zero}(x) \lor (\exists y \in x) \, \mathit{Succ}(y,x) \rrbracket$ we define $f': (\Pi x : \textbf N) \llbracket \mathit{Zero}(\Delta'(x)) \lor (\exists y \in \Delta'(x)) \, \mathit{Succ}(y,\Delta'(x)) \rrbracket$ by recursion as $f'(0) = i(g_0)$ and $f'(s(n)) = j(n,g(\Delta'(n)))$ so that by precomposition with $\tau_4$ we get the desired term. Similarly for $h' = \lambda y.(s(y),g(\Delta'(y))) : (\Pi y: \textbf N)(\Sigma x: \textbf N) \llbracket \mathit{Succ}(y,x) \rrbracket$ and we conclude as usual.

\end{enumerate}
\end{proof}

Now we focus on choice principle, we shall see that each instance of the set-theoretic axiom scheme of dependent choices is valid in the interpretation.\\
First of all we prove in Martin-L\"of type theory that the type theoretic scheme of dependent choice holds.

\begin{thm}[type-theoretic DC]\label{dc}
The following type is inhabited:
$$[(\Pi x:A)(B(x) \to (\Sigma y:A)(B(y) \times F(x,y)))] \to [(\Pi x:A)(B(x) \to (\Sigma z: \textbf N \to A)G(x,z))] $$

\nind
where $G(x,z) := Id_A(z(0) , x) \times (\Pi n: \textbf N)[B(z(n)) \times F(z(n) , z(s(n)))]$.
\end{thm}

\begin{proof}
If $f$ is an element of the premiss i.e. $f: (\Pi x:A)(B(x) \to (\Sigma y:A)(B(y) \times F(x,y)))$ then define $C := (\Sigma x:A)B(x)$ and let:
\begin{center}$h := \lambda u. f(\textsf{p}(u))(\textsf{q}(u)) : C \to (\Sigma y:A)( B(y) \times F(\textsf{p}(u) , y))$
$$g := \lambda u.(\textsf{p}(h(u)) , \textsf{p}(\textsf{q}(h(u)))) : C \to C$$
$$k := \lambda u. \textsf{q}(\textsf{q}(h(u))) : (\Pi u:C)F(\textsf{p}(u) , \textsf{p}(g(u))) $$
\end{center}

\nind
Now let $a:A$ and $b:B(a)$. Then by recursion over $\textbf N$ we may define $e: \textbf N \to C$ by $e(0) := (a,b)$ and $e(s(n)) := g(e(n))$. Observe that we can derive the following inhabitation judgements:
\begin{center}
$\lambda z.\textsf{p}(e(z)) : \textbf N \to A$
$$ \textsf{refl} : Id_A(\textsf{p}(e(0)) , a)$$
$$\lambda z. (\textsf{q}(e(z)) , k(e(z))) : (\Pi z:\textbf N)[B(\textsf{p}(e(z))) \times F(\textsf{p}(e(z)) , \textsf{p}(es(z)))]$$
\end{center}

\nind
hence putting these three together we obtain the consequent of DC.
\end{proof}

\nind
Using the basic axiom it is easy to define the standard set-theoretic notions, interpreted in type theory. So for example a set $\alpha : V$ is a \textit{relation} iff the following holds $(\Pi \beta \in \alpha)(\Sigma \gamma \in V)(\Sigma \delta \in V)(\beta \doteq ( \gamma , \delta ))$.\\
Let now introduce an useful function: if $\alpha , \beta : V$ such that $\overbar \alpha = \overbar \beta$ then define $T(\alpha , \beta) :=\textsf{sup}(\overbar \alpha , \lambda x. ( \tilde{\alpha}(x) , \tilde{\beta}(x) )) : El(\overbar \alpha) \to V$.

\begin{lem}
\begin{enumerate}[label=(\roman*)]
\item[]

\item If $\alpha , \beta : V$ with $\overbar \alpha = \overbar \beta$ then $T(\alpha , \beta)$ is a relation with domain $\alpha$ and range $\beta$;

\item If $\alpha , \gamma : V$ such that $\gamma$ is a relation with domain $\alpha$ then, for some $\beta : V$ with $\overbar \alpha = \overbar \beta$, $T(\alpha , \beta) \subseteq \gamma$.

\end{enumerate}
\end{lem}

\begin{proof}
\begin{enumerate}[label=(\roman*)]
\item[]

\item Let $\gamma$ be the set $T(\alpha , \beta)$. Then by choosing $x=y=u$ we get a term in:
$$(\Pi u:El(\overbar \alpha))(\Sigma x:El(\overbar \alpha))(\Sigma y: El(\overbar \alpha)) (\tilde{\gamma}(u) \doteq ( \tilde{\alpha}(x) , \tilde{\beta}(y) )) $$
so that we have a proof of:
$$(\forall u \in \gamma)(\exists x \in \alpha )(\exists y \in \beta)(u \doteq ( x,y )) $$

hence $\gamma$ is a relation whose domain is a subset of $\alpha$ and whose range is a subset of $\beta$. Also, by choosing $u=y=x$, we get:
$$(\forall x \in \alpha)(\exists u \in \gamma)(\exists y \in \beta) (u \doteq ( x,y )) $$
thus  $\alpha$ is a subset of the domain of $\gamma$, and hence $\gamma $ has domain $\alpha$. Similarly, $\gamma $ has range $\beta$.

\item Let $\gamma : V$ be a relation with domain $\alpha$, then $(\Pi x:El(\overbar \alpha))(\Sigma x:V)(( \tilde{\alpha}(x) , z ) \in \gamma) $. By the type-theoretic axiom of choice \ref{choice} there is $b:El(\overbar \alpha) \to V$ such that $(\forall x:A)(( \tilde{\alpha}(x) , b(x) ) \in \gamma)$, so that if $\beta := \textsf{sup}(\overbar \alpha , b)$ then $\beta : V$ with $\overbar \beta = \overbar \alpha$ and $(\Pi x \in T(\alpha , \beta)) (x \in y)$.

\end{enumerate}
\end{proof}

\nind
Now we give a definition in type theory, which is internal to type theory and is not an extensional notion, but is closely related to choice principles and will be one of the key ingredient of the proof, being a bridge between the extensional equality $\doteq$ and the intensional identity types of sets.

\begin{defn}
A set $\alpha : V$ is \textbf{injectively presented}\index{injectively presented set} iff for all $x_1 , x_2 : El(\overbar \alpha)$ we have $$(\tilde{\alpha}(x_1) \doteq \tilde{\alpha}(x_2)) \to Id_{El(\overbar \alpha)} (x_1 , x_2)$$
\end{defn}

\begin{lem}\label{dclem}
Let $\alpha : V$ be injectively presented. Then:
\begin{enumerate}[label=(\roman*)]

\item If $ \beta : V$, such that $\overbar \beta = \overbar \alpha$ and $\delta : V$, then for all $x:El(\overbar \alpha)$ we have $$(( \tilde{\alpha}(x) , \delta ) \in T(\alpha , \beta)) \Leftrightarrow (\delta \doteq \tilde{\beta}(x)) $$

\item If $\gamma : V$ then $\gamma$ is a function with domain $\alpha$ iff $\gamma \doteq T(\alpha , \beta)$ for some $\beta : V$ such that $\overbar \beta = \overbar \alpha$; 

\item If $\beta_1 , \beta_2 : V$ such that $\overbar \beta_1 = \overbar \beta_2 = \overbar \alpha$ then $$(T(\alpha , \beta_1) \doteq T(\alpha , \beta_2)) \Leftrightarrow (\forall x \in \alpha) (\tilde{\beta_1}(x) \doteq \tilde{\beta_2}(x))$$
\end{enumerate}
\end{lem}

\begin{proof}
\begin{enumerate}[label=(\roman*)]
\item[]

\item Let $\beta : V$ such that $\overbar \beta = \overbar \alpha$ and $\delta : V$, $x: El(\overbar \alpha)$ . Then $( \tilde{\alpha}(x , \delta ) \in T(\alpha , \beta))$ iff $(\Sigma y : El(\overbar \alpha))(( \tilde{\alpha}(x) , \delta ) \doteq ( \tilde{\alpha}(y) , \tilde{\beta}(y) )) $ iff $(\Sigma y : El(\overbar \alpha))(\tilde{\alpha}(x) \doteq \tilde{\alpha}(y) \times \delta \doteq \tilde{\beta}(y))$ iff $(\Sigma y : El(\overbar \alpha)) (Id_{El(\overbar \alpha)}(x,y) \times \delta \doteq \tilde{\beta}(y))$ iff $\delta \doteq \tilde{\beta}(x)$.

\item If $\gamma : V$ is a function with domain $\alpha$ then it is in particular a relation with domain $\alpha$ hence by the second part of the previous theorem there is $\beta : V$ such that $\overbar \alpha = \overbar \beta$ and $T(\alpha , \beta) \subseteq \gamma$. As $\gamma$ is a function with domain $\alpha$ and $T(\alpha , \beta)$ is a relation with domain $\alpha$ it follows that $T(\alpha , \beta) \doteq \gamma$. For the converse let $\gamma \doteq T(\alpha , \beta)$ where $\beta : V$ such that $\overbar \beta = \overbar \alpha$. By the first part of the previous theorem $\gamma$ is a relation with domain $\alpha$. Also, for $x,y : El(\overbar \alpha)$, if $( \tilde{\alpha}(x) , \tilde{\beta}(y) ) \in \gamma$ then by the first part of this theorem we have $\tilde{\beta}(y) \doteq \tilde{\beta}(x)$. Hence:
$$(\forall x \in \alpha)(\forall y_1 \in \beta)(\forall y_2 \in \beta)(( x,y_1 ) \in \gamma \land ( x , y_2 ) \in \gamma \Rightarrow(y_1 \doteq y_2)) $$
so that $\gamma$ is a function with domain $\alpha$.

\item Let $\beta_1 , \beta_2 : V$ such that $\overline{\beta_1} = \overline{\beta_2} = \overbar \alpha$. Then $T(\alpha , \beta_1) \subseteq T(\alpha , \beta_2)$ iff $(\Pi x:El(\overbar \alpha))(( \tilde{\alpha}(x) , \tilde{\beta_1}(x) ) \in T(\alpha , \beta_2))$ iff $(\Pi x:El(\overbar \alpha))(\tilde{\beta_1}(x) \doteq \tilde{\beta_2}(x))$. Similarly for the converse.

\end{enumerate}
\end{proof}

\begin{lem}
The set of natural numbers $\omega$ is injectively presented.
\end{lem}

\begin{proof}
First of all recall the definition of the set of natural numbers in type theory $\omega := \textsf{sup}(n , \Delta)$ with the usual diagram:

$$\xymatrix@1{
El(n) \ar[r]^{\tau} \ar@/_2pc/[rr]^{\Delta}
& \mathbf N \ar[r]^{\Delta'}  \ar@<1ex>[l]  \ar@/_1pc/[l]_{\theta} 
& V 
}$$

\nind
Our goal is to prove $(\Delta(x) \doteq \Delta(y)) \to Id_{El(n)} (x,y)$. In order to do this we split the map in three as: $(\Delta' \tau(x) \doteq \Delta' \tau(y)) \to Id_{\textbf N}(\tau(x) , \tau(y)) \to Id_{El(n)}(\theta \tau(x) , \theta \tau(y)) \to Id_{El(n)} (x,y)$.\\
The latter is given as usual by the principle of indiscernibility of identicals, the second is simply the application of $\theta$ to a path, whereas the first can be obtained by a routine double induction on $\textbf N$ on $\tau(x)$ and within that on $\tau(y)$, at the key steps two terms are needed one for the empty set and one for the successors: $(\void \doteq \void) \to Id_{\textbf N}(0,0)$ and $(\Delta'(s(x)) \doteq \Delta'(s(y))) \to Id_{\textbf N}(x,y)$, extracting these is easy.
\end{proof}

\begin{thm}
The set-theoretic axiom of dependent choice is valid in the interpretation, more generally for any extensional family of types $B(x)$ over $V$ and for any family of types $F(x,y)$ over $V \times V$ extensional in each argument, such that
$$(\Pi x : V)(B(x) \to (\Sigma y : V)(B(x) \times F(x,y))) $$
Then for each $\alpha : V$ such that $B(\alpha)$ there is $\delta : V$ such that $\delta$ is a function with domain $\omega$, moreover $( \void , \alpha ) \in \delta$ and for every $x \in \omega$ and for all $\beta , \gamma : V$
$$( \llbracket ( x , \beta ) \in \delta  \rrbracket \times \llbracket ( S(x) , \gamma  ) \in \delta \rrbracket) \to (B(\beta) \times F(\beta , \gamma)) $$
\nind
where $S(x)$ is the successor of $x$.
\end{thm}

\begin{proof}
The idea is to use the type-theoretic axiom of dependent choice and then obtain the set-theoretic analogue using the fact that $\omega$ is injectively presented and the previous constructed function $T(\alpha , \beta)$.\\
Let $\alpha : V$ be such that $B(\alpha)$. Then by the type theoretic axiom of dependent choice \ref{dc} exists $c':N \to V$ such that $Id_{V}(c'(0), \alpha)$ and $B(c'(n)) \times F(c'(n),c'(s(n)))$. Then consider the diagram:

$$\xymatrix@1{
El(n) \ar[r]^{\tau} \ar@/_2pc/[rr]^{c}
& N \ar[r]^{c'}  \ar@<1ex>[l]  \ar@/_1pc/[l]_{\theta} 
& V 
}$$

\nind
and define $\eta := \textsf{sup}(n,c)$, recall that $\omega = \textsf{sup}(n, \Delta)$, hence $\overbar \eta = \overbar \omega$. If $\delta = T(\omega , \eta)$. $\omega$ is injectively presented by the previous lemma then by the second part of \ref{dclem} $\delta$ is a function with domain $\alpha$. As $( \void , \alpha )  \doteq ( \void , c'(0) ) \doteq ( \Delta\theta(0) , c\theta(0) ) \doteq ( \tilde{\omega}\theta(0) ,  \tilde{\eta}\theta(0) ) \doteq \tilde{\delta}\theta(0)$, by the usual trick it follows that $( \void , \alpha  ) \in \delta$.\\
Finally, let $x:El(n)$ and $\beta , \gamma : V$ such that $(  \tilde{\omega}(x) , \beta )\in \delta$ and $( S(\tilde{\omega}(x)) , \gamma ) \in \delta$. Observe that $( \tilde{\omega}(x) , \beta ) \in \delta$ implies directly by the first part of \ref{dclem} that $\beta \doteq \tilde{\eta}(x) = c(x) =c'(\tau(x)) $. In order to use the extensionality of the families $B$ and $F$ we just need the analogous for $\gamma$. So consider the successor internal to $El(n)$ given by $\theta s \tau (x)$, we have that $S(\tilde{\omega}(x)) \doteq \tilde{\omega}(\theta s \tau (x))$, which is a simple calculation. Therefore by the first part of \ref{dclem} we obtain $\gamma \doteq \tilde{\eta}(\theta s \tau (x)) = c(\theta s \tau (x)) = c'(s \tau (x))$, so that we have the thesis.

\end{proof}

\nind
Finally, we study the regular extension axiom.

\begin{thm}
REA is valid in the interpretation.
\end{thm}

\begin{proof}
In CZF every set is a subset of a transitive set, for example the transitive closure \ref{trans}. Hence it suffices to show that if $\alpha_0 : V$ is transitive, then there is a regular set $\alpha :V$ such that $\alpha_0 \subseteq \alpha$. Then it suffices to show that $\alpha $ is transitive and that for each family of types $F$ over $V \times V$ we have:

\begin{itemize}
\item[(*)]$(\forall x \in \beta)(\exists y \in \alpha)F(x,y) \Rightarrow (\exists \beta ' \in \alpha) F'(\beta , \beta ')$ 
\end{itemize}

\nind
where $F'(\beta , \beta ') := (\forall x \in \beta)(\exists y \in \beta ')F(x,y) \times (\forall y \in \beta ')(\exists x \in \beta)F(x,y)$.\\
So let $\alpha_0 : V$ be transitive and define $A_0 := El(\alpha_0)$ hence we have $b_0 := \lambda x. \overline{\tilde{\alpha}(x)} : El(\alpha_0) \to U$ and then $B_0(x) := El( b_0(x))$. Now we form the $W$-type of this family: $A := (W x:A_0)B_0(x) = (W x:El(\overbar \alpha_0))El(\overline{\tilde{\alpha}(x)})$, and the usual diagram:

$$
\xymatrix@1{
El(w(\overbar \alpha_0 , \overline{\tilde{\alpha}_0(x)})) \ar[r]^-{\tau} \ar@/_2pc/[rr]^{h}
& A \ar[r]^{h'}  \ar@<1ex>[l]  \ar@/_1pc/[l]_{\theta} 
& V 
}$$

\nind
where $h'$ is defined by transfinite recursion over the $W$-type as $h'(\textsf{sup}(a,f)) := (\textsf{sup} \, u:B_0(a)) h'f(u)$ given $a:A_0$ and $f:B_0(a) \to A$. We define the desired extension as $\alpha := \textsf{sup}(w(\overbar \alpha_0 , \overline{\tilde{\alpha_0}(x)}) , h)$.\\

Now we prove an intermediate step: let $\beta : V$, if $\beta \doteq \gamma$ for some $\gamma : V$ such that $\overbar \gamma = \overline{\tilde{\alpha_0}(a)}$ for some $x:El(\overbar \alpha_0 )$, then for every extensional family of types $F$ over $V \times V$ the $(*)$ is valid.\\
Assume that $(\Pi x:El(\overbar \beta))(\Sigma y : El(\overbar \alpha))F(\tilde{\beta}(x) , \tilde{\alpha}(y))$. Recall the definition of $\alpha$, hence by the hypotheses we have $(\Pi x:El(\overline{\tilde{\alpha_0}(a)}))(\Sigma y : El(w(\overbar \alpha_0 , \overline{\tilde{\alpha_0}(a)})))  F(\tilde{\beta}(x) , h(y))$. Hence by the type-theoretic axiom of choice there is an $f:B_0(a) \to El(w(\overbar \alpha_0 , \overline{\tilde{\alpha_0}(a)}))$ such that the following condition $[*]$ holds: $(\Pi x:B_0(a))F(\tilde{\gamma}(x) , hf(x))$.\\
Next consider the composition $\displaystyle B_0 \mathop{\longrightarrow}^{\tau f} A \mathop{\longrightarrow}^{h'} V$ and consider the supremum $\textsf{sup}(a , \tau f) : A$. We define the desired set as $\beta ' := h'(\textsf{sup}(a , \tau f)) = (\textsf{sup} \, u:B_0(a))h f(u)$ hence $\overbar \beta ' = \overbar \gamma$ and $\tilde{\beta '} = hf$. This yields to $(\Pi x:B_0(a))F(\tilde{\gamma}(x) , \tilde{\beta '}(x))$ by condition $[*]$. Therefore we have $F'(\gamma , \beta ')$. As $\beta \doteq \gamma$ we get $F'(\beta , \beta ')$ by extensionality, as desired.\\

Now return to our set $\alpha$, we show that it is regular. For simplicity define $w:= w(\overbar \alpha_0 , \overline{\tilde{\alpha_0}(a)})$. Let $\beta \in \alpha$ then $\beta \doteq h(c)$ for some $c: El(w)$, we want a map:
$$(\Sigma c:A) \llbracket \beta \doteq h'(c) \rrbracket \to \llbracket \beta \subseteq \alpha \rrbracket \times \llbracket (\forall x \in \beta)(\exists y \in \alpha)F(x,y) \Rightarrow (\exists \beta ' \in \alpha) F'(\beta , \beta ') \rrbracket$$

\nind
and we build it on canonical terms: so by induction suppose $c= \textsf{sup}(a,f)$ for $a:El(\overbar \alpha_0)$ and $f:B_0(a) \to A$, then $\beta \doteq h'(\textsf{sup}(a,f)) = (\textsf{sup} \, u:B_0(a))h'f(u)$. As $h'f(u) \in \alpha$ for $u:B_0(a)$ it follows that $\beta \subseteq \alpha$. Finally, note that the assumptions of the intermediate step hold with $\gamma = h'(c)$.

It remains to show that $\alpha_0 \subseteq 	\alpha$. We show that $\llbracket \beta \in \alpha_0 \rrbracket \to \llbracket \beta \in \alpha \rrbracket$ by set-induction on $\beta :V$. As inductive hypothesis we assume $(\forall y \in \beta)(y \in \alpha_0 \Rightarrow y \in \alpha)$. Now if $\beta \in \alpha_0$ then $\beta \doteq \tilde{\alpha_0}(a)$ for some $a:El(\overbar \alpha_0)$ so the assumption of the intermediate step holds for $\gamma = \tilde{\alpha_0}(a)$, hence $(*)$ holds for $F(x,y) = (x \doteq y)$. As $\alpha_0$ is transitive by the induction hypothesis we have $\beta \subseteq \alpha$, from this proof we want a proof of the thesis.\\
Note that $(\forall x \in \beta)(\exists y \in \alpha)(x \doteq y) \Rightarrow (\exists \beta ' \in \alpha)(\beta \doteq \beta ')$ is equivalent to $(\forall x \in \beta)(x \in \alpha) \Rightarrow (\beta \in \alpha)$.

\end{proof}

\section*{Conclusions}

Let start with some remarks: extensionality, set-induction and subset collection do not need any change in the previous theorems, the standard proofs generalise word by word. For the other axioms some changes are needed in order to handle the weakening of the universe. In particular the three nontrivial points are given by restricted separation and infinity which rely on lemma \ref{mainlem}, and the discussion about dependent choices which needs few more words than the usual reformulations.\\
Observe that we have used function extensionality just for lemma \ref{mainlem}. This yields to the question if it is necessary for the lemma and for restricted separation. However, we will not address this question here.

We cannot omit a brief comparison with the treatment of the cumulative hierarchy in \cite{UFP}. It is defined as an higher inductive type with the aim to provide a model of ZFC (for some improvements on the definition see \cite{Led}). Thank to its definition is possible to prove that for $x,y:V$ the type $(x \doteq y) \to Id_V(x,y)$ is inhabited. In spite of this result this notion seems unable to prove strong collection and subset collection. Finally, in order to build the desired model of ZFC an additional axiom is added to type theory, namely the set-theoretical axiom of choice, which yields full excluded middle.\\

This generalised interpretation has implications on the status of the axiom of univalence, in fact univalence is a statement about the universe. Now we know that weak univalent Tarski universes - which are easier to obtain in semantics - are good as the strict ones, if we are interested in obtaining models of constructive set theory.\\
Therefore one may wonder if there are substantial differences between the consequences of the existence of strict univalent universes and weak ones. We expect that differences between these two variants emerge when the proof-theoretic and computational structure of the theory are analysed. Indeed, the rules for the weak universe break the symmetry of the computational rules introducing equivalences instead of syntactical definitional equalities.\\

In addition, this result may be seen as a hint that also (some of) the other coherence conditions can be weakened, hence that homotopical models can provide motivations and play a role in the development of type theoretical system with explicit substitution, as suggested in \cite{AW}(for example they are studied in \cite{ACCL}). These systems have strong syntactical motivations and are studied with the hope to fill the gap between type theories and concrete implementations, in fact substitutions are usually treated implicitly, hiding relevant computations in the meta-theoretical level.\\

In our opinion the next step in the direction of the study of weak univalent universes is to check if they imply function extensionality. We expect this to be true. Indeed, the main reason for the introduction of an extensional principle like function extensionality - with the risk to break the computational content of the theory - is the study of the models in which this principle is validated, which are also the main motivation for weak universes.

\appendix

\chapter{Locally Cartesian Closed Categories}
\label{lccc}

\lettrine{I}{n} this appendix we recall the very basics about locally cartesian closed categories, we refer the reader to \cite{Joh} chapter A1.5.

\begin{defn}
A category $\mathcal C$ is \textbf{locally cartesian closed}\index{category!locally cartesian closed} (lccc for short) iff every slice $\mathcal C / A$ is cartesian closed.
\end{defn}

\nind
Observe that \textbf{Cat} is not locally cartesian closed because pullback functors do not in general preserve coequalizers.
Given an object $B$ in $\mathcal C$ we use the same letter for the unique morphism $B \to 1$, thus $B^*: \mathcal C \to \mathcal C / B$ denotes the right adjoint of the forgetful functor $\Sigma_B : \mathcal C / B \to \mathcal C$. $B^*(A)$ is the product projection $A \times B \to B$.

\begin{lem}
Let $\mathcal C$ be a category with finite limits. An object $B$ of $\mathcal C$ is exponentiable iff the functor $B^*: \mathcal C \to \mathcal C /B$ has a right adjoint $\Pi_B : \mathcal C / B \to \mathcal C$.
\end{lem}

\begin{proof}
Observe that $(-) \times B$ is the composite $\displaystyle \mathcal C \mathop{\longrightarrow}^{B^*} \mathcal C / B \mathop{\longrightarrow}^{\Sigma_B} \mathcal C$, so if $\Pi_B$ exists we may define $(-)^B$ as the composite $\Pi_B \circ B^*$.\\
Conversely, suppose $B$ exponentiable. Given $f:A \to B$ form the pullback:

$$\xymatrix@1{
& \Pi_B(f) \ar[r] \ar[d] 
& A^B \ar[d]^{f^B} \\
& 1 \ar[r]^{\bar q}
& B^B
}$$

\nind
where $\bar q$ is obtained by adjunction from the projection $q:1 \times B \to B$. Now for any object $C$ of $\mathcal C$, morphisms $C \to \Pi_B(f)$ correspond to morphisms $\bar h : C \to A^B$ such that $f^B \bar h = \bar q C$, and hence t morphisms $h:C \times B \to A$ such that $fh$ is the product projection $C \times B \to B$, i.e. to morphisms $B^*(C) \to f$ in $\mathcal C / B$. It is straightforward to verify that this correspondence is natural in $C$ and $f$.
\end{proof}

\begin{cor}
A category with finite limits is locally cartesian closed iff for every morphism $f:A \to B$, the base change functor $f^*: \mathcal C / B \to \mathcal C / A$ has a right adjoint $\Pi_f$.
\end{cor}

\begin{proof}
It is easy to check that we have an isomorphism $(\mathcal C / B) / f \cong \mathcal / A$, so that we have reduced the statement to the previous lemma.
\end{proof}

\begin{lem}
let $\mathcal C$ be a category and $F: \mathcal C \to \textbf{Set}$ a functor. Then there is a category $\mathcal F$ such that the slice category $(\textbf{Set}^{\mathcal C}) / F$ is equivalent to the functor category $\textbf{Set}^{\mathcal F}$. Moreover, if $\mathcal C$ is small [with small $\mathit{Hom}$-sets], then so is $\mathcal F$.
\end{lem}
\begin{proof}
We define $\mathcal F$ as a category structured over $\mathcal C$ its objects are the elements of the disjoint union of sets $\coprod_{A \in Ob(\mathcal C)} F(A)$, the underlying object of $x$ being the unique $A$ such that $x \in F(A)$. Sources and targets are defined by $f:x \to y$ iff $F(f)(x) = y$. The fact that $\mathcal F$ is small (or has small $\mathit{Hom}$-sets) if $\mathcal C$ is, follows automatically from the definition.\\
Given an object $\alpha : G \to F$ of $\textbf{Set}^{\mathcal C} / F$, we define a functor $\Phi(\alpha) : \mathcal F \to \textbf{Set}$ by:
$$\Phi(\alpha)(x) := \{ y \in G(A) \, | \, \alpha_A(y) = x \} $$
if $x \in F(A)$, with $\Phi(\alpha)(f)(y) = G(f)(y)$ whenever this makes sense. It is clear that a morphism $\gamma : \alpha \to \beta$ in $\textbf{Set}^{\mathcal C} / F$ induces a natural transformation $\Phi(\gamma) : \Phi(\alpha) \to \Phi(\beta)$, so that $\Phi$ becomes a functor $\textbf{Set}^{\mathcal C} / F \to \textbf{Set}^{\mathcal F}$.\\
Conversely, given a functor $H: \mathcal F \to \textbf{Set}$, let $\Psi(H)$ denote the functor $\mathcal C \to \textbf{Set}$ defined by:
$$\Psi(H)(A) := \coprod_{x \in F(A)} H(x)$$
with $\Psi(H)(f)$, for $f:A \to B$ in $\mathcal C$, defined as the union of the functions $H(f:x \to F(f)(x))$ over all $x \in F(A)$. $\Psi(H)$ comes equipped with an obvious natural transformation to $F$, and $\Psi$ itself is a functor $\textbf{Set}^{\mathcal F} \to \textbf{Set}^{\mathcal C} / F$. It is easy to check that $\Psi$ and $\Psi$ are actually quasi-inverses.
\end{proof}

\begin{thm}
For every small category $\mathcal C$, the functor category $\textbf{Set}^{\mathcal C}$ is locally cartesian closed.
\end{thm}

\begin{proof}
By the previous lemma every slice is equivalent to another category of presheaves, then it suffices to show that $\textbf{Set}^{\mathcal C}$ is cartesian closed. The idea of the proof is a common argument: if exponentials exists, their definition is forced by the Yoneda lemma. Indeed, if $G^F$ exists, then elements of $G^F(A)$ must correspond bijectively to morphisms $\mathit{Hom}_{\mathcal C}(A, -) \to G^F$, and hence to morphisms $\mathit{Hom}_{\mathcal C}(A, -) \times F \to G$. The functor category $\textbf{Set}^{\mathcal C }$ has small $\mathit{Hom}$-sets, so we take the set of all such morphisms as the definition of $G^F(A)$ on objects. Given a morphism $f:A \to B$ in $\mathcal C$, $G^F(f)$ is defined as the precomposition with $\mathit{Hom}_{\mathcal C}(f,-) \times 1 : \mathit{Hom}_{\mathcal C}(B, -) \times F \to \mathit{Hom}_{\mathcal C}(A,-) \times F$.\\
Next we define the evaluation map $ev: G^F \times F \to G$ by $ev_A(\phi, x) := \phi_A(1_A, x)$ where $\phi : \mathit{Hom}_{\mathcal C}(A,-) \times F \to G$ and $x \in F(A)$. Finally, given a morphism $\theta :H \times F \to G$ we define its exponential transpose as $\bar \theta : H \to G^F$ by 
$$(\bar \theta_A(z) )_B (f,x) := \theta_B (H(f)(z) , x)$$
for $z \in H(A)$, $f:A \to B$ and $x \in F(B)$.
\end{proof}

\printindex

\end{document}